\documentclass[oneside]{amsart}
\usepackage[english]{babel}
\usepackage[T1]{fontenc}
\usepackage[latin9]{inputenc}
\usepackage{units}
\usepackage{mathrsfs}
\usepackage{mathtools}
\usepackage{amsbsy}
\usepackage{amsthm}
\usepackage{amssymb}
\usepackage{setspace}
\usepackage{esint}
\usepackage[hidelinks]{hyperref}
\makeatletter
\numberwithin{equation}{section}
\numberwithin{figure}{section}

\usepackage{amsthm}
\newtheorem{theorem}{Theorem}[section]
\newtheorem*{theorem*}{Theorem}
\newtheorem{corollary}[theorem]{Corollary}
\newtheorem{lemma}[theorem]{Lemma}
\newtheorem{proposition}[theorem]{Proposition}

\theoremstyle{definition}
\newtheorem{definition}{Definition}

\theoremstyle{remark}
\newtheorem{remark}{Remark}[section]

\makeatother

\usepackage{babel}
\begin{document}
\newcommand{\SO}{\operatorname{SO}} 

\newcommand{\SL}{\operatorname{SL}} 

\newcommand{\PSL}{\operatorname{PSL}}

\newcommand{\GL}{\operatorname{GL}} 

\newcommand{\supp}{\operatorname{supp}} 

\newcommand{\diag}{\operatorname{diag}} 

\newcommand{\BC}{\operatorname{BC}} 

\newcommand{\Lie}{\operatorname{Lie}} 

\newcommand{\euc}{\operatorname{euc}} 

\newcommand{\thin}{\operatorname{thin}} 

\newcommand{\thick}{\operatorname{thick}} 

\newcommand{\cusp}{\operatorname{cusp}} 

\newcommand{\core}{\operatorname{core}} 

\newcommand{\rank}{\operatorname{rank}} 

\newcommand{\inj}{\operatorname{inj}} 

\newcommand{\hull}{\operatorname{hull}} 

\newcommand{\Isom}{\operatorname{Isom}} 

\newcommand{\htop}{\operatorname{top}} 

\newcommand{\BM}{\operatorname{BM}} 

\newcommand{\tr}{\operatorname{tr}} 

\newcommand{\Id}{\operatorname{Id}} 

\newcommand{\Per}{\operatorname{Per}} 
\newcommand\numberthis{\addtocounter{equation}{1}\tag{\theequation}}

\title[Excursions to the cusps]{Excursions to the cusps for geometrically finite hyperbolic orbifolds, and equidistribution of closed geodesics in regular covers}
\author{Ron Mor}
\begin{abstract}
We give a finitary criterion for the convergence of measures on non-elementary
geometrically finite hyperbolic orbifolds to the unique measure of
maximal entropy. We give an entropy criterion controlling escape of mass to the cusps of the orbifold. Using this criterion we prove new results on the distribution of collections of closed geodesics on such orbifold, and as a corollary we prove equidistribution of closed geodesics up to a certain length in amenable regular covers of geometrically finite orbifolds.
\end{abstract}
\thanks{This work was supported by ERC 2020 grant HomDyn (grant no.~833423)}
\address{The Einstein Institute of Mathematics\\
	Edmond J. Safra Campus, Givat Ram, The Hebrew University of Jerusalem
	Jerusalem, 91904, Israel}
\maketitle

\section{Introduction}

In this paper we consider the geodesic flow on
the unit tangent bundle of non-elementary geometrically finite hyperbolic
orbifolds, as well as the frame flow on the frame bundle of such orbifolds. Specifically, we study conditions guaranteeing that a given sequence of measures which are invariant under the frame flow on such orbifolds converges to the unique
invariant measure of maximal entropy.

In \cite{einsiedler2011distribution} a criterion was introduced for
a specific hyperbolic surface, namely the modular surface. This criterion was of interest in part since it was used in the same paper to give a new proof of a special case of a theorem of Duke \cite{duke1988}, along the lines of partial results towards this theorem by Linnik and Skubenko.

In this paper we extend the ergodic theoretic results of \cite{einsiedler2011distribution}, namely a finitary form of the uniqueness of measure of maximal entropy for the geodesic flow on the modular surface, to the frame flow on non-elementary
geometrically finite hyperbolic orbifolds of any dimension. This flow is isomorphic to the action of a 1-parameter diagonal group $A=a_\bullet$ on
$\Gamma\backslash G$ for $G=\SO^{+}(1,\mathtt{d})$ and $\Gamma<G$ a
non-elementary geometrically finite subgroup.

Let $\delta=\delta(\Gamma)$ denote the critical exponent of $\Gamma$, and let $X=\Gamma\backslash\mathbb{H}^{\mathtt{d}}$. In Theorem~\ref{theorem1}, $B_{r}^{H}$ denotes a ball in the group $H$ of radius $r$ centered at the identity, $M$ the centralizer of $A$ in a maximal compact subgroup $K<G$, \ $N^+$ the unstable horospherical subgroup, and $N^{-}$ the stable horospherical subgroup (see~\S\ref{flow section}).
In these notations, for $y\in \Gamma\backslash G$ and $x\in yB_{1}^{N^{+}}B_{1}^{MA}B_{\lambda}^{N^{-}}$, where $\lambda>0$ is small, the backward trajectories of $x$ and $y$ under the frame flow stay at distance $\ll 1$ from each other for at least $\gg\log \frac{1}{\lambda}$ time, i.e.\ $d(xa_{-t},ya_{-t})\ll 1$ for all $0\leq t \ll \log\frac{1}{\lambda}$.

\begin{theorem}\label{theorem1}Let $(\mu_{i})_{i\in\mathbb{N}}$
be a sequence of $A$-invariant probability measures on $\Gamma\backslash G$
for a non-elementary geometrically finite subgroup $\Gamma<G$. Suppose there is a sequence $\lambda_{i}\to0^{+}$ and a constant $\alpha>0$ such that for all
sufficiently small $\epsilon_{0}>0$ the ``heights'' $H_{i}=\lambda_{i}^{\epsilon_{0}}$ satisfy:
\begin{enumerate}
\item \label{assumption_1_thm_1}$\mu_{i}(\mathcal{F}\cusp_{H_{i}}(X))\to0$ as $i\to\infty$
\item\label{assumption_2_thm_1}
\begin{align*}
\mu_{i}\times\mu_{i}(\{&(x,y)\in \mathcal{F}\core_{H_{i}}(X)\times \mathcal{F}\core_{H_{i}}(X):\:x\in yB_{1}^{N^{+}}B_{1}^{MA}B_{\lambda_{i}}^{N^{-}}\})\\&\ll_{\epsilon_0} \lambda_{i}^{\delta-\alpha\epsilon_{0}}
\end{align*}
\end{enumerate}
Then 
\begin{enumerate}
    \item \label{thm1 item1} the sequence of measures $(\mu_{i})_{i\in\mathbb{N}}$ is \emph{tight}, i.e.\ for any $\epsilon>0$ there is a compact set $X_\epsilon \subset \Gamma\backslash G$ such that $\mu _ i (X _ \epsilon) > 1 - \epsilon$ holds for all $i$ large enough.
    
    \item \label{thm1 item2} any weak-$\star$ limit of a subsequence of $(\mu_{i})_{i\in\mathbb{N}}$ has entropy $\delta(\Gamma)$.
    
    \item \label{thm1 item3} Suppose $\Gamma$ is Zariski-dense in $G$. Then $\mu_i\to m_{\BM}^{\mathcal{F}}$ in the weak-$\star$ topology, for $m_{\BM}^{\mathcal{F}}$ the Bowen-Margulis measure on $\Gamma\backslash G$ (see \S\ref{PS BM Section}).
\end{enumerate} 
\end{theorem}

The same methods can be used to control the amount of mass an invariant measure gives to the cusps of such orbifold. The amount of mass can be quantified in relation to the entropy of the invariant measure, as well as the rank of the cusps of the hyperbolic orbifold. Here $T_a$ stands for the time-one-map of the flow $A=a_{\bullet}$.

\begin{theorem}\label{theorem2}
Let $\mu$ be an $A$-invariant probability
measure on $\Gamma\backslash G$, for a non-elementary geometrically
finite subgroup $\Gamma<G$.
Then
\begin{equation*}
h_{\mu}(T_{a})\leq\delta-\underset{i=1}{\overset{\mathtt{d}-1}{\sum}}\frac{2\delta-i}{2}\mu(\mathcal{F}\cusp_{\epsilon}^{i}(X))+\frac{2\log(|\log\epsilon|)}{|\log\epsilon|}
\end{equation*}
for all small enough $0<\epsilon<\epsilon_{\mathtt{d}}$.
\end{theorem}

\begin{remark}
By a result of Beardon \cite{beardon1968exponent} (cf.\ \cite[Corollary 2.2]{mcmullen1999hausdorff}), the critical exponent of a non-elementary discrete subgroup $\Gamma < G$ is greater than $\frac{r_{\mathrm{max}}}{2}$, where $r_{\mathrm{max}}$ is the maximal rank of a parabolic fixed point of $\Gamma$, hence
the term $\frac{2\delta-i}{2}$ is positive, for all $1\leq i\leq\mathtt{d}-1$
for which $\cusp_{\epsilon}^{i}(X)\not=\emptyset$. Therefore the
correction term in the RHS is negative, that is to say that the higher
the measure of the cusp the smaller is the upper bound on the entropy.
Cusps of higher rank cut down the entropy by a lesser amount.
\end{remark}

Theorem \ref{theorem2} gives, quantitatively, the relation between entropy and escape of mass. A natural question is what happens in the case all of the mass escapes in a weakly-$\star$ converging sequence. Formally, we define the \textbf{entropy in the cusp} as \[h_\infty(T_a)=\underset{\{\nu_{n}\rightharpoonup 0\}}{\sup}\underset{n\to\infty}{\limsup}\:h_{\nu_n}(T_a).\]

An immediate corollary of Theorem~\ref{theorem2} is an upper bound on $h_\infty(T_a)$. 

\begin{corollary}\label{entropy in the cusp cor}
Let $\Gamma<G$ be a non-elementary geometrically finite subgroup, for which the maximal rank of a cusp is $r_{\mathrm{max}}$. Then $h_\infty(T_a)\leq \frac{r_{\mathrm{max}}}{2}$.
\end{corollary}

This upper bound (but not Theorem~\ref{theorem2}) also follows from the results of Riquelme and Velozo \cite{riquelme2017escape} (the proof there is rather different from what we give here). Using an appropriate closing (or shadowing) lemma it is not hard to see that the bound in Corollary~\ref{entropy in the cusp cor} is sharp, i.e.\ $h_\infty(T_a)=\frac{r_{\mathrm{max}}}{2}$, e.g. along the lines of Kadyrov's proof in \cite{kadyrov2012positive}.
To avoid having to set up the relevant closing lemma we omit the proof, but note that the fact that the bound in Corollary~\ref{entropy in the cusp cor} is sharp is shown in \cite{riquelme2017escape}.

The above results have partial overlap with some results that were recently proved by other authors. In particular, similar results for the geodesic flow on the modular surface were proved by Einsiedler, Lindenstrauss, Michel and Venkatesh in \cite{einsiedler2011distribution}. Einsiedler, Kadyrov and Pohl generalized these results to diagonal actions on spaces $\Gamma\backslash G$ where $G$ is a connected semisimple real Lie group of rank 1 with finite center, and $\Gamma$ is a lattice \cite{einsiedler2015escape}. Finally, Iommi, Riquelme and Velozo (in two papers with different sets of coauthors) considered entropy in the cusp for geometrically finite Riemannian manifolds with pinched negative sectional curvature and uniformly bounded derivatives of the sectional curvature \cite{iommi2015entropy,riquelme2017escape}. This latter setting is substantially more general than ours, though in the constant curvature case Theorem~\ref{theorem2} gives more information. Perhaps more importantly, our methods differ from those of \cite{iommi2015entropy,riquelme2017escape}, and give finitary versions of the above qualitative results. This allows us to apply the entropy results on invariant measures obtained as weak-$\star$ limits of certain measures of interest, before going to the limit, in the spirit of the results of \cite{einsiedler2011distribution}. 

Several such applications are given below. In particular, Theorem~\ref{theorem1} implies that on any non-elementary geometrically finite orbifold, large enough sets of closed geodesics must equidistribute. To be more precise, let $\Per_{\Gamma}(T)$ be the set of all periodic orbits of the geodesic flow in $\Gamma\backslash G/M$ of length at most $T$; we will at times implicitly identify between this set and the set of closed geodesics in $\Gamma\backslash \mathbb H^{\mathtt{d}}$ with the same length restriction.
\begin{theorem}\label{criterion for closed geodesics}
Let $\Gamma<G$ be a non-elementary geometrically finite subgroup.
Let $\psi(T)\subset\Per_{\Gamma}(T)$ be some subset, and let $\mu_{T}$ be the natural $A$-invariant probability measure on $\psi(T)$ (see \S\ref{proof of criterion Section}). Assume that there are sequences $T_i\underset{i\to\infty}{\to}\infty$ and $\alpha_{i}\underset{i\to\infty}{\to}0$, such that $|\psi(T_{i})|>e^{(\delta-\alpha_i)T_{i}}$ for all $i$. Then the sequence $(\mu_{T_i})_{i\in\mathbb{N}}$ converges to $m_{\BM}$ in the weak-$\star$ topology.
\end{theorem}

We use Theorem \ref{criterion for closed geodesics} to draw some results regarding the equidistribution of closed geodesics.
We show that the number of periodic $a_\bullet$-orbits up to a certain length, on which the integral of some bounded continuous function differs noticeably from the integral over the whole orbifold, is exponentially smaller (by a difference of $h$ in the exponent) than the number of all periodic $a_\bullet$-orbits with the same length restriction. 

\begin{theorem}\label{theorem3}Let $\Gamma<G$ be a non-elementary
geometrically finite subgroup. Fix $f\in C_{b}(\Gamma\backslash G/M)$ and
$\epsilon>0$. Then there is a constant $h>0$, such that for all large
enough $T>0$
\begin{equation*}
\#\left\{l\in\Per_{\Gamma}(T):\:\left|\int_{l}fd\mu_{l}-\int_{\Gamma\backslash G/M}f\,dm_{\BM}\right|>\epsilon\right\}\leq e^{(\delta-h)T}
\end{equation*}
where $\mu_{l}$ is the natural probability measure on the periodic $a_\bullet$-orbit $l$.
\end{theorem}

We are also able to use Theorem \ref{criterion for closed geodesics} to extend some of our results to regular covers of geometrically finite orbifolds. These, of course, are not geometrically finite unless the covering group is finite in which case the claims are trivial.

Consider the following well-known equidistribution theorem regarding closed geodesics on geometrically finite orbifolds \cite{margulis2004some,roblin2003ergodicite,paulin2015equilibrium}:
\begin{theorem*}
Let $\Gamma<G$ be a non-elementary geometrically finite subgroup. Let $\mu_T$ be the natural $A$-invariant probability measure on $\Per_{\Gamma}(T)$. Then the net $\{\mu_T\}_{T>0}$ converges to $m_{\BM}$ in the weak-$\star$ topology, as $T\to\infty$.
\end{theorem*}
Using our entropy estimates we are able to extend this and prove an equidistribution result for closed geodesics on regular covers with amenable covering groups. In the proof we use the deep fact proved in  \cite{paulin2015equilibrium,roblin2005un,sharp2007critical,brooks1985the} that the critical exponent is not changed under the taking of a subgroup with an amenable quotient.
\begin{theorem}\label{amenable covers}
Let $\Gamma_0<G$ be a geometrically finite group, and let $\Gamma\triangleleft\Gamma_0$ be a non-elementary normal subgroup such that the covering group $\Gamma\backslash\Gamma_0$ is amenable.
Let $\phi(T)\subset\Per_{\Gamma_0}(T)$ be the set of periodic $a_\bullet$-orbits in $\Gamma_0\backslash G/M$ of length at most $T$, which remain periodic and of the same length in $\Gamma\backslash G/M$. Let $\nu_{T}$ be the natural $A$-invariant probability measure on $\phi(T)$.
Then the net $\{\nu_{T}\}_{T>0}$ converges to $m_{\BM}$ in the weak-$\star$ topology, as $T\to\infty$.
\end{theorem}

As a corollary, we prove equidistribution in the covering space as well. A similar result was given by Oh and Pan \cite{hee2019local}, using different methods, for $\mathbb{Z}^{n}$-covers of convex cocompact rank one locally symmetric spaces. We emphasize that the key feature in our result is that $\Gamma_0\backslash \mathbb{H}^{\mathtt{d}}$ may have cusps.
\begin{corollary}\label{equidistribution in the covering}
Let $\Gamma_0<G$ be a geometrically finite group, and let $\Gamma\triangleleft\Gamma_0$ be a non-elementary normal subgroup such that the covering group $\Gamma\backslash\Gamma_0$ is amenable.
Let $N_T$ be the number of $\Gamma\backslash\Gamma_0$-equivalence classes of $\Per_{\Gamma}(T)$.
Then for all $f\in C(\Gamma\backslash G/M)$ such that $$\sup_{x\in\Gamma\backslash G/M} \sum_{\tau\in\Gamma\backslash\Gamma_0} \left|f(\tau x)\right|<\infty$$ the following holds:
$$\lim_{T\to\infty}\frac{1}{N_T}\sum_{l\in\Per_{\Gamma}(T)}\int_{l} fd\mu_l=\int_{\Gamma_0\backslash G/M}\sum_{\tau\in\Gamma\backslash\Gamma_0} f(\tau\Gamma v)dm_{\BM}(\Gamma_{0}v)$$
\end{corollary}

\begin{remark}
The assumption on $f$ in Corollary \ref{equidistribution in the covering} is satisfied by any $f\in C_{c}(\Gamma\backslash G/M)$.
\end{remark}

Next, we prove the following result, indicating that not all of the mass escapes for the collection of closed geodesics in (not necessarily amenable) regular covers of geometrically finite orbifolds, if the critical exponent of the cover is large enough.
\begin{theorem}\label{not all of the mass escapes in covers}
Let $\Gamma_0<G$ be a geometrically finite group, and let $\Gamma\triangleleft\Gamma_0$ be a non-elementary normal subgroup. 
Assume $\delta(\Gamma)>\frac{r_{\max}(\Gamma_0)}{2}$. Let $f\in C(\Gamma\backslash G/M)$ satisfy $f\geq 0$ and $$\sum_{\tau\in\Gamma\backslash\Gamma_0}f(\tau x)>0$$ for all $x\in\Gamma\backslash G/M$.
Then
$$\liminf_{T\to\infty}\frac{1}{N_T}\sum_{l\in\Per_{\Gamma}(T)}\int_{l} fd\mu_l>0$$
\end{theorem}
\begin{remark}
The assumption on $f$ in Theorem \ref{not all of the mass escapes in covers} is satisfied by any strictly positive $f\in C(\Gamma\backslash G/M)$.
\end{remark}
\subsection{Acknowledgements} This paper is a part of the author's PhD studies in The Hebrew University of Jerusalem. I would like to thank my advisor, Prof.\ Elon Lindenstrauss, to whom I am grateful for his guidance and support throughout this work.

\section{Preliminaries}\label{sec:preliminaries}
In this section we recall some basic facts from the theory of geometrically
finite groups, Patterson-Sullivan measures, Bowen-Margulis measures, and entropy of the geodesic flow.
Good references for the material covered in this section is Nicholls' book \cite{nicholls1989ergodic} as well as Bowditch's paper \cite{bowditch1993geometrical}.
A broader review can also be found in \cite{My_MSc}.

\subsection{Hyperbolic geometry}

Fix a natural number $\mathtt{d}\geq2$. Let $\mathbb{H}^{\mathtt{d}}$ be the $\mathtt{d}$-dimensional hyperbolic space. We use the upper half space model and identify $\mathbb{H}^{\mathtt{d}}$ with $\{ x \in \mathbb{R}^{\mathtt{d}}:\:x_{\mathtt{d}}>0 \}$
equipped with the metric $ds ^2 =\frac{\| dx \| ^2}{x_{\mathtt{d}} ^2}$.
The conformal ball model $\mathbb{B}^{\mathtt{d}}$
or the hyperboloid model $\mathcal{H}^{\mathtt{d}}$ will also be used on occasion, especially for expositional reasons.

In this work we will be interested in spaces of the form $\Gamma\backslash\mathbb{H}^{\mathtt{d}}$,
for $\Gamma<\Isom\mathbb{H}^{\mathtt{d}}$ a discrete group of isometries
of $\mathbb{H}^{\mathtt{d}}$. 
If $\Gamma$ is torsion free, $\Gamma\backslash\mathbb{H}^{\mathtt{d}}$ is a manifold of constant negative curvature; in general though it is only an orbifold.

\begin{definition}
Let $\Gamma<G$ be a discrete subgroup. The \textbf{limit set} of $\Gamma$ is the set $\Lambda(\Gamma)$ of accumulation points in the compactification $\overline{\mathbb{H}^{\mathtt{d}}}$, of any $\Gamma$-orbit $\Gamma x$ for $x\in\mathbb{H}^{\mathtt{d}}$. We call $\Gamma$ \textbf{elementary} if it has a finite limit set.
\end{definition}

The action of $\Isom\mathbb{H}^{\mathtt{d}}$ extends to an action
on the boundary of the hyperbolic $\mathtt{d}$-space using Moebius transformations of $\mathbb{B}^{\mathtt{d}}$. The isometries of the hyperbolic $\mathtt{d}$-space can be classified
into three mutually disjoint types, depending on their fixed points
in $\overline{\mathbb{H}^{\mathtt{d}}}$, as follows.

$g\in\Isom\mathbb{H}^{\mathtt{d}}$ is called:
\begin{enumerate}
\item \textbf{parabolic} if it has precisely one fixed point, which lies
on $\partial\mathbb{H}^{\mathtt{d}}$.
\item \textbf{loxodromic} (in the $\mathtt{d}=2$ case, also \textbf{hyperbolic})
if it has precisely two fixed points, which lie on $\partial\mathbb{H}^{\mathtt{d}}$.
\item \textbf{elliptic} if it has a fixed point in $\mathbb{H}^{\mathtt{d}}$. 
\end{enumerate}

For a point $x\in\partial\mathbb{H}^{\mathtt{d}}$ and a subgroup $H<\Isom\mathbb{H}^{\mathtt{d}}$, let $H_{x}$ stand for the stabilizer of $x$ in $H$.
\label{page of bounded parabolic}
A parabolic fixed point $\xi\in\partial\mathbb{H}^{\mathtt{d}}$ of
a discrete subgroup $\Gamma<\Isom\mathbb{H}^{\mathtt{d}}$, i.e.\ a point which is fixed by a parabolic element of $\Gamma$, is called
\textbf{bounded} if $$\Gamma_{\xi}\backslash\big(\Lambda(\Gamma)\smallsetminus\{\xi\}\big)$$
is compact. In particular, if $\xi=\infty$ is a parabolic fixed point, then it is bounded if and only if $$\sup_{x\in\Lambda(\Gamma)\smallsetminus\{\xi\}}d_{\euc}(x,L_{0})<\infty$$ where $L_{0}$ is some (or every) minimal $\Gamma_{\xi}$-invariant affine
subspace of $\partial\mathbb{H}^{\mathtt{d}}\smallsetminus\{\xi\}$.
In the geometrically finite case, all
parabolic fixed points are bounded (\cite[Lemma 4.6]{bowditch1993geometrical}).

Recall that the \textbf{rank} of a parabolic fixed point $\xi$ is defined by the rank of a maximal free abelian finite-index subgroup of the stabilizer $\Gamma_{\xi}$. It is denoted by $\rank(\xi)$. A key description of the rank is given as follows \cite[Section 2]{bowditch1993geometrical} (see also \cite[Theorem 3.4]{apanasov1983discrete}).
By conjugating $\Gamma$, assume $\xi=\infty$. Then there is a free abelian normal subgroup $\Gamma_{\infty}^{\ast}\triangleleft \Gamma_{\infty}$ of finite index, whose rank we have denoted by $\rank(\infty)$, and a non-empty $\Gamma$-invariant affine subspace $L\subset\mathbb{R}^{\mathtt{d}-1}\subset\partial\mathbb{H}^{\mathtt{d}}$ on which $\Gamma_{\infty}^{\ast}$ acts co-compactly by translations. This description may be used to give an equivalent definition of the rank of a parabolic fixed point, as the dimension of such affine subspace \cite{nicholls1989ergodic}.

\subsection{The thin-thick decomposition}\label{thin_thick_section}

We proceed to describe the thin-thick decomposition. It involves maximal parabolic subgroups, which are precisely the stabilizers of parabolic fixed points, and maximal loxodromic subgroups, which are precisely the stabilizers of loxodromic axes \cite{bowditch1993geometrical}.

Given $x\in\mathbb{H}^{\mathtt{d}}$, $\epsilon>0$ and any subgroup $H\leq\Gamma$, let
\begin{equation*}
H_{\epsilon}(x)=\langle\{\gamma\in H:\:d(x,\gamma x)\leq\epsilon\}\rangle
\end{equation*}
and
\begin{equation*}
T_{\epsilon}(H)=\{x\in\mathbb{H}^{\mathtt{d}}:\:|H_{\epsilon}(x)|=\infty\}.
\end{equation*}

In particular, one can take $H=\Gamma$ and then the set $T_{\epsilon}(\Gamma)$ contains $T_{\epsilon}(H)$ for any $H<\Gamma$. 
The \textbf{$\boldsymbol{\epsilon}$}-\textbf{thin
}part of $X=\Gamma\backslash\mathbb{H}^{\mathtt{d}}$ is defined to be $\thin_{\epsilon}(X)=\Gamma\backslash T_{\epsilon}(\Gamma$).
The following theorem (which we quote without proof) implies that the $\epsilon$-thin part of $X$ is, topologically, a disjoint union of its connected components,
each of the form $H\backslash T_{\epsilon}(H)$ for $H<\Gamma$ maximal
parabolic or maximal loxodromic. By that we mean that $H\backslash T_{\epsilon}(H)$
is embedded in $\thin_{\epsilon}(X)$. $H\backslash T_{\epsilon}(H)$
is called a \textbf{Margulis cusp} if $H$ is parabolic, and a \textbf{Margulis
tube} if $H$ is loxodromic.

\begin{theorem}[{\cite[Proposition 3.3.3]{bowditch1993geometrical}}]\label{thin_part}For
a discrete subgroup $\Gamma<\Isom\mathbb{H}^{\mathtt{d}}$ and for
$0<\epsilon<\epsilon_{\mathtt{d}}$ (where $\epsilon_{\mathtt{d}}$ is called the \textbf{Margulis constant}),
the set $\:T_{\epsilon}(\Gamma)$ is a disjoint union of the sets
$T_{\epsilon}(H)$, as $H$ ranges over all maximal parabolic and
maximal loxodromic subgroups of $\Gamma$. Moreover, if $H_{1},H_{2}$
are two distinct such subgroups, then
\begin{equation*}
d(T_{\epsilon}(H_{1}),T_{\epsilon}(H_{2}))\geq\frac{\epsilon_{\mathtt{d}}-\epsilon}{2}.
\end{equation*}
\end{theorem}

We recall the following definition.
\begin{definition}Let $G$ be a group acting on a space $X$, and
let $H<G$ be a subgroup. A subset $E\subset X$ is called \textbf{precisely
$\boldsymbol{H}$-invariant} if $h(E)=E$ for all $h\in H$ and $g(E)\cap E=\emptyset$
for all $g\in G\smallsetminus H$.\end{definition}
\label{page of precisely invariant}
It can be shown that $T_{\epsilon}(H)$ is precisely $H$-invariant,
for maximal parabolic and maximal loxodromic subgroups $H<\Gamma$. It follows that
$H\backslash T_{\epsilon}(H)$ is embedded in $\Gamma\backslash\mathbb{H}^{\mathtt{d}}$,
and clarifies the statement that cusps and tubes are embedded in the
thin part of $X$.

It is worth mentioning that in case $\Gamma$ is torsion-free, that
is $\Gamma\backslash\mathbb{H}^{\mathtt{d}}$ is a manifold, the definition
of the $\epsilon$-thin part agrees with the more common definition
\begin{equation*}
\thin_{\epsilon}(X)=\{x\in\Gamma\backslash\mathbb{H}^{\mathtt{d}}:\:\inj(x)\leq\frac{\epsilon}{2}\},
\end{equation*}
where $\inj(x)$ is the injectivity radius at $x$.\\

Let us give notations for some other useful subsets of the orbifold
$X=\Gamma\backslash\mathbb{H}^{\mathtt{d}}$. For a set $A\subset\overline{\mathbb{H}^{\mathtt{d}}}$, let $\hull(A)$ stand for its hyperbolic convex hull in $\overline{\mathbb{H}^{\mathtt{d}}}$.
Then for any $0<\epsilon<\epsilon_{\mathtt{d}}$, and for any parabolic fixed point $\xi\in\partial\mathbb{H}^{\mathtt{d}}$, we define:
\begin{align}
\cusp_{\epsilon}(X)&=\underset{H<\Gamma\:\mathrm{maximal\:parabolic}}{\bigcup}H\backslash T_{\epsilon}(H)\\
\cusp_{\epsilon}^{i}(X)&=\underset{H<\Gamma\:\mathrm{maximal\:parabolic\:of\:rank\:}i}{\bigcup}H\backslash T_{\epsilon}(H)\\
\cusp_{\epsilon}(\xi)&=\Gamma_{\xi}\backslash T_{\epsilon}(\Gamma_{\xi})\\
\core(X)&=\Gamma\backslash\big(\mathbb{H}^{\mathtt{d}}\cap\hull(\Lambda(\Gamma))\big)\\
\core_{\epsilon}(X)&=\core(X)\cap\overline{X\smallsetminus\cusp_{\epsilon}(X)}
\end{align}

\begin{definition}\label{geom-finite-def}
A discrete subgroup $\Gamma<\Isom\mathbb{H}^{\mathtt{d}}$ is \textbf{geometrically finite} if $\core_{\epsilon}(X)$
is compact for some (equivalently, for all) $0<\epsilon<\epsilon_{\mathtt{d}}$.\end{definition}

\begin{remark}In this case, $\cusp_{\epsilon}(X)$ is a union of
finitely many neighborhoods $\Gamma_{\xi}\backslash T_{\epsilon}(\Gamma_{\xi})$.
Definition~\ref{geom-finite-def} is one of several
definitions, most of them turn out to be equivalent; cf.\ \cite{bowditch1993geometrical}.
\end{remark}

\subsection{The geodesic flow and the frame flow}\label{flow section}
We will occasionally prefer to study dynamics on the group of isometries rather than on the tangent or frame bundles of $X=\Gamma\backslash\mathbb{H}^{\mathtt{d}}$ \cite{quint2006overview}. First, recall that the group of orientation-preserving isometries
of $\mathbb{H}^{\mathtt{d}}$ is isomorphic to $\SO^{+}(1,\mathtt{d})=O^{+}(1,\mathtt{d})\cap\SL(\mathtt{d}+1,\mathbb{R})$,
where $O^{+}(1,\mathtt{d})$ is the connected component of the indefinite
orthogonal group $O(1,\mathtt{d})$. $\SO^{+}(1,\mathtt{d})$ is a Lie subgroup of $\SL(\mathtt{d}+1,\mathbb{R})$,
whose Lie algebra is donated by $\mathfrak{so}(1,\mathtt{d})$.

We give notations for the following Lie subalgebras of $\mathfrak{sl}(\mathtt{d}+1,\mathbb{R})$:
\begin{equation*}
\mathfrak{a}=\left\{\begin{pmatrix}0 & 0 & t\\0 & 0_{(\mathtt{d}-1)\times(\mathtt{d}-1)} & 0\\t & 0 & 0\end{pmatrix}:\:t\in\mathbb{R}\right\},\:\mathfrak{k}=\left\{\begin{pmatrix}0 & 0\\0 & B\end{pmatrix}:\:B\in\mathfrak{so}(\mathtt{d})\right\}
\end{equation*}

\begin{equation*}
\mathfrak{n}^{-}=\left\{\begin{pmatrix}0 & u^{T} & 0\\u & 0 & -u\\0 & u^{T} & 0\end{pmatrix}:\:u\in\mathbb{R}^{\mathtt{d}-1}\right\},\:\mathfrak{n}^{+}=\left\{\begin{pmatrix}0 & u^{T} & 0\\u & 0 & u\\0 & -u^{T} & 0\end{pmatrix}:\:u\in\mathbb{R}^{\mathtt{d}-1}\right\}
\end{equation*}

\begin{equation*}
\mathfrak{m}=\left\{\begin{pmatrix}0 & 0 & 0\\0 & B & 0\\0 & 0 & 0\end{pmatrix}:\:B\in\mathfrak{so}(\mathtt{d}-1)\right\}
\end{equation*}
where $\mathfrak{so}(\mathtt{d})=\{B\in M_{\mathtt{d}}(\mathbb{R}):\:B^{T}=-B\}$.
Denote the corresponding Lie subgroups by $A=a_\bullet,K,N^{-},N^{+},M$ respectively. Recall that $\mathfrak{so}(1,\mathtt{d})=\mathfrak{m}\oplus\mathfrak{a}\oplus\mathfrak{n}^{+}\oplus\mathfrak{n}^{-}$.

Let $(g_t)_{t\in\mathbb{R}}$ stand for the geodesic flow on the unit tangent bundle $T^{1}\mathbb{H}^{\mathtt{d}}$ of $\mathbb{H}^{\mathtt{d}}$. 
The structure of $\mathfrak{so}(1,\mathtt{d})$ allows to study
the geodesic flow as follow. In the hyperboloid model, $K$ is the stabilizer in $G$ of $e_{0}=(1,0,\ldots,0)^{T}$,
and $e_{0}$ is the unique fixed point of $K$ in $\mathcal{H}^{\mathtt{d}}$. The subgroup $M$ is the stabilizer in $K$ of the unit vector tangent
at $e_{0}$ to the geodesic $(a_{t}e_{0})_{t\in\mathbb{R}}$, where $a_{t}=\exp\begin{pmatrix}0 & 0 & t\\
0 & 0_{(\mathtt{d}-1)\times(\mathtt{d}-1)} & 0\\
t & 0 & 0
\end{pmatrix}\in A$. The subgroup $K$ acts transitively on $T_{e_{0}}^{1}\mathcal{H}^{\mathtt{d}}$,
and so we identify $T^{1}\mathcal{H}^{\mathtt{d}}\cong G/M$. Moreover, the
geodesic flow reads as the action of $A$ by right translation on
$G/M$, namely the geodesic flow $g_{t}$ satisfies $g_{t}(yM)=ya_{t}M$
for all $y\in G$.
Likewise, the oriented orthonormal frame bundle (shortly, ``frame bundle'') $\mathcal{F}\mathbb{H}^{\mathtt{d}}$ of $\mathbb{H}^{\mathtt{d}}$ may be realized as a bundle over $T^{1}\mathbb{H}^{\mathtt{d}}$ with fibers isomorphic to $M$, and so it may be identified with $G$. The frame flow $(g_t)_{t\in\mathbb{R}}$ is then defined on $G$ the same way as the geodesic flow is on $G/M$, by right-multiplication by $a_t$. The action of $A$ on a frame translates it along the geodesic defined by the frame's first vector, while the other orthogonal vectors are determined by parallel transport. 

We endow $G/M\cong T^{1}\mathbb{H}^{\mathtt{d}}$ and $G\cong\mathcal{F}\mathbb{H}^{\mathtt{d}}$ with $\Isom^{+}\mathbb{H}^{\mathtt{d}}$-invariant metrics by
\begin{equation*}
d_{T^{1}\mathbb{H}^{\mathtt{d}}}(g_1 M,g_2 M)=\underset{t\in[-1,1]}{\sup}d_{\mathbb{H}^{\mathtt{d}}}(\pi_{K}(g_{1}a_t M),\pi_{K}(g_{2}a_t M))
\end{equation*}
and 
\begin{equation*}
d_{\mathcal{F}\mathbb{H}^{\mathtt{d}}}(g_1,g_2)=\sum_{i=1}^{\mathtt{d}}d_{T^{1}\mathbb{H}^{\mathtt{d}}}(\pi_{i}(g_1),\pi_{i}(g_2))
\end{equation*}
where $\pi_{K}:T^{1}\mathbb{H}^{\mathtt{d}}\to \mathbb{H}^{\mathtt{d}}$ is the base point projection, and $\pi_i:\mathcal{F}\mathbb{H}^{\mathtt{d}}\to T^{1}\mathbb{H}^{\mathtt{d}}$ is the projection of a frame to its $i$'th vector.

As invariant metrics, $d_{\mathbb{H}^{\mathtt{d}}}$, $d_{T^{1}\mathbb{H}^{\mathtt{d}}}$ and $d_{\mathcal{F}\mathbb{H}^{\mathtt{d}}}$ naturally descend to metrics on $X=\Gamma\backslash\mathbb{H}^{\mathtt{d}}$, $T^{1}X=\Gamma\backslash T^{1}\mathbb{H}^{\mathtt{d}}$ and $\mathcal{F}X=\Gamma\backslash\mathcal{F}\mathbb{H}^{\mathtt{d}}$ denoted by $d_{X}$, $d_{T^{1}X}$ and $d_{\mathcal{F}X}$ respectively, for any discrete subgroup $\Gamma<G$.
The geodesic and frame flows descend to $\Gamma\backslash G/M$ and $\Gamma\backslash G$ respectively. These flows are denoted by $(g_t)_{t\in\mathbb{R}}$ as well.\\

We end up this subsection by giving notations which will be very useful
through this paper. Let $\Omega$ be the non-wandering set of the
geodesic flow on $T^{1}(\Gamma\backslash\mathbb{H}^{\mathtt{d}})$. It can be shown \cite{eberlein1972geodesic} that $\Omega$
is precisely the vectors in $T^{1}(\Gamma\backslash\mathbb{H}^{\mathtt{d}})$
that lift to vectors in $T^{1}\mathbb{H}^{\mathtt{d}}$
which define geodesics whose both end points in $\partial\mathbb{H}^{\mathtt{d}}$ belong to the limit set $\Lambda(\Gamma)$. Therefore, $\Omega\subset T^{1}\core(X)$. Similarly, let $\Omega_{\mathcal{F}}$ stand for the non-wandering set of the frame flow, which is just the set of frames whose first vectors are in $\Omega$.

We give the following notations ($\mathrm{nc}$ stands for non-cusp,
$\mathrm{c}$ for cusp):
\begin{align}
    \Omega_{\mathrm{nc}}^{\epsilon}&=\Omega_{\mathcal{F}}\smallsetminus \mathcal{F}\cusp_{\epsilon}(X)\\
    \Omega_{\mathrm{c}}^{\epsilon}&=\Omega_{\mathcal{F}}\cap \mathcal{F}\cusp_{\epsilon}(X)\\
\Omega_{\mathrm{c},i}^{\epsilon}&=\Omega_{\mathcal{F}}\cap \mathcal{F}\cusp_{\epsilon}^{i}(X)
\end{align}

\subsection{Patterson-Sullivan measures and the Bowen-Margulis measure}\label{PS BM Section}
Let $\Gamma<\Isom\mathbb{H}^{\mathtt{d}}$ be a discrete subgroup. In this section we briefly recall the construction of the Patterson-Sullivan measures on $\partial\mathbb{H}^{\mathtt{d}}$ \cite{patterson1976limit,sullivan1979density,sullivan1984entropy} and the Bowen-Margulis measure on $T^1(\Gamma\backslash\mathbb{H}^{\mathtt{d}})$. We refer to \cite{nicholls1989ergodic} for more details. We will not require this construction in the following sections, but rather just use the characterization of the Bowen-Margulis measure as the unique measure of maximal entropy (see \S\ref{Entropy Section}).\\

For every $x,y\in\mathbb{H}^{\mathtt{d}}$ and $0<s\in\mathbb{R}$
the \textbf{Poincare series} is defined by
\begin{equation*}
g_{s}(x,y)=\underset{\gamma\in\Gamma}{\sum}e^{-s\cdot d(x,\gamma y)},
\end{equation*}
and the \textbf{critical exponent} by
\begin{equation*}
\delta(\Gamma)=\inf\{s\in\mathbb{R}:\:g_{s}(x,y)<\infty\}
\end{equation*}
which is independent of $x,y\in\mathbb{H}^{\mathtt{d}}$. The critical exponent of a non-elementary discrete subgroup with a parabolic fixed point of rank $k$ satisfies $\delta(\Gamma)>\frac{k}{2}$ \cite{beardon1968exponent} (cf.\ \cite[Corollary 2.2]{mcmullen1999hausdorff}).

The subgroup $\Gamma$ is said to
be of \textbf{convergence} or \textbf{divergence} type if the Poincare
series converges or diverges, respectively, at $s=\delta(\Gamma)$. The construction of the Patterson-Sullivan measures can be done for subgroups of either type \cite{patterson1976limit}, but the details are simpler for the latter case. Fortunately, it can be shown that non-elementary geometrically finite groups are all of divergence type and so we will focus on this case.

Fix some reference point $y\in\mathbb{H}^{\mathtt{d}}$. For any $x\in\mathbb{H}^{\mathtt{d}}$
and $s>\delta(\Gamma)$, define
\begin{equation*}
\mu_{x,s}=\frac{1}{g_{s}(y,y)}\underset{\gamma\in\Gamma}{\sum}e^{-s\cdot d(x,\gamma y)}\delta_{\gamma y}
\end{equation*}
where $\delta_{\gamma y}$ stands for the Dirac measure at $\gamma y$.
We consider $\mu_{x,s}$ as a measure on $\overline{\mathbb{H}^{\mathtt{d}}}$. Any weak-$\star$ limit $\mu_{x}$ of a sequence $\mu_{x,s_n}$, where $s_n$ strictly decreases to $\delta(\Gamma)$, is called a \textbf{Patterson-Sullivan} measure with respect to $x$.
Such a limit measure always exists, but in general it doesn't have to be unique
and may as well depend on the reference point $y$. However, in the non-elementary
and geometrically finite case the family is unique and independent of $y$, up to a multiplicative constant.
A key property of the Patterson-Sullivan measure is that it is supported
on the limit set $\Lambda(\Gamma)$.

We use the conformal ball model $\mathbb{B}^{\mathtt{d}}$. Recall that $T^{1}\mathbb{B}^{\mathtt{d}}$ may be identified with $(\mathbb{S}^{\mathtt{d}-1}\times\mathbb{S}^{\mathtt{d}-1}\smallsetminus\diag)\times\mathbb{R}$,
where $\diag=\{(\xi,\xi):\:\xi\in\mathbb{S}^{\mathtt{d}-1}\}$, as follows; 
each vector $x\in T^{1}\mathbb{B}^{\mathtt{d}}$ defines a unique geodesic $(x_t)_{t\in\mathbb{R}}$ in $\mathbb{B}^{\mathtt{d}}$, which is parametrized by hyperbolic length in such a way that $x_0$ is the Euclidean midpoint of the geodesic.
Let $\eta_{\pm}=\lim_{t\to\pm\infty}x_t\in\partial\mathbb{B}^{\mathtt{d}}=\mathbb{S}^{\mathtt{d}-1}$ be the end points of the geodesic.
Then $x$ is identified with $(\eta_{-},\eta_{+},s)$ where $s$ is the unique real number such that $x_s\in\mathbb{B}^{\mathtt{d}}$ is the base point of the vector $x$.
Conversely, each triplet
defines a unique point in $T^{1}\mathbb{B}^{\mathtt{d}}$ in the same
way.

Using this parametrization, we define the \textbf{Bowen-Margulis} measure by
\begin{equation*}
dm_{\BM}(\eta_{-},\eta_{+},t)=\frac{d\mu_{0}(\eta_{-})d\mu_{0}(\eta_{+})d\lambda(t)}{\|\eta_{+}-\eta_{-}\|^{2\delta}}
\end{equation*}
where $\lambda$ is the Lebesgue measure on $\mathbb{R}$, $\mu_0$ is the Patterson-Sullivan measure with respect to $0\in\mathbb{B}^{\mathtt{d}}$ and $\|\cdot\|$ is the Euclidean norm on $\mathbb{S}^{\mathtt{d}-1}\subset\mathbb{R}^{\mathtt{d}}$.

The normalization by $\|\eta_{+}-\eta_{-}\|^{2\delta}$ makes sure
that $m_{\BM}$ is invariant under the action of orientation preserving
isometries, and so it descends to a measure (denoted
by $m_{\BM}$ as well) on $\Gamma\backslash T^{1}\mathbb{B}^{\mathtt{d}}$. 
In case $m_{\BM}(\Gamma\backslash T^{1}\mathbb{B}^{\mathtt{d}})<\infty$, we normalize $m_{\BM}$ to be a probability measure. It can be shown that this is indeed the case for non-elementary geometrically finite subgroups.
The Bowen-Margulis measure on $T^{1}\mathbb{H}^{\mathtt{d}}\cong G/M$ is naturally lifted to the frame bundle $\mathcal{F}\mathbb{H}^{\mathtt{d}}\cong G$ using the Haar measure on $M$. It is denoted $m_{\BM}^{\mathcal{F}}$.

\subsection{Entropy}\label{Entropy Section}
Let $(X,\mathscr{B},\mu,T)$ be a measure-preserving system. Recall the definition of the measure-theoretic entropy \cite{einsiedlerEntropyBook}.

\begin{definition}The \textbf{entropy} of a measurable partition
$\xi$ is
\begin{equation*}
H_{\mu}(\xi)=-\underset{A\in\xi}{\sum}\mu(A)\log\mu(A),
\end{equation*}
where we take the convention ``$0\log0=0$''. 
\end{definition}

\begin{definition}The \textbf{conditional entropy} of a partition
$\xi$, given a partition $\eta$, is
\begin{equation*}
H_{\mu}(\xi|\eta)=\underset{B\in\eta}{\sum}\mu(B)H_{\mu_{B}}(\xi),
\end{equation*}
where $\mu_{B}=\frac{1}{\mu(B)}\mu|_{B}$ is the restriction of $\mu$
to $B$, normalized to be a probability measure.
\end{definition}

\begin{definition}Let $\xi$ be a measurable partition with finite entropy. The \textbf{entropy}
of $T$ with respect to $\xi$ is
\begin{equation*}
h_{\mu}(T,\xi)=\underset{n\to\infty}{\lim}\frac{1}{n}H_{\mu}(\underset{i=0}{\overset{n-1}{\bigvee}}T^{-i}\xi)=\underset{n\in\mathbb{N}}{\inf}\frac{1}{n}H_{\mu}(\underset{i=0}{\overset{n-1}{\bigvee}}T^{-i}\xi)
\end{equation*}
where $\xi_{1}\vee\xi_{2}=\{A\cap B:\:A\in\xi_{1},\:B\in\xi_{2}\}$
is the common refinement of $\xi_{1}$ and $\xi_{2}$.
\end{definition}

\begin{remark}The fact that the former limit exists, and equals to
the infimum of the sequence $\frac{1}{n}a_{n}=\frac{1}{n}H_{\mu}(\underset{i=0}{\overset{n-1}{\bigvee}}T^{-i}\xi)$,
is due to $(a_{n})_{n=1}^{\infty}$ being sub-additive. \end{remark}

\begin{definition}The \textbf{entropy} of $T$ is $h_{\mu}(T)=\underset{\xi:\:H_{\mu}(\xi)<\infty}{\sup}h_{\mu}(T,\xi)$.\end{definition}

An analogous notion of entropy, in the context of topological and
metric spaces, is the topological entropy. For compact topological spaces it is defined by replacing the role of measurable partitions by open covers \cite{adler1965topological}. For non-compact metric spaces $(X,d)$, the definition extends as follows \cite{bowen1971entropy}.

\begin{definition}Let $T:\:X\to X$ be a uniformly continuous map.
Let $K\subset X$ be a compact subset, and take $\epsilon>0$ and
$n\in\mathbb{N}$. A set $E\subset K$ is called $\boldsymbol{(n,\epsilon,K,d,T)}$-\textbf{separated}
if for all $x,y\in E$ there is an integer $0\leq i<n$ such that
$d(T^{i}x,T^{i}y)>\epsilon$. We denote by $r_{d}(n,\epsilon,K,T)$
the maximal cardinality of a $(n,\epsilon,K,d,T)$-separated set.
\end{definition}

\begin{definition}$h_{d}(T,K)=\underset{\epsilon\to0^{+}}{\lim}\underset{n\to\infty}{\limsup}\frac{1}{n}\log(r_{d}(n,\epsilon,K,T))$.\end{definition}

\begin{definition}\label{top_entropy_1}$h_{\htop,d}(T)=\underset{K\subset X\:\mathrm{compact}}{\sup}h_{d}(T,K)$.\end{definition}

Often $h_{\htop,d}(T)$ is referred to as the topological entropy
of $T$. However, sometimes the following definition is given as well. 

\begin{definition}\label{top_entropy_2}$h_{\htop}(T)=\underset{\rho}{\inf}\:h_{\htop,\rho}(T)$,
where the infimum is taken over the set of all metrics $\rho$ equivalent
to $d$.\end{definition} 

Recall the variational principle in the non-compact context, which compares between the measure-theoretic entropy and the topological entropy.
\begin{theorem}[The variational principle]\label{var_non_compact} Let $X$ be a metric space, and $T:X\to X$ a homeomorphism.
Then
\begin{equation*}
\underset{\mu\in M_{T}}{\sup}h_{\mu}(T)\leq h_{\htop}(T)
\end{equation*}
for $M_T$ the set of all $T$-invariant Borel probability measures on $X$.
\end{theorem}

Unlike the compact case \cite{goodwyn1969topological,dinaburg1970relationship,goodman1971relating,misiurewicz1975short}, a strict inequality in Theorem \ref{var_non_compact} may be true
\cite{handel1995metrics}.
This principle gives rise to a natural notion of \textbf{measures
of maximal entropy}, that is measures $\mu\in M_{T}$ for which $h_{\mu}(T)=h_{\htop}(T)$, which may or may not exist (and if exist, may or may not be unique).

When restricting ourselves to the discussion of the time-one-map of
the geodesic flow $T_{a}(x)\coloneqq g_{1}(x)=xa$, over the unit tangent bundle of a hyperbolic orbifold $T^{1}(\Gamma\backslash\mathbb{H}^{\mathtt{d}})$,
much more can be said about both the topological entropy and measures
of maximal entropy. A good reference for that is \cite{otal2004principe} (see also \cite{paulin2015equilibrium}).

\begin{theorem}\label{maximal_entropy}Let
$\Gamma<\Isom\mathbb{H}^{\mathtt{d}}$ be a non-elementary discrete
subgroup. The following hold for $T_a$, the time-one-map of the geodesic flow on $T^{1}(\Gamma\backslash\mathbb{H}^{\mathtt{d}})$:
\begin{enumerate}
\item $h_{\htop}(T_a)=\delta(\Gamma$).
\item There is a measure of maximal entropy for the restriction of the geodesic
flow to its non-wandering set, if and only if $m_{\BM}(\Gamma\backslash T^{1}\mathbb{H}^{\mathtt{d}})<\infty$.
Furthermore, in that case, $m_{\BM}$ is the unique measure of maximal
entropy.
\item If $\Gamma$ is geometrically finite, then $h_{\htop,d_{T^{1}X}}(T_a)=\delta(\Gamma)$, for $d_{T^{1}X}$ the metric defined in \S\ref{flow section}.
\end{enumerate}
\end{theorem}

\subsection{Closed geodesics, and counting elements in discrete subgroups}
We quote two results related to counting in a discrete subgroup of isometries.

First, a useful proposition which can be found in \cite{nicholls1989ergodic} for example.
\begin{proposition}\label{number_of_elements_which_translate}Let $\Gamma<\Isom\mathbb{H}^{\mathtt{d}}$
be discrete. Then for all $o\in\mathbb{H}^{\mathtt{d}}$ there is
a constant $B$ such that $|N(r,o)|\leq Be^{r\delta}$
for all $r>0$, where
\begin{equation*}
N(r,o)=\{\gamma\in\Gamma:\:d(\gamma o,o)\leq r\}.
\end{equation*}

\end{proposition}

Next, the following theorem \cite{margulis2004some,roblin2003ergodicite,paulin2015equilibrium} justifies the mentioned claim that the bound on the number of ``bad'' periodic $a_\bullet$-orbits in Theorem \ref{theorem3} is exponentially smaller than the number of all periodic $a_\bullet$-orbits.
\begin{theorem}\label{number_of_periodic_orbits}Let $\Gamma<\Isom\mathbb{H}^{\mathtt{d}}$ be a non-elementary geometrically finite subgroup. Then the number of periodic $a_\bullet$-orbits
of lengths at most $T$ is asymptotically $\frac{e^{\delta T}}{\delta T}$,
that is $\underset{T\to\infty}{\lim}|\Per_{\Gamma}(T)|\cdot(\frac{e^{\delta T}}{\delta T})^{-1}=1$.\end{theorem}

We will revisit this topic in $\S\ref{Amenable covers Section}$ when we study closed geodesics in regular covers of geometrically finite orbifolds.

\subsection{Notations}
We end up this section with some common notations.
\begin{enumerate}
\item Let $s_1,\ldots,s_n$ be parameters, with $S_i$ the range of $s_i$ for all $1\leq i\leq n$. Let $f_{1},f_{2}$ be positive functions on some space $T$, which may also implicitly depend on the values of the parameters. We denote $f_{1}\ll_{s_1,\ldots,s_n} f_{2}$ if there is a positive function $g$ on $\prod_{i=1}^{n}S_{i}$ such that $f_1(t,s)\leq g(s)f_2(t,s)$ for all $t\in T,s\in\prod_{i=1}^{n}S_{i}$. If this relation holds with some constant function $g$, we simply write $f_{1}\ll f_{2}$. The subgroup $\Gamma$, as well as the dimension $\mathtt{d}$, are considered constant throughout this paper, hence will not be explicitly mentioned in the $\ll$ notation.
\item We denote by $\pi_{K}$ the projection from $G$ to $G/K$ sending a frame to its base point, and use the same notation $\pi_{K}$ for the projection from $G/M$ to $G/K$ sending a tangent vector to its base point. Likewise, we denote by $\pi_{\Gamma}$ the projections from $G$ and $G/M$ to $\Gamma\backslash G$ and $\Gamma\backslash G/M$ respectively.
\end{enumerate}

\section{Main Lemma and Proof of Theorem \ref{theorem1}}\label{section_proofs}

\subsection{A treatment of non-parabolic elements}
In a few points in the proofs, elliptic elements require specific treatment. The key tool to deal with that is Lemma~\ref{deal_with_elliptic}, which is an immediate corollary of Selberg's Lemma.

\begin{proposition}[Selberg's Lemma, \cite{selberg1962discontinuous}]\label{Selberg's Lemma}
Let $k$ be a field of characteristic $0$. Then any finitely generated subgroup of $\GL_{n}(k)$ contains a torsion-free subgroup of finite index.
\end{proposition}

\begin{lemma}\label{deal_with_elliptic}
Let $\Gamma<G$ be a discrete finitely generated subgroup. 
Then there is a constant $l_0$, depending only on $\Gamma$, such that if $d(\gamma g,g)<l_0$ for some $\gamma\in \Gamma$ elliptic and $g\in G$, then $\gamma=e$.
\end{lemma}

\begin{proof}
Due to Selberg's Lemma, $\Gamma$ has a finite index torsion-free subgroup $\Gamma_0<\Gamma$. It follows that the order of every elliptic element of $\Gamma$ is at most $[\Gamma:\Gamma_0]$.

We treat $\Isom^{+}{\mathcal{H}^{\mathtt{d}}}$ as the group of orientation-preserving Moebius transformations of $\mathbb{R}^{\mathtt{d}}$ which preserve the unit ball $\mathbb{B}^{\mathtt{d}}$. When done so, it can be shown \cite[Section 3]{gehring2017introduction} that every elliptic element $\gamma\in\Gamma$ is conjugated to an orthogonal map $T$. We use the canonical form of orthogonal matrices as the (orthogonal) conjugate of a matrix of the form
\begin{equation*}
\begin{pmatrix}R_{1} & & & & &\\
 & \ddots & & & &\\
 &  & R_{k} & & &\\
 &  &  & \pm1 & &\\
 &  &  &  & \ddots &\\
 &  &  &  &  & \pm1
\end{pmatrix}
\end{equation*}
where $R_1,\ldots,R_k$ are $2\times2$ rotation matrices. 

Since the order of $T$ (as an element of the group of orthogonal matrices) is bounded, it is clear from the canonical form that there is some constant (depending only on the bound of the order) $c_0>0$ such that if $T\not=\Id$ then $|\mathtt{d}-\tr T|>c_0$, where $\tr T$ is the trace of $T$. 
Since the trace function is continuous, is invariant under conjugation, and of course satisfies $\mathtt{d}=\tr\Id$, it follows that there is some constant $l_0>0$ such that $d(\gamma g,g)=d(g^{-1}\gamma g,e)>l_0$, unless $\gamma=e$. \end{proof}

\begin{remark}
By \cite[Proposition 3.1.6]{bowditch1993geometrical}, geometrically finite subgroups $\Gamma<G$ are finitely generated, and so Lemma \ref{deal_with_elliptic} holds for such groups as well.
\end{remark}

\begin{remark}\label{loxodromic_remark}
For geometrically finite groups $\Gamma<G$ there is a constant $l_1>0$ (which depends only on $\Gamma$) such that $d(x,gx)>l_1$ for all $x\in\mathbb{H}^{\mathtt{d}}$ and all loxodromic elements $g\in\Gamma$. This fact is related to the fact that in $\Gamma\backslash\mathbb{H}^{\mathtt{d}}$ there are only finitely many closed geodesics up to length $T$, for any $T>0$. So for such groups we will assume that the constant $l_0$ from Lemma \ref{deal_with_elliptic} satisfies $l_0\leq l_1$. By doing so, it follows that only parabolic elements can move points by a positive arbitrarily small distance.
\end{remark}

\subsection{The Main Lemma}\label{main lemma section}
\label{page_of_going_up_in_the_cusp}
For the rest of this section, assume $\Gamma<G$ is a non-elementary geometrically finite subgroup.

Let us first define the notion of a frame $\Gamma z\in \mathcal{F}\cusp_{\epsilon}(\xi)$ going up or down in the cusp, for some bounded parabolic fixed point $\xi$. Without loss of generality, assume $\xi=\infty$ and $\pi_{K}(z)\in T_{\epsilon}(\Gamma_{\infty})$.
It can be shown \cite[Section 3]{gehring2017introduction} that every $\gamma\in\Isom{\mathbb{H}^{\mathtt{d}}}$ fixing the point $\xi=\infty$ acts on $\mathbb{H}^{\mathtt{d}}$ by $\gamma x=\beta Ax+x^{0}$ for $0<\beta\in\mathbb{R}$, $x^{0}=(x_{1}^{0},\ldots,x_{\mathtt{d}-1}^{0},0)^{T}\in\mathbb{R}^{\mathtt{d}}$ and $A=\begin{pmatrix}B & 0\\
0 & 1
\end{pmatrix}$ for $B\in O(\mathtt{d}-1)$.
It follows that if the first vector of the frame $z$, i.e.\ the vector that determines the geodesic direction, points towards the point $\infty$ (upwards), i.e.\ $z$ is going in the upwards side of the geodesic semi-circle defined by it, then the same holds for $\gamma z$ for all $\gamma\in \Gamma_{\infty}$. In this case we say that $\Gamma z$ is \textbf{going up in the cusp}, and down in the cusp if otherwise.\\

Before getting into the main lemma of this paper, we need some technical lemmas, which describe the structure of the cusp and calculate the distance between
the cusp and the compact parts of $\Omega_{\mathcal{F}}$. In some sense, these follow from the fact that $G$ is a rank $1$ Lie group. We give sketches of the proofs, and leave the rest of the details to be filled by the reader.

\begin{lemma}\label{size_of_cusps} There are constants $c_{1},c_{2},\tilde{\epsilon}_{\mathtt{d}},\epsilon_{\mathtt{d}}^{\prime},t_0>0$ which depend only on $\Gamma$ and satisfy $\epsilon_{\mathtt{d}}^{\prime}<\tilde{\epsilon}_{\mathtt{d}}<\epsilon_{\mathtt{d}}$,
such that for all bounded parabolic fixed
points $\xi$ and for all $0<\epsilon\leq\epsilon_{\mathtt{d}}^{\prime}$: 
\begin{enumerate}
\item\label{size_of_cusps: item 1}
If $z\in\pi_{\Gamma}^{-1}(\Omega_{\mathcal{F}})\cap \mathcal{F}(T_{\epsilon}(\Gamma_{\xi}))$
then $za^{n}\in\mathcal{F}(T_{\tilde{\epsilon}_{\mathtt{d}}}(\Gamma_{\xi}))$
for all $n\in\mathbb{Z}$ with $|n|\leq\lceil|\log c_{1}\epsilon|\rceil$.
\item\label{size_of_cusps: item 2}
If $\Gamma z\in\Omega_{\mathcal{F}}\smallsetminus \mathcal{F}\cusp_{\epsilon}(\xi)$
then there is $n\in\mathbb{Z}$ with $|n|\leq\lceil|\log c_{2}\epsilon|\rceil$
such that $\Gamma za^{n}\not\in\mathcal{F}\cusp_{\tilde{\epsilon}_{\mathtt{d}}}(\xi)$.
\item\label{size_of_cusps: item 3} Assume $\Gamma z\in\Omega_{\mathcal{F}}\smallsetminus \mathcal{F}\cusp_{\epsilon}(\xi)$ is going down (resp.\ up) in the cusp and $\Gamma za^{-1}$ (resp.\
$\Gamma za$) is in $\mathcal{F}\cusp_{\epsilon}(\xi)$. Then
\begin{enumerate}
    \item $\Gamma za^{n}\not\in \mathcal{F}\cusp_{\epsilon_{\mathtt{d}}^{\prime}}(\xi)$, for $|n|=\lceil|\log c_{1}\epsilon|\rceil-1$ with $n>0$ (resp.\ $n<0$).
    \item $\Gamma za^{n}\not\in \mathcal{F}\cusp_{\epsilon}(\xi)$ for all $t_0\leq |n|\leq \lceil |\log c_1\epsilon|\rceil$ with $n>0$ (resp.\ $n<0$).
\end{enumerate}
\end{enumerate}
\end{lemma}

\begin{proof}
We use the upper-half space model, and assume by conjugation of $\Gamma$
that $\xi=\infty$. Recall that $\mathbb{R}^{\mathtt{d}-1}$ can be decomposed into a
product $\mathbb{R}^{\mathtt{d}-1}=\mathbb{R}^{k}\times\mathbb{R}^{\mathtt{d}-k-1}$
where $\rank(\infty)=k$ and $\mathbb{R}^{k}$ is a $\Gamma_{\infty}$-invariant
subspace such that $\Gamma_{\infty}\backslash\mathbb{R}^{k}$ is compact. Let $l_{\Gamma}>0$ be the minimal length of a translation in $\Gamma_{\infty}$. Combining the definition of the thin part of $\Gamma\backslash\mathbb{H}^{\mathtt{d}}$
with the analysis of the possible types of parabolic isometries
in $\Gamma_{\infty}$ as in \cite[Section 3.3]{apanasov2000conformal}, we conclude
that
\begin{align*}
\Gamma_{\infty}\backslash\left\{x\in\mathbb{H}^{\mathtt{d}}:\:x_{\mathtt{d}}\geq\frac{l_{\Gamma}\sqrt{1+4\|\frac{1}{l_{\Gamma}}(x_{k+1},\ldots,x_{\mathtt{d}-1})\|^{2}}}{2\sinh\frac{\epsilon}{2}}\right\}&\subset\cusp_{\epsilon}(\infty)\\
&\hspace{-1.5cm}\subset\Gamma_{\infty}\backslash\left\{x\in\mathbb{H}^{n}:\:x_{\mathtt{d}}\geq\frac{l_{\Gamma}}{2\sinh\frac{N_{\mathtt{d}}\epsilon}{2}}\right\}
\end{align*}
for all small enough $\epsilon>0$, and for some fixed $N_{\mathtt{d}}>0$.

Since $\infty$ is a bounded parabolic fixed point (see p.~\pageref{page of bounded parabolic}), there is a constant $c_{\Gamma}$ such that
\begin{align}\label{how_cusps_look_eq}
\begin{split}
\Gamma_{\infty}\backslash\left\{x\in \hull(\Lambda(\Gamma)):\: x_{\mathtt{d}}\geq\frac{c_{\Gamma}l_{\Gamma}}{2\sinh\frac{\epsilon}{2}}\right\}&\subset
\cusp_{\epsilon}(\infty)\cap\core(X)\\&\subset\Gamma_{\infty}\backslash\left\{x\in \hull(\Lambda(\Gamma)):\:x_{\mathtt{d}}\geq\frac{l_{\Gamma}}{2\sinh\frac{N_{\mathtt{d}}\epsilon}{2}}\right\}
\end{split}
\end{align}

Item \ref{size_of_cusps: item 1} follows directly from Equation \eqref{how_cusps_look_eq}, together with:
\begin{enumerate}
    \item The action of $\Gamma_{\infty}$ by $\gamma x=\beta Ax+x^0$ as described in the beginning of \S\ref{main lemma section}, where $\beta=1$ for parabolic and elliptic elements.
    \item The fact that discrete subgroups cannot contain both parabolic and loxodromic elements with a common fixed point \cite{bowditch1993geometrical}, so $\Gamma_{\infty}$ contains only parabolic and elliptic elements.
\end{enumerate}

For item \ref{size_of_cusps: item 2}, if $\Gamma z\not\in\mathcal{F}\cusp_{\tilde{\epsilon}_{\mathtt{d}}}(X)$ there is nothing to show, so we assume that $\pi_{K}(z)\in T_{\tilde{\epsilon}_{\mathtt{d}}}(\Gamma_{\infty})$.
Either in the future, if $\Gamma z$ goes down in the cusp, or in the past if $z$ goes up, the trajectory $\{\Gamma za^{n}\}_{n\in\mathbb{Z}}$ visits $\Omega_{\mathrm{nc}}^{\tilde{\epsilon}_{\mathtt{d}}}$, so there is an integer $n\in\mathbb{Z}$ with minimal absolute value such that $\Gamma za^{n}\in\Omega_{\mathrm{nc}}^{\tilde{\epsilon}_{\mathtt{d}}}$. This is of
course true if the non-wandering parts of the cusps of $X$ are distant enough from each other
(more than one unit), i.e.\ if $\tilde{\epsilon}_{\mathtt{d}}$ is small enough. Without loss of generality, $y=za^{n}$ and $z$ are on the same side of
the geodesic semi-circle defined by $z$, and $\pi_{K}(y)_{\mathtt{d}}<\pi_{K}(z)_{\mathtt{d}}$.

Consider the following path in $\mathbb{H}^{\mathtt{d}}$. First
we connect by a straight line the points $\pi_{K}(y)$ and $w$, where $w$ is the point right above $\pi_{K}(y)$ whose $\mathtt{d}$'th
coordinate is equal to $\pi_{K}(z)_{\mathtt{d}}$.
Then we draw a straight horizontal line (with the same $\mathtt{d}$'th
coordinate) from $w$ to a point which is $(\pi_{K}(z))_{\mathtt{d}}$ far in
the Euclidean direction $\overrightarrow{w\pi_{K}(z)}$. This path starts at $\pi_K(y)$ and passes through $\pi_K(z)$, and is of length $\log(\frac{\pi_{K}(z)_{\mathtt{d}}}{\pi_{K}(y)_{\mathtt{d}}})+1\geq n$. To conclude item \ref{size_of_cusps: item 2}, use Equation \eqref{how_cusps_look_eq} again.

Item \ref{size_of_cusps: item 3}, follows similarly.\end{proof}
\begin{remark}
For most of the paper, $\tilde{\epsilon}_{\mathtt{d}}$ will serve as a replacement for the Margulis constant $\epsilon_{\mathtt{d}}$, which is just slightly too large for some technical reasons as it is the marginal constant in Theorem \ref{thin_part}.
\end{remark}
\begin{remark}
Lemma \ref{size_of_cusps} shows in particular that the choice of $T_{\epsilon}(\Gamma_{\xi})$ as the cusp neighborhoods is compatible with the frame flow, in the sense that up to small ``fluctuations'' in time intervals of size at most $t_0$, the notion of ``going up in the cusp'' as defined in p.~\pageref{page_of_going_up_in_the_cusp} is the same as the notion of moving to regions $T_{\epsilon}(\Gamma_{\xi})$ with smaller $\epsilon$. This statement could be made precise, by considering the map $a_{t_0}$ instead of $a_1$.
\end{remark}
\begin{remark}\label{log_distance_remark}
It follows from Equation \eqref{how_cusps_look_eq}, by the arguments of Lemma \ref{size_of_cusps}, that the distance between the regions $\Omega_{\mathrm{c}}^{\epsilon_1}$ and $\Omega_{\mathrm{nc}}^{\epsilon_2}$, for $\epsilon_1\ll\epsilon_2$ small enough, is approximately $\log\frac{\epsilon_2}{\epsilon_1}$, up to an additive constant.
\end{remark}

We need another related lemma, regarding the volume of the thick part
of a hyperbolic orbifold. Here $B_{\eta}^{G}$ stands for the radius $\eta$ ball around the identity $e\in G$.
\begin{lemma}\label{number_of_balls}For all $\epsilon,\eta>0$ small enough, $\Omega_{\mathrm{nc}}^{\epsilon}$ can be covered by $N_{0}\ll_{\eta}|\log\epsilon|$
balls $\Gamma O_{i}$, for $O_{i}=k_{i}B_{\eta}^{G}$ and
$k_{i}\in G$.
Moreover, $\{O_i\}_{i=1}^{N_0}$ may be chosen so that:
\begin{enumerate}
\item Let $R$ be a set of representatives for the $\Gamma$-orbits of bounded parabolic fixed points of $\Gamma$. Then $O_i\cap \mathcal{F}T_{\tilde{\epsilon}_{\mathtt{d}}}(\Gamma_{\xi})=\emptyset$ for all bounded parabolic fixed points $\xi\not\in R$.
\item $d(o,g)<r_0$ for all $g\in O_{i}$ such that $\Gamma g\in\Omega_{\mathrm{nc}}^{\epsilon_{\mathtt{d}}^{\prime}}$,
for some reference point $o$ satisfying $\Gamma o\in\Omega_{\mathrm{nc}}^{\tilde{\epsilon}_{\mathtt{d}}}$ and some constant $r_{0}$ which depends only on $\Gamma$.
\end{enumerate}
\end{lemma}

\begin{proof}
We will give a covering of $\core_{\epsilon}(X)$ by $X$-balls $B(\Gamma z,\eta)$, rather than a covering of $\Omega_{\mathrm{nc}}^{\epsilon}$ (which is a subset of $\mathcal{F}\core_{\epsilon}(X)$).
To obtain a covering of $\Omega_{\mathrm{nc}}^{\epsilon}$ one only needs to choose $\ll_{\eta}1$ many frames $\{\Gamma g_i\}_{i\in I}$ with $\pi_{K}(\Gamma g_i)\in B(\Gamma z,\eta)$, such that $$\mathcal{F}B(\Gamma z,\eta)\subset\bigcup_{i\in I}B(\Gamma g_i,\eta).$$ This is easily done using compactness of $K$ and the definition of the metrics as in \S\ref{flow section}.
Since we only want to bound the number of balls
up to a multiplicative constant, this will suffice.

Recall, once again, the structure of the
cusps. For simplicity assume that $\xi=\infty$ is a bounded parabolic
fixed point in the upper half-space model. The boundary $\mathbb{R}^{\mathtt{d}-1}\subset\partial\mathbb{H}^{\mathtt{d}}$
decomposes into $\mathbb{R}^{\mathtt{d}-1}=\mathbb{R}^{k}\times\mathbb{R}^{\mathtt{d}-k-1}$
where $\rank(\infty)=k$ and $\mathbb{R}^{k}$ is a $\Gamma_{\infty}$-invariant
subspace such that $\Gamma_{\infty}\backslash\mathbb{R}^{k}$ is
compact.

As $\Gamma_{\infty}\backslash\mathbb{R}^{k}$ is compact, there is a subset $D_{\infty}\subset\mathbb{H}^{\mathtt{d}}$
which contains representatives of all elements in $X=\Gamma\backslash\mathbb{H}^{\mathtt{d}}$,
and is bounded in the directions defined by $\mathbb{R}^{k}$. As $\infty$ is a bounded parabolic fixed point (see p.~\pageref{page of bounded parabolic}), the distance $d_{\euc}(\zeta,\mathbb{R}^{k})$
is bounded for $\zeta\in\Lambda(\Gamma)\smallsetminus\{\infty\}$, and so
the part of $D_{\infty}$ corresponding to $\core(X)$
is bounded as well in the remaining $\mathtt{d}-k-1$ directions defined
by $\mathbb{R}^{\mathtt{d}-k-1}$. 

In other words, we have simply created a $\mathtt{d}$-dimensional box, bounded in $\mathtt{d}-1$ directions, which contains a fundamental domain for $\core(X)$. For $\epsilon,\epsilon^\prime$ small enough, which satisfy $\epsilon<\tilde{\epsilon}_{\mathtt{d}}<\epsilon^{\prime}<\epsilon_{\mathtt{d}}$, the part of the fundamental domain corresponding to $T_{\epsilon^\prime}(\Gamma_\infty)\smallsetminus T_{\epsilon}(\Gamma_\infty)$ is unbounded only in the upwards direction $e_{\mathtt{d}}$, where it is of hyperbolic length $\ll|\log\epsilon|$ due to Lemma \ref{size_of_cusps}.
We make note that since the $\mathtt{d}$'th coordinate of a point
in $T_{\epsilon^\prime}(\Gamma_{\infty})$ is bounded from below
by Equation \eqref{how_cusps_look_eq}, boundedness also implies that the
hyperbolic diameter in these directions is bounded from above (since
$ds^2=\frac{\|dx\|^2}{x_{\mathtt{d}}^2}$).
Therefore, we can cover
$$\hull(\Lambda)\cap (T_{\epsilon^\prime}(\Gamma_\infty)\smallsetminus T_{\epsilon}(\Gamma_\infty))$$ by $\ll_{\eta}|\log\epsilon|$ many $\eta$-balls. Each such ball, under the quotient by
$\Gamma$, is of the desired form.

As $\Gamma$ is geometrically finite, its bounded parabolic fixed
points consist of a finite number of $\Gamma$-orbits. So we may repeat this argument for each $\xi\in R$, where $R$ is a set of representatives for these orbits.
Moreover, we can choose a cover of the compact set $\core_{\epsilon^\prime}(X)$ by $\ll_{\eta} 1$ balls $\{\Gamma O_i\}_{i\in I}$, where we choose $\{O_i\}_{i\in I}$ to be contained in some compact subset of $\mathbb{H}^{\mathtt{d}}$.
This would yield a cover of $\core_{\epsilon}(X)$ of the desired size.

Both requirements on $k_i$ as in the statement of Lemma \ref{number_of_balls} are trivially satisfied, perhaps after making sure (in Lemma \ref{size_of_cusps}) that $\tilde{\epsilon}_{\mathtt{d}}$ is small enough with respect to $\epsilon^\prime$, and by choosing $\eta$ to be small enough so that these balls would not intersect $T_{\tilde{\epsilon}_{\mathtt{d}}}(\xi^\prime)$ for $\xi^{\prime}\not\in R$.\end{proof}
\begin{remark}\label{remark_about_eta}
Through this paper, we will fix $\eta\ll 1$ to be constant.
\end{remark}

We are headed towards the main lemma of this paper. First, we give some notations for the sets which the main lemma deals with, and prove a simple estimate.

For $0<\epsilon<\epsilon^\prime<\epsilon_{\mathtt{d}}$, $N\in\mathbb{N}$ and a function
\begin{equation*}
V:\{-N,\ldots,N\}\to\{0,1,\ldots,\mathtt{d}-1\}
\end{equation*}
define
\begin{align*}
Z(V,\epsilon,\epsilon^\prime)&=\Big\{x\in T_{a}^{N}(\Omega_{\mathrm{nc}}^{\epsilon^\prime})\cap T_{a}^{-N}(\Omega_{\mathrm{nc}}^{\epsilon^\prime}):\\
&\hspace{1cm}T_{a}^{m}x\in\Omega_{\mathrm{c},i}^{\epsilon}\iff V(m)=i,\:\forall m\in[-N,N]\cap\mathbb{Z},\:\forall i\in[1,\mathtt{d}-1]\cap\mathbb{Z}\Big\}.
\end{align*}
\begin{remark}
We denote $Z(V,\epsilon,\epsilon^\prime)=Z(V,\epsilon)$ if $\epsilon=\epsilon^\prime$.
\end{remark}

\begin{lemma}\label{number of Z}
For all $0<\epsilon<\epsilon^{\prime}$ small enough, there are at most $$|\log\epsilon|^{3}e^{\frac{3\log(|\log\epsilon|)}{|\log\epsilon|}N}$$
different functions $V$ for which $Z(V,\epsilon,\epsilon^{\prime})\not=\emptyset$.
\end{lemma}

\begin{proof} Let $V$ be a function with $Z(V,\epsilon,\epsilon^\prime)\not=\emptyset$, and define the interval $$J=\big[-\lceil|\log c_{1}\epsilon|\rceil,\lceil|\log c_{1}\epsilon|\rceil\big]\cap\mathbb{Z}.$$ Assume that the restriction\footnote{The main case of interest to us is $\lceil|\log c_{1}\epsilon|\rceil\leq N$. In the other case, this restriction may be thought of as the trivial restriction to $[-N,N]$. The following estimates might be much larger than required for the $\lceil|\log c_1\epsilon|\rceil>N$ case, yet it will not matter for the rest of the paper.} $V|_J$ is not identically zero.
Then, by Lemma \ref{size_of_cusps}, $V$ is positive and constant (up to $\ll 1$ ``fluctuations'' in value) on some sub-interval of $J$ in which the trajectory of any $x\in Z(v,\epsilon,\epsilon^\prime)$ visits $\Omega_{\mathrm{c}}^{\epsilon}$, and $V$ is zero outside of this sub-interval.
So the number of different restrictions $V|_J$ is bounded from above (up to a constant) by the number of different sub-intervals of $J$, i.e.\ by $c\lceil|\log c_1\epsilon|\rceil^2$.

By taking images and pre-images, there is the same number of possible
restrictions to any $I\subset[-N,N]$ of length $2\lceil|\log c_{1}\epsilon|\rceil$. As we can divide $[-N,N]$ into $\lceil\frac{2N+1}{2\lceil|\log c_{1}\epsilon|\rceil}\rceil$
sub-intervals of length up to $2\lceil|\log c_{1}\epsilon|\rceil$,
we obtain up to 
\begin{equation*}
(c\lceil|\log c_{1}\epsilon|\rceil^{2})^{\lceil\frac{2N+1}{2\lceil|\log c_{1}\epsilon|\rceil}\rceil}\leq|\log\epsilon|^{3(\frac{N}{|\log\epsilon|}+1)}=|\log\epsilon|^{3}\cdot e^{\frac{3\log(|\log\epsilon|)}{|\log\epsilon|}N}
\end{equation*}
different functions $V$, for small enough $\epsilon$.
\end{proof}

We define \textbf{Bowen balls} in the following manner. Let
\begin{equation*}
B_{N,\rho}=\underset{n=-N}{\overset{N}{\bigcap}}a^{-n}B_{\rho}^{G}(e)a^{n}.
\end{equation*}
Then a Bowen $(N,\rho)$-ball is a set of the form $xB_{N,\rho}$ for
some $x\in\Gamma\backslash G$.

Likewise, let
\begin{equation*}
    B_{N,\rho}^{+}=\underset{n=0}{\overset{N}{\bigcap}}a^{n}B_{\rho}^{G}(e)a^{-n},\:B_{N,\rho}^{-}=\underset{n=0}{\overset{N}{\bigcap}}a^{-n}B_{\rho}^{G}(e)a^{n}.
\end{equation*}
Then a forward Bowen $(N,\rho)$-ball is a set $xB_{N,\rho}^{+}$ for some $x\in\Gamma\backslash G$. A backward Bowen $(N,\rho)$-ball is a set $xB_{N,\rho}^{-}$.
\begin{remark}\label{bowen_ball_remark}
Bowen $(N,\rho)$-balls have width $\ll \rho e^{-N}$ in the $N^{+}$ and $N^{-}$ directions, and width $\ll \rho$ in the $MA$ directions. In particular, $xB_{N,\rho^{\prime}}$ may be covered by $\ll \lceil\frac{\rho^{\prime}}{\rho}\rceil^{\dim G}$ many Bowen $(N,\rho)$-balls, by considering small translations of $xB_{N,\rho}$ in all the directions of $G$. Note that $$xB_{N,\rho}\subset\Big\{y\in\Gamma\backslash G:\:d(T_{a}^n(x),T_{a}^{n}(y))<\rho,\:\forall n\in[-N,N]\cap\mathbb{Z}\Big\},$$
and equality holds if $\Gamma\backslash \mathbb{H}^{\mathtt{d}}$ is compact and $\rho$ is small enough.
\end{remark}

As will be shown in Lemma \ref{entropy_bound_by_Bowen_balls}, in order to
estimate the entropy of the frame flow, we should give an upper
bound to the number of Bowen $N$-balls needed to cover large subsets
of $\Gamma\backslash G$. The main step is to cover the set of all points in a given unit neighborhood in $\Gamma\backslash G$,
such that their trajectories enter another given unit neighborhood
after $N$ steps. As the map $T_{a}^{n}$ shrinks the $N^{-}$ part
of the unit ball in $G$ by $e^{n}$ in each direction, doesn't
change the size of the $MA$ part and enlarges the $N^{+}$ part by
$e^{n}$ in each direction, the trivial bound of the number of required forward Bowen $N$-balls in order to cover this set is $\ll e^{(\mathtt{d}-1)N}$, which would yield entropy $\mathtt{d}-1$.
As will be indicated in the proof of Lemma \ref{main step}, since
we will be restricting to the non-wandering set $\Omega_{\mathcal{F}}$, using the
geometry of $\Gamma$ we can show that the smaller amount $\ll e^{\delta(\Gamma)N}$
is a better bound. The key idea is that using the knowledge that a
trajectory enters some cusp, we can cut down the number of balls
even more, in a rate which correlates to the rank of the cusp. The smaller the rank, the more we can cut down the number of balls. This is the main lemma of this paper.\\

The main step of the main lemma is as follows.
For $\epsilon^\prime>0$, let $\{\Gamma O_i\}_{i=1}^{N_0(\epsilon^\prime)}$ be a covering of $\Omega_{\mathrm{nc}}^{\epsilon^\prime}$ as in Lemma \ref{number_of_balls}, where $\{\Gamma O_i\}_{i=1}^{N_1}$ cover $\Omega_{\mathrm{nc}}^{\epsilon_{\mathtt{d}}^{\prime}}$.
Note that $N_1$ is an absolute constant, independent of $\epsilon^{\prime}$.

For simplicity we shall initially consider the sets
\begin{align*}
Z_{+}(O_{i_{1}},O_{i_{2}},V,\epsilon)=\Big\{&x\in(\Gamma O_{i_{1}}\cap\Omega_{\mathcal{F}})\cap T_{a}^{-N}(\Gamma O_{i_{2}}\cap\Omega_{\mathcal{F}}):\:\\&\hspace{-5mm}T_{a}^{m}x\in \Omega_{\mathrm{c},i}^{\epsilon}\iff V(m)=i,\:\forall m\in[0,N]\cap\mathbb{Z},\:\forall i\in[1,\mathtt{d}-1]\cap\mathbb{Z}\Big\}
\end{align*}
for $1\leq i_{1},i_{2}\leq N_{0}(\epsilon^\prime)$, instead of $Z(V,\epsilon,\epsilon^{\prime})$.

Assume $Z_{+}(O_{i_{1}},O_{i_{2}},V,\epsilon)\not=\emptyset$
(otherwise the following lemma will be trivial). Consider the decomposition of $V^{-1}(\{1,\ldots,\mathtt{d}-1\})$
into a union of maximal mutually disjoint intervals $[i,j]\cap\mathbb{Z}$. If two such subsequent intervals $[i_1,j_1]\cap\mathbb{Z}$ and $[i_2,j_2]\cap\mathbb{Z}$ are less than $t_0$ units apart from each other (where $t_0$ is as in Lemma \ref{size_of_cusps}), we replace them with $[i_1,j_2]\cap\mathbb{Z}$, and continue to do so until all intervals are at least $t_0$ units apart from each other. Then, by Lemma \ref{size_of_cusps}, we emerge with intervals distanced at least
$2\lceil|\log c_{1}\epsilon|\rceil$ from each other. We may define
$I_{1},\ldots,I_{p}$ to be these intervals, extended by $\lceil|\log c_{1}\epsilon|\rceil-1$ 
in each direction, and then intersected with $\{0,\ldots,N\}$ (in
case we have exceeded this set, in either $I_{1}$ or $I_{p}$). These intervals $\{I_i\}_{i=1}^{p}$ are still mutually disjoint. In the times defined by some $I_{i}$,
the trajectory of any $x\in Z_{+}(O_{i_{1}},O_{i_{2}},V,\epsilon)$ 
meets only one cusp, whose rank is denoted by $r_{i}$. Let us divide $\{0,\ldots,N\}\smallsetminus\underset{i=1}{\overset{p}{\bigcup}}I_{i}$
into maximal mutually disjoint intervals $J_{1},\ldots,J_{l}$. 

We note that $\Gamma ga^{n}\in\Omega_{\mathrm{nc}}^{\epsilon_{\mathtt{d}}^{\prime}}$
for all $n$ which are endpoints of any of the intervals, with the
possible exception of $n=0$ or $n=N$, due to Lemma \ref{size_of_cusps}.
\begin{lemma}[The main step]\label{main step}
Assume $i_1,i_2\leq N_1$.
There is a constant $C$, depending only on $\Gamma$, such that the
following holds: For every $0\leq K\leq N$ such that $[0,K]\cap\mathbb{Z}=\underset{i=1}{\overset{s}{\bigcup}}I_{i}\cup\underset{j=1}{\overset{t}{\bigcup}}J_{j}$
for some $s,t$, the set $Z_{+}(O_{i_{1}},O_{i_{2}},V,\epsilon)$ can be covered by
\begin{equation*}
\leq C^{s+t}\cdot e^{\delta\underset{j=1}{\overset{t}{\sum}}|J_{j}|+\underset{i=1}{\overset{s}{\sum}}\frac{r_{i}}{2}|I_{i}|}
\end{equation*}
many sets of the form
\begin{equation*}
\Gamma(\gamma_1 O_l a^{-\min\{K+1,N\}}\cap O_{i_{1}})
\end{equation*}
for some $\gamma_1\in \Gamma$ and $1\leq l\leq N_1$.
\end{lemma}

\begin{proof}
First note that the ball $\{\Gamma O_{i_{1}}\}$ is clearly a covering
of $Z_{+}(O_{i_{1}},O_{i_{2}},V,\epsilon)$ of the correct form $\Gamma(\gamma_1 O_{l}a^{-0}\cap O_{i{}_{1}})$.
The size of this covering, $1$, is of course bounded from above by
\begin{equation*}
C^{s+t}\cdot e^{\delta\underset{j=1}{\overset{t}{\sum}}|J_{j}|+\underset{i=1}{\overset{s}{\sum}}\frac{r_{i}}{2}|I_{i}|}
\end{equation*}
for $t=s=0$. For reasons of readability, the constant $C>0$ will be defined only later in the proof, independently
of the following construction. It is important to note that this is
consistent and does not raise any circular argument.

We prove the lemma by going through the time intervals by their order, assuming by induction that the lemma is true for all previous intervals. For both the first interval (base case) and the following intervals (induction step), we start
with a covering of $Z_{+}(O_{i_{1}},O_{i_{2}},V,\epsilon)$ by up to
\begin{equation*}
C^{s+t}\cdot e^{\delta\underset{j=1}{\overset{t}{\sum}}|J_{j}|+\underset{i=1}{\overset{s}{\sum}}\frac{r_{i}}{2}|I_{i}|}
\end{equation*}
sets of the form $\Gamma(\gamma_1 O_l a^{-K}\cap O_{i_{1}})$,
where the interval we consider is $[K,K^\prime]\cap\mathbb{Z}$, either of type
$I_{s+1}$ or type $J_{t+1}$. We write $$K^{\prime}=\begin{cases}K+L & K^{\prime}<N\\ K+L+1 & K^{\prime}=N\end{cases}$$ in order to simplify notations in the proof.

The main idea is the following.
Let $x\in Z_{+}(O_{i_1},O_{i_2},V,\epsilon)$. Then $x=\Gamma g$ for some $g\in \gamma_1 O_l a^{-K}\cap O_{i_{1}}$. By the construction of the intervals, $\Gamma ga^{K+L+1}\in\Omega_{\mathrm{nc}}^{\epsilon_{\mathtt{d}}^{\prime}}$ (even in the $K+L+1=N$ case, i.e.\ the last interval, by the assumption $i_2\leq N_1$). Therefore, $ga^{K+L+1}=\gamma_{2}\tilde{g}_{2}$
for some $\gamma_{2}\in\Gamma$ and $\tilde{g}_{2}\in O_{j}$, for $1\leq j\leq N_1$.
Then
\begin{equation*}
g\in\gamma_{2}O_{j}a^{-(K+L+1)}\cap O_{i_{1}},
\end{equation*}
which under the quotient by $\Gamma$, is indeed of
the desired form. We need to bound the number of possible sets of this form.
This can be done by counting the number of possible
2-tuples $(O_{j},\gamma_{2})$, or alternatively the number of 3-tuples $(O_{j},\gamma_{1},\gamma_{1}^{-1}\gamma_{2})$.

The number of $O_j$'s is at most $N_1$. The number of $\gamma_1$'s is at most the number of sets given in the previous step, which is bounded by
\begin{equation*}
C^{s+t}\cdot e^{\delta\cdot\underset{j=1}{\overset{t}{\sum}}|J_{j}|+\underset{i=1}{\overset{s}{\sum}}\frac{r_{i}}{2}|I_{i}|}.
\end{equation*}

It only remains to count the elements of type $\gamma_{1}^{-1}\gamma_{2}$. 
Let $\tilde{g}_1=\gamma_1^{-1}ga^{K}\in O_l$.
By construction, $\Gamma \tilde{g}_{1},\Gamma \tilde{g}_{2}\in\Omega_{\mathrm{nc}}^{\epsilon_{\mathtt{d}}^{\prime}}$, and so by Lemma \ref{number_of_balls} we get $d(o,\tilde{g}_{1}),d(o,\tilde{g}_{2})<r_0$. Moreover, since $$\gamma_{1}^{-1}\gamma_{2}\tilde{g}_{2}=\tilde{g}_{1}a^{L+1}$$
we have
\[d(\gamma_{1}^{-1}\gamma_{2}\pi_{K}(o),\pi_{K}(o))\leq2r_0+L+1.\]
Therefore, these maps $\gamma_{1}^{-1}\gamma_{2}$ are contained in
\begin{equation*}
\Gamma_{J}=\{\gamma\in\Gamma:\:d(\gamma\pi_{K}(o),\pi_{K}(o))\leq2r_0+L+1\}.
\end{equation*}

In case the next interval is of type $J_{t+1}$ we cannot give a better
restriction of what maps of $\Gamma$ can be of the former type, and
we will have to bound $\Gamma_{J}$ itself using Proposition \ref{number_of_elements_which_translate},
obtaining
\begin{equation*}
|\Gamma_J|\leq Be^{(2r_0+L+1)\delta}\leq\tilde{B}e^{\delta|J_{t+1}|},
\end{equation*}
where $\tilde{B}=Be^{2r_0\delta}$ depends only on $\Gamma$ and $o$.\\

In case the next interval is of type $I_{t+1}$, we shall be
using the information about the trajectory spending time in the cusp
in order to show that all the maps of the type $\gamma_{1}^{-1}\gamma_{2}$
are in fact contained in a (proper) subset of $\Gamma_{J}$. This
will allow us to cut down on the number of sets. Let $R$ be a set of representatives for the $\Gamma$-orbits of bounded parabolic fixed points, as in Lemma \ref{number_of_balls}.

Indeed, in the $I_{t+1}$ case, $\Gamma\tilde{g}_{1},\Gamma\tilde{g}_{2}\in\Omega_{\mathrm{c}}^{\tilde{\epsilon}_{\mathtt{d}}}$ due to Lemma \ref{size_of_cusps}.
By the choice of the cover as in Lemma \ref{number_of_balls}, there is a parabolic fixed point $\xi\in R$ such that $\tilde{g}_{1},\tilde{g_2}\in \mathcal{F}T_{\tilde{\epsilon}_{\mathtt{d}}}(\Gamma_{\xi})$.
Since $\{\Gamma\tilde{g}_1a^{m}\}_{m=0}^{L+1}\subset\mathcal{F}\cusp_{\tilde{\epsilon}_{\mathtt{d}}}(X)$, i.e.\ the trajectory stayed in the cusp during the whole interval,
we obtain from Lemma \ref{size_of_cusps} that $\tilde{g}_{1}a^{L+1}\in \mathcal{F}T_{\tilde{\epsilon}_{\mathtt{d}}}(\Gamma_{\xi})$
as well.
So we have obtained $\tilde{g}_{1}a^{L+1},\tilde{g}_{2}\in \mathcal{F}T_{\tilde{\epsilon}_{\mathtt{d}}}(\Gamma_{\xi})$.
Since $\gamma_{1}^{-1}\gamma_{2}\tilde{g}_{2}=\tilde{g}_{1}a^{L+1}$, we
get from the precise invariance of $T_{\tilde{\epsilon}_{\mathtt{d}}}(\Gamma_{\xi})$
that $\gamma_{1}^{-1}\gamma_{2}\in\Gamma_{\xi}$ (see p.~\pageref{page of precisely invariant}).

Therefore, we have shown that in the $I_{t+1}$ case it suffices
to bound
\begin{equation*}
\Gamma_{I}=\{\nu\in\Gamma_{\xi}:\:d(\gamma\pi_{K}(o),\pi_{K}(o))\leq2r_0+L+1\}\subset\Gamma_{J}.
\end{equation*}
Let us do so. For simplicity, we may assume $\xi=\infty$. Recall that there is a subgroup $H\triangleleft\Gamma_{\xi}$ of finite index
$n(\Gamma)$, isomorphic to $\mathbb{Z}^{r_{\xi}}$, such that $H$
acts co-compactly on $\mathbb{R}^{r_{\xi_{j}}}\subset\mathbb{R}^{\mathtt{d}-1}=\partial\mathbb{H}^{\mathtt{d}}\backslash\{\xi\}$
by translations. Say $\{\nu_{1},\ldots\nu_{n(\Gamma)}\}$ are representatives
for $\Gamma_{\xi}/H$, then every $\nu\in\Gamma_{I}$ is of the form
$\nu=\gamma\nu_{k}$ for $\gamma\in H,\:1\leq k\leq n(\Gamma)$. Therefore
\begin{align*}
d(\gamma \pi_{K}(o),\pi_{K}(o))&\leq d(\nu\pi_{K}(o),\pi_{K}(o))+d(\nu_{k}\pi_{K}(o),\pi_{K}(o))\\&
\leq2r_0+L+1+\underset{1\leq k\leq n(\Gamma)}{\max}d(\nu_{k}\pi_{K}(o),\pi_{K}(o)).
\end{align*}
We have obtained $|\Gamma_{I}|\leq n(\Gamma)|\Gamma_{I}^{\prime}|$
for
\begin{equation*}
\Gamma_{I}^{\prime}=\{\gamma\in H:\:d(\gamma\pi_{K}(o),\pi_{K}(o))\leq2r_0+\underset{1\leq k\leq n(\Gamma)}{\max}d(\upsilon_{k},o)+L+1\}.
\end{equation*}
A simple hyperbolic geometry calculation shows that if $\gamma\in H$
then
\begin{equation*}
\|\gamma z_{0}-z_{0}\|_{\mathbb{R}^{\mathtt{d}-1}}\ll e^{\frac{1}{2}d(\gamma\pi_{K}(o),\pi_{K}(o))}
\end{equation*}
for $z_{0}$ the point on $\partial\mathbb{H}^{\mathtt{d}}$ right
below $\pi_{K}(o)$, i.e.\ with $\mathtt{d}$'th coordinate equal to $0$. Therefore, $|\Gamma_{I}^{\prime}|\leq|\Gamma_{\partial\mathbb{H}^{\mathtt{d}}}|$
for 
\begin{equation*}
\Gamma_{\partial\mathbb{H}^{\mathtt{d}}}=\{\gamma\in H:\:\|\gamma z_{0}-z_{0}\|_{\mathbb{R}^{\mathtt{d}-1}}\leq\tilde{F}e^{\frac{1}{2}L}\}
\end{equation*}
for some constant $\tilde{F}>0$.

In order to bound the size of this set, note that $Hz_0-z_0$ is a lattice in $\mathbb{R}^{r_{\xi}}$,
and so $|\Gamma_{\partial\mathbb{H}^{\mathtt{d}}}|$ is roughly equal
to the volume of a $r_{\xi}$-dimensional ball of radius $\tilde{F}e^{\frac{1}{2}L}$,
divided by $d(H)$, the volume of a fundamental polyhedron. That is,
$|\Gamma_{\partial\mathbb{H}^{\mathtt{d}}}|\ll e^{\frac{r_{\xi}}{2}L}$. 
Therefore, we obtain $|\Gamma_{I}|\leq\tilde{E} e^{\frac{r_{s+1}}{2}|I_{s+1}|}$,
where $\tilde{E}$ depends only on $\Gamma$.\\

So we get that $Z_{+}(O_{i_{1}},O_{i_{2}},V,\epsilon)$ can be covered by
\begin{equation*}
C^{s+t}\cdot e^{\delta\cdot\underset{j=1}{\overset{t}{\sum}}|J_{j}|+\underset{i=1}{\overset{s}{\sum}}\frac{r_{i}}{2}|I_{i}|}\cdot N_1\cdot\begin{cases}\tilde{E} e^{\frac{r_{s+1}}{2}|I_{s+1}|} & \mbox{the next interval is }I_{s+1}\\\tilde{B}e^{\delta|J_{t+1}|} & \mbox{the next interval is }J_{t+1}\end{cases}
\end{equation*}
sets of the desired form, which proves the claim for $C=N_1\cdot\max\{\tilde{E},\tilde{B}\}$.
\end{proof}

We can now prove the main lemma, which is a mild generalization of Lemma~\ref{main step}.
\begin{lemma}[Main Lemma]\label{main_lemma}
There is a constant $C>0$, depending only on
$\Gamma$, such that for all $0<\epsilon<\epsilon^{\prime}<\epsilon_{\mathtt{d}}$ small enough, for all $N\in\mathbb{N}$ and for all $$V:\{-N,\ldots,N\}\to\{0,1,\ldots,\mathtt{d}-1\},$$ the set
$Z(V,\epsilon,\epsilon^{\prime})$ can be covered by
\begin{equation*}
    \ll_{\epsilon^\prime} C^{\frac{4N}{|\log\epsilon|}}\cdot e^{(2N+1)\delta-\underset{i=1}{\overset{\mathtt{d}-1}{\sum}}\frac{2\delta-i}{2}\cdot|V^{-1}(i)|}
\end{equation*}
Bowen $(N,\eta)$-balls.
\end{lemma}

\begin{proof}
Let $\epsilon^{\prime}<\epsilon_{\mathtt{d}}^{\prime}$ be small enough with respect to all previous lemmas.

Recall that in Lemma \ref{main step} we assumed that $i_1,i_2\leq N_1$ i.e.\ $\Gamma O_{i_{1}},\Gamma O_{i_{2}}\subset\Omega_{\mathrm{nc}}^{\epsilon_{\mathtt{d}}^{\prime}}$.
Even if this is not the case, the main idea of the proof remains intact, and only minor adjustments are needed in the first and last intervals, as follows.
Now, it is possible that in the first interval $\tilde{g}_1$ is in $\Omega_{\mathrm{nc}}^{\epsilon^{\prime}}$ rather than in $\Omega_{\mathrm{nc}}^{{\epsilon_{\mathtt{d}}^{\prime}}}$, and similarly for $\tilde{g}_2$ in the last interval, and so
\begin{equation*}
d(\tilde{g}_{i},o)<\tilde{r}_0+\lceil|\log c_{2}\epsilon^{\prime}|\rceil
\end{equation*}
rather than $<r_0$. Therefore, $d(\gamma_1^{-1}\gamma_2\pi_{K}(o),\pi_{K}(o))$ would increase accordingly.
This would yield an increase of $\ll e^{2|\log\epsilon^{\prime}|\delta}$ in the number of element $\gamma_1^{-1}\gamma_2$ and so in the number of sets in the cover.

In any case, at the end of the $p+l$ iterations of Lemma \ref{main step}, we emerge with a covering of $Z_{+}(O_{i_{1}},O_{i_{2}},V,\epsilon)$ by
\begin{align*}
f&\ll e^{2|\log\epsilon^{\prime}|\delta}\cdot C^{p+l}\cdot e^{\delta\cdot\underset{j=1}{\overset{l}{\sum}}|J_{j}|+\underset{i=1}{\overset{p}{\sum}}\frac{r_{i}}{2}|I_{i}|}
\ll_{\epsilon^\prime} C^{p+l}\cdot e^{\delta N-\underset{i=1}{\overset{\mathtt{d}-1}{\sum}}\frac{2\delta-i}{2}\cdot|V^{-1}(i)|}
\end{align*}
sets, where we used the fact mentioned in \S\ref{PS BM Section} that $\delta>\frac{r_{i}}{2}$ for all cusps. Each of the sets is contained in the union of $\ll 1$ forward Bowen $(N,\eta$)-balls.
Note that
\begin{equation*}
p+l\leq\Big\lceil\frac{N}{\frac{2\lceil|\log c_{1}\epsilon|\rceil+1}{3}}\Big\rceil\leq\Big\lceil\frac{2N}{|\log\epsilon|}\Big\rceil,
\end{equation*}
because each interval of type $I_{t}$ is of length at least $2\lceil|\log c_{1}\epsilon|\rceil-1$. So we got
\begin{equation*}
\ll_{\epsilon^{\prime}}C^{\frac{2N}{|\log\epsilon|}}\cdot e^{\delta N-\underset{i=1}{\overset{\mathtt{d}-1}{\sum}}\frac{2\delta-i}{2}\cdot|V^{-1}(i)|}
\end{equation*}
forward Bowen $(N,\eta$)-balls.\\

Let us now cover $Z(V,\epsilon,\epsilon^{\prime})$ by (non-forward) Bowen $(N,\eta$)-balls. It is simply done by evoking Lemma \ref{main step} $N_0(\epsilon^\prime)^2$ times (for all $i_1,i_2$) to obtain a covering of  
\begin{align*}
T_{a}^{-N}Z(V,\epsilon,\epsilon^{\prime})&=\Big\{x\in \Omega_{\mathrm{nc}}^{\epsilon^{\prime}}\cap T_{a}^{-2N}(\Omega_{\mathrm{nc}}^{\epsilon^{\prime}}):\\&
\hspace{-2mm}T_{a}^{m}x\in \Omega_{\mathrm{c},i}^{\epsilon}\iff V(m-N)=i,\:\forall m\in[0,2N]\cap\mathbb{Z},\:\forall i\in[1,\mathtt{d}-1]\cap\mathbb{Z}\Big\}
\end{align*}
by
\begin{equation*}
\ll_{\epsilon^\prime} N_0(\epsilon^\prime)^{2}\cdot C^{\frac{4N}{|\log\epsilon|}}\cdot e^{(2N+1)\delta-\underset{i=1}{\overset{\mathtt{d}-1}{\sum}}\frac{2\delta-i}{2}\cdot|V^{-1}(i)|}
\end{equation*}
forward Bowen $(2N,\eta)$-balls. Clearly $T_{a}^{N}(yB_{2N,\eta}^{+})\subset(T_{a}^{N}y)B_{N,\eta}$ and so this yields a covering of $Z(V,\epsilon,\epsilon^{\prime})$ by Bowen $(N,\eta)$-balls.

To finish the proof, recall that $N_{0}(\epsilon^{\prime})\ll_{\epsilon^\prime} 1$ due to Lemma~\ref{number_of_balls} and Remark~\ref{remark_about_eta}.\end{proof}

\subsection{Proof of Theorem \ref{theorem1}}
Lemma \ref{main_lemma} is the key tool for estimating the entropy of the frame flow. Using this lemma we can proceed in proving the desired results. The rest of the proofs of Theorems \ref{theorem1}-\ref{theorem2} follow the scheme provided in \cite{einsiedler2011distribution}.

The following lemma concludes item \ref{thm1 item1} of Theorem \ref{theorem1}. Here $r_{\max}$ stands for the maximal rank of a cusp of $\Gamma\backslash\mathbb{H}^{\mathtt{d}}$.
\begin{lemma}\label{lemma_1}Let $(\mu_i)_{i\in\mathbb{N}}$ be as in the statement of Theorem \ref{theorem1}. Let $\mu$ be a weak-$\star$ limit
of a subsequence of $(\mu_{i})_{i\in\mathbb{N}}$. Then
\begin{equation*}
\mu(\Omega_{\mathrm{nc}}^{\epsilon})\geq1-\frac{1}{2\delta-r_{\max}}\cdot\frac{4\log(|\log\epsilon|)}{|\log\epsilon|}
\end{equation*}
for all $0<\epsilon<\epsilon_{\mathtt{d}}$ small enough.
In particular, $\mu$ is a probability measure. \end{lemma}

\begin{proof}
Without loss of generality, by passing to a subsequence, assume $(\mu_i)_{i\in\mathbb{N}}$ converges to $\mu$ in the weak-$\star$ topology.
For simplicity, for any $m\in\mathbb{Z}$ and $h>0$, let $\mathrm{Compact}(m,h) = T_{a}^{-m}(\Omega_{\mathrm{nc}}^{h})$ and $\mathrm{Cusp}(m,h) = T_{a}^{-m}(\Omega_{\mathrm{c}}^{h})$.

Let $\epsilon>0$ be small enough with respect to Lemma \ref{main_lemma}, and let
\begin{equation*}
\kappa>\kappa_{\epsilon}\coloneqq\frac{1}{2\delta-r_{\max}}\cdot\frac{4\log(|\log\epsilon|)}{|\log\epsilon|}.
\end{equation*}
Take $\epsilon_{0}>0$ small enough such that
\begin{equation}\label{eq_choice_of_epsilon_0}
g(\epsilon,\epsilon_{0})\coloneqq\frac{4\log(|\log\epsilon|)}{|\log\epsilon|}(1+3\epsilon_{0})+6(\delta+\alpha)\epsilon_{0}-(2\delta-r_{\max})(1+4\epsilon_{0})^{-1}(\kappa-\frac{\epsilon_{0}}{1+2\epsilon_{0}})<0    
\end{equation}
where $\alpha$ is as in the statement of Theorem \ref{theorem1}.
We may also assume that $\frac{\epsilon_{0}}{1+2\epsilon_{0}}<\kappa$,
and that $3\epsilon_{0}$ is small enough such that the assumptions
of Theorem \ref{theorem1} and the results of Lemma \ref{main_lemma}
hold. This is possible because
\begin{equation*}
\underset{\epsilon_{0}\to0^{+}}{\lim}g(\epsilon,\epsilon_{0})=\frac{4\log(|\log\epsilon|)}{|\log\epsilon|}-(2\delta-r_{\max})\kappa<0
\end{equation*}
by the choice of $\kappa$. The reason for this specific choice of $\epsilon_0$ will be clear at
the end of the proof.

Let $(\lambda_i)_{i\in\mathbb{N}}$ be as in the statement of Theorem \ref{theorem1}. Set the heights $H_{i}=\lambda_{i}^{\epsilon_{0}/4}$, where $i\in\mathbb{N}$ is large enough such that $\lambda_j<1$ for all $j\geq i$. Set $N_{i}=\lfloor-\frac{1}{2}\log\lambda_{i}\rfloor$ and $N_{i}^{\prime}=N_{i}+2\lfloor\epsilon_{0}N_{i}\rfloor$.
Set
\begin{equation*}
E_{\kappa,i}=\{x\in \mathcal{F}X:\:\frac{1}{2N_{i}^{\prime}+1}\underset{n=-N_{i}^{\prime}}{\overset{N_{i}^{\prime}}{\sum}}
1_{\mathrm{Cusp}(n,\epsilon)}(x)
>\kappa\}
\end{equation*}
and
\begin{equation*}
X_{\kappa,i}=\mathrm{Compact}(-N_{i}^{\prime},H_i)\cap\mathrm{Compact}(N_{i}^{\prime},H_i)\cap E_{\kappa,i}.
\end{equation*}

In what follows we will use Lemma~\ref{number of Z} and Lemma~\ref{main_lemma} to cover $X_{\kappa,i}$ by a relatively small amount of Bowen balls, and deduce that $\mu_{i}(X_{\kappa,i})$ is small. From this, the desired result will follow.

Note that once $H_{i}$ is small enough (i.e.\ $i$ is large enough), any trajectory of
a point 
$x\in\mathrm{Compact}(0,H_i)$ 
visits 
$\mathrm{Compact}(0,\tilde{\epsilon}_{\mathtt{d}})$
in up to $\lfloor\epsilon_{0}N_{i}\rfloor$ steps, either in the past
or the future. This is immediate from Lemma \ref{size_of_cusps}, since $\lceil|\log c_{2}H_{i}|\rceil\leq\lfloor\epsilon_{0}N_{i}\rfloor$.

It follows that
\begin{equation*}
\mathrm{Compact}(0,H_i)
\subset\underset{k=-\lfloor\epsilon_{0}N_{i}\rfloor}{\overset{\lfloor\epsilon_{0}N_{i}\rfloor}{\bigcup}}
\mathrm{Compact}(-k,\tilde{\epsilon}_{\mathtt{d}})
\end{equation*}
Therefore,
\begin{equation*}
X_{\kappa,i}\subset\underset{(k_{1},k_{2})\in\{-\lfloor\epsilon_{0}N_{i}\rfloor,\ldots,\lfloor\epsilon_{0}N_{i}\rfloor\}^{2}}{\bigcup}F_{k_{1},k_{2}}
\end{equation*}
for
\begin{equation*}
F_{k_{1},k_{2}}=\mathrm{Compact}(-N_{i}^{\prime}-k_{1},\tilde{\epsilon}_{\mathtt{d}})\cap \mathrm{Compact}(N_{i}^{\prime}-k_{2},\tilde{\epsilon}_{\mathtt{d}})\cap E_{\kappa,i}.
\end{equation*}

For any $(k_{1},k_{2})$, set $c=\lfloor-\frac{k_{1}+k_{2}}{2}\rfloor,\:d=\lfloor\frac{k_{1}-k_{2}}{2}\rfloor$
and $N=N_{i}^{\prime}+d$. Note that
\begin{equation*}
N_{i}+\lfloor\epsilon_{0}N_{i}\rfloor\leq N\leq N_{i}+3\lfloor\epsilon_{0}N_{i}\rfloor.
\end{equation*}
Then
\begin{equation*}
T_{a}^{c}F_{k_{1},k_{2}}\subset\mathrm{Compact}(-N,\tilde{\epsilon}_{\mathtt{d}})\cap\mathrm{Compact}(N,\tilde{\epsilon}_{\mathtt{d}})
\end{equation*}
or
\begin{equation*}
T_{a}^{c}F_{k_{1},k_{2}}\subset
\mathrm{Compact}(-N,\tilde{\epsilon}_{\mathtt{d}})\cap\mathrm{Compact}(N+1,\tilde{\epsilon}_{\mathtt{d}})
\end{equation*}
depending on whether $k_{1}$ and $k_{2}$ have the same parity or not. In any case,
\begin{equation*}
T_{a}^{c}F_{k_{1},k_{2}}\subset 
\mathrm{Compact}(-N,\epsilon)\cap \mathrm{Compact}(N,\epsilon)
\end{equation*}
assuming $\epsilon$ was initially chosen to be small enough.

Now, note that
\begin{equation*}
\frac{2N_{i}^{\prime}+1}{2N+1}\geq\frac{2N_{i}}{2(N_{i}+3\lfloor\epsilon_{0}N_{i}\rfloor)+1}\geq\frac{2N_{i}}{2(1+4\epsilon_{0})N_{i}}=(1+4\epsilon_{0})^{-1}
\end{equation*}
for large enough $i$, and that
\begin{equation*}
-\frac{|c-d|+|c+d|}{2N_{i}^{\prime}+1}\geq-\frac{2\lfloor\epsilon_{0}N_{i}\rfloor}{2(N_{i}+2\lfloor\epsilon_{0}N_{i}\rfloor)+1}\geq-\frac{\epsilon_{0}}{1+2\epsilon_{0}}.
\end{equation*}

Therefore, for all $x=T_{a}^{c}y\in T_{a}^{c}F_{k_{1},k_{2}}$,
\begin{align*}
\frac{1}{2N+1}\underset{n=-N}{\overset{N}{\sum}}1_{\mathrm{Cusp}(n,\epsilon)}(x)&=
\frac{1}{2N+1}\underset{n=-N}{\overset{N}{\sum}}1_{\mathrm{Cusp}(n+c,\epsilon)}(y)\\&=
\frac{1}{2N+1}\underset{n=-N_{i}^{\prime}+(c-d)}{\overset{N_{i}^{\prime}+(c+d)}{\sum}}1_{\mathrm{Cusp}(n,\epsilon)}(y)\\&\geq
\frac{2N_{i}^{\prime}+1}{2N+1}\frac{1}{2N_{i}^{\prime}+1}(\underset{n=-N_{i}^{\prime}}{\overset{N_{i}^{\prime}}{\sum}}1_{\mathrm{Cusp}(n,\epsilon)}(y)-|c-d|-|c+d|)\\&\geq
(1+4\epsilon_{0})^{-1}(\kappa-\frac{\epsilon_{0}}{1+2\epsilon_{0}}).
\end{align*}

We have obtained
\begin{align*}
T_{a}^{c}F_{k_{1},k_{2}}\subset\Big\{&x\in\mathrm{Compact}(-N,\epsilon)\cap\mathrm{Compact}(N,\epsilon):\:\\&
\hspace{-3mm}\big|\{n\in[-N,N]\cap\mathbb{Z}:\:x\in\mathrm{Cusp}(n,\epsilon)\}\big|\geq(1+4\epsilon_{0})^{-1}(\kappa-\frac{\epsilon_{0}}{1+2\epsilon_{0}})(2N+1)\Big\}.
\end{align*}

By Lemma \ref{number of Z}, $T_a^{c}F_{k_{1},k_{2}}$ can be covered by
\begin{equation*}
\ll_{\epsilon}e^{\frac{3\log(|\log\epsilon|)}{|\log\epsilon|}N}\leq
e^{\frac{3\log(|\log\epsilon|)}{|\log\epsilon|}(1+3\epsilon_{0})N_{i}}
\end{equation*}
sets $Z_{+}(V,\epsilon)$ with 
\begin{align*}
    |V^{-1}(\{1,\ldots,\mathtt{d}-1\})|&\geq (2N+1)(1+4\epsilon_{0})^{-1}(\kappa-\frac{\epsilon_{0}}{1+2\epsilon_{0}})\\&\geq 2N_{i}(1+4\epsilon_{0})^{-1}(\kappa-\frac{\epsilon_{0}}{1+2\epsilon_{0}}).
\end{align*}
Each set is covered, due to Lemma \ref{main_lemma}, by
\begin{align*}
f&\ll_{\epsilon} C^{\frac{4N}{|\log\epsilon|}}\cdot e^{2N\delta-\frac{2\delta-r_{\max}}{2}\cdot|V^{-1}(\{1,\ldots,\mathtt{d}-1\})|}\\&
\ll_{\epsilon} \exp\left(\left[\frac{4(1+3\epsilon_0)\log C}{|\log\epsilon|}+2(1+3\epsilon_0)\delta-(2\delta-r_{\max})(1+4\epsilon_{0})^{-1}(\kappa-\frac{\epsilon_{0}}{1+2\epsilon_{0}})\right]N_{i}\right)
\end{align*}
Bowen $(N,\eta)$-balls.

Therefore $F_{k_{1},k_{2}}$ can be covered by that many $T_{a}^{-c}$
images of Bowen $(N,\eta)$-balls. Each such image is contained in a Bowen $(N-|c|,\eta)$-ball. Note that these balls are in fact subsets of Bowen $(N_{i},\eta)$-balls because
\begin{equation*}
N-|c|=(N_{i}+2\lfloor\epsilon_{0}N_{i}\rfloor)+d-|c|\geq N_{i}.
\end{equation*}

Therefore, $X_{\kappa,i}$ can be covered by 
\begin{align*}
l&\ll_{\epsilon}(\epsilon_{0}N_{i})^{2}\cdot e^{\frac{3\log(|\log\epsilon|)}{|\log\epsilon|}(1+3\epsilon_{0})N_{i}}\cdot f\\&
\ll_{\epsilon,\epsilon_0}N_{i}^{2}\exp\Big(\Big[\frac{4\log(|\log\epsilon|)}{|\log\epsilon|}(1+3\epsilon_{0})+2(1+3\epsilon_0)\delta\\
&\hspace{2.5cm}-(2\delta-r_{\max})(1+4\epsilon_{0})^{-1}(\kappa-\frac{\epsilon_{0}}{1+2\epsilon_{0}})\Big]N_{i}\Big)
\end{align*}
Bowen $(N_{i},\eta)$-balls, denoted $S_{1},\ldots S_{l}$. We will use this cover to bound $\mu_i(X_{\kappa,i})$.

Note that these balls satisfy 
$$T_{a}^{N_i}S_j\subset\mathrm{Compact}(-N_i+N+c,\epsilon).$$
Moreover, $$-N_i+N+c=2\lfloor\epsilon_0 N_i\rfloor+(d+c)$$
and so $|-N_i+N+c|\leq 4\lfloor\epsilon_0 N_i\rfloor$.
It follows, using Lemma \ref{size_of_cusps}, that $T_{a}^{N_i}S_j\subset\Omega_{\mathrm{nc}}^{\lambda_{i}^{3\epsilon_0}}$ once $i$ is large enough.\\

Let us make the cover disjoint, i.e.\ define $\tilde{S}_{j}=S_{j}\smallsetminus\underset{k=1}{\overset{j-1}{\bigcup}}S_{k}$ for all $1\leq j\leq l$. Note that $S_{j}$ is a Bowen $(N_{i},\eta)$-ball and so $T_a^{N_i}S_j$ is a backward Bowen $(2N_i,\eta)$-ball. As such, it has width at most
\begin{equation*}
2\eta e^{-2N_{i}}\leq\lambda_{i}
\end{equation*}
in the $N^{-}$ direction, and width at most $1$ in the $MA$ and $N^+$ directions. 
Therefore, we have obtained
\begin{equation*}
(T_a^{N_i}\tilde{S}_{j})\times(T_a^{N_i}\tilde{S}_{j})\subset\left\{(x,y)\in \mathcal{F}\core_{\lambda_{i}^{3\epsilon_0}}(X)\times \mathcal{F}\core_{\lambda_{i}^{3\epsilon_0}}(X):\:x\in yB_{1}^{N^{+}}B_{1}^{MA}B_{\lambda_{i}}^{N^{-}}\right\}
\end{equation*}
for all $j$. So, by the assumption of Theorem \ref{theorem1},
\begin{equation*}
\mu_{i}\times\mu_{i}(\underset{j=1}{\overset{l}{\bigcup}}(T_a^{N_i}\tilde{S}_{j})\times(T_a^{N_i}\tilde{S}_{j}))\ll_{\epsilon_0}\lambda_{i}^{\delta-3\alpha\epsilon_{0}}\leq e^{-2(\delta-3\alpha\epsilon_{0})N_{i}}.
\end{equation*}
As $\{\tilde{S}_{j}\times\tilde{S}_{j}\}_{j=1}^{l}$ is a mutually
disjoint collection, and since $\mu_i$ is invariant under the frame flow,
\begin{align*}
\underset{j=1}{\overset{l}{\sum}}\mu_{i}(\tilde{S}_{j})^{2}&=\underset{j=1}{\overset{l}{\sum}}\mu_{i}(T_{a}^{N_i}\tilde{S}_{j})^{2}=\underset{j=1}{\overset{l}{\sum}}\mu_{i}\times\mu_{i}(T_{a}^{N_i}\tilde{S}_{j}\times T_{a}^{N_i}\tilde{S}_{j})\\&=\mu_{i}\times\mu_{i}(\underset{j=1}{\overset{l}{\bigcup}}T_{a}^{N_i}\tilde{S}_{j}\times T_{a}^{N_i}\tilde{S}_{j})\ll_{\epsilon_0}e^{-2(\delta-3\alpha\epsilon_{0})N_{i}}.
\end{align*}
As $\{\tilde{S}_{j}\}_{j=1}^{l}$ is a covering of $X_{\kappa,i}$,
\begin{align*}
\mu_{i}(X_{\kappa,i})^{2}&\leq(\underset{i=1}{\overset{l}{\sum}}\mu_{i}(\tilde{S}_{j}))^{2}\leq l\cdot\underset{i=1}{\overset{l}{\sum}}\mu_{i}(\tilde{S}_{j})^{2}\\&    \ll_{\epsilon,\epsilon_0}N_{i}^{2}\exp\Big(\Big[\frac{4\log(|\log\epsilon|)}{|\log\epsilon|}(1+3\epsilon_{0})+6(\delta+\alpha)\epsilon_0\\&-(2\delta-r_{\max})(1+4\epsilon_{0})^{-1}(\kappa-\frac{\epsilon_{0}}{1+2\epsilon_{0}})\Big]N_{i}\Big).
\end{align*}
By the choice of $\epsilon_{0}$ in Equation \eqref{eq_choice_of_epsilon_0}, the
latter exponent is negative, and therefore (as $\underset{i\to\infty}{\lim}N_{i}=\infty$)
$\underset{i\to\infty}{\lim}\mu_{i}(X_{\kappa,i})=0$.\\

Since $\mu_{i}$ is $T_{a}$-invariant and $\supp(\mu_{i})\subset\Omega_{\mathcal{F}}$, and using the fact that
\begin{equation*}
X = X_{\kappa,i} \cup \mathrm{Compact}(-N_{i}^{\prime},H_i)^{c} \cup \mathrm{Compact}(N_{i}^{\prime},H_i)^{c} \cup E_{\kappa,i}^{c},
\end{equation*}
we have for all $n\in\mathbb{N}$
\begin{align*}
    \mu_{i}(\Omega_{\mathrm{c}}^{\epsilon})&=\mu_i(\mathrm{Cusp}(n,\epsilon))=\int_{X} 1_{\mathrm{Cusp}(n,\epsilon)}d\mu_i\\&
    \leq \mu_i(X_{\kappa,i}) + \mu_i(\mathrm{Compact}(-N_{i}^{\prime},H_i)^{c}) +\mu_i(\mathrm{Compact}(N_{i}^{\prime},H_i)^{c}) \\&+\int_{E_{\kappa,i}^{c}} 1_{\mathrm{Cusp}(n,\epsilon)}d\mu_i\\&
    =\mu_{i}(X_{\kappa,i})+2\mu_{i}(\Omega_{\mathrm{c}}^{H_i})+\int_{E_{\kappa,i}^{c}} 1_{\mathrm{Cusp}(n,\epsilon)}d\mu_i
\end{align*}
and so, by the definition of $E_{\kappa,i}$,
\begin{align*}
\numberthis \label{inequality for escape of mass}
    \mu_{i}(\Omega_{\mathrm{c}}^{\epsilon})&=\frac{1}{2N_i^{\prime}+1}\sum_{n=-N_i^\prime}^{N_i^\prime}\mu_{i}(\mathrm{Cusp}(n,\epsilon))\\&
    \leq \mu_{i}(X_{\kappa,i})+2\mu_i(\Omega_{\mathrm{c}}^{H_i})+\int_{E_{\kappa,i}^{c}}\frac{1}{2N_i^{\prime}+1}\sum_{n=-N_i^\prime}^{N_i^\prime} 1_{\mathrm{Cusp}(n,\epsilon)}d\mu_i\\&
    \leq\mu_i(X_{\kappa,i})+2\mu_i(\Omega_{\mathrm{c}}^{H_i})+\kappa.
\end{align*}

Note that $\underset{\epsilon\to0^{+}}{\lim}\kappa_{\epsilon}=0$ and so we can take $\kappa$ to be as small as desired, if $\epsilon$ is small enough. Moreover, the first two terms in \eqref{inequality for escape of mass} tend to zero as $i\to\infty$. So, we see that for all $\upsilon>0$ there is $\epsilon>0$ small enough and $j\in\mathbb{N}$ such that  $\mu_i(\Omega_{c}^{\epsilon})<\upsilon$ for all $i>j$. This shows that the measures are ``almost'' supported on a compact set, namely that the sequence of measures is ``tight''. It follows from Prokhorov's Theorem \cite{prokhorov1956convergence} that the sequence $(\mu_i)_{i\in\mathbb{N}}$ converges in the weak topology, and its limit $\mu$ is a probability measure. 

Assume that $\epsilon$ was chosen such that $\mu(\partial\Omega_{\mathrm{c}}^{\epsilon})=0$. Then, taking the limit of Equation \eqref{inequality for escape of mass} as $i\to\infty$, we get from weak convergence that $\mu(\Omega_{\mathrm{nc}}^{\epsilon})\geq1-\kappa$. Since it holds for all $\kappa>\kappa_\epsilon$, it follows that
\begin{equation}\label{bound_on_measure_of_cusp}
    \mu(\Omega_{\mathrm{nc}}^{\epsilon})\geq 1-\kappa_\epsilon
\end{equation}
as well, as desired.\\

Note that since $\mu$ is a finite measure, there cannot be more than countably many values of $\epsilon^\prime$ for which $\mu(\partial\Omega_{\mathrm{c}}^{\epsilon^\prime})>0$. So for such an $\epsilon^\prime$ we may choose arbitrarily close $\epsilon_i>\epsilon^\prime$, for which Equation \eqref{bound_on_measure_of_cusp} holds, and take the limit as $\epsilon_i\to\epsilon^\prime$ to obtain the result for $\epsilon^{\prime}$ as well.
\end{proof}

\begin{lemma}\label{lemma_2}For every $0<\epsilon<\epsilon_{\mathtt{d}}$ there is a finite partition $\mathscr{P}$ of $\Omega_{\mathcal{F}}$ such
that for every $\kappa\in(0,\frac{1}{2})$ and every $N>0$, ``most'' of the
refinement $\mathscr{P}_{N}=\underset{n=-N}{\overset{N}{\bigvee}}T_{a}^{-n}\mathscr{P}$
is controlled by Bowen $(N,\eta)$-balls, in the sense that 
there is a subset $X^{\prime}\subset \Omega_{\mathcal{F}}$ such that:
\begin{enumerate}
\item\label{lemma_2:condition 1} $X^{\prime}=\underset{j=1}{\overset{l}{\bigcup}}S_{j}$ for $S_{j}\in\mathscr{P}_{N}$.
\item\label{lemma_2:condition 2} $X^{\prime}\subset T_{a}^{-N}(\Omega_{\mathrm{nc}}^{\epsilon})$.
\item Each $S_{j}$ is contained in a union of at most $3^{(\mathtt{d}-1)\kappa(2N+1)}$
Bowen $(N,\eta)$-balls.
\item $\mu(X^{\prime})\geq1-2\kappa^{-1}\mu(\Omega_{\mathrm{c}}^{\epsilon})$ for every $A$-invariant probability measure $\mu$.

\item For a given $\mu$, $\mathscr{P}$ can be chosen such that $\mu(\partial A)=0$
for all $A\in\mathscr{P}$.
\end{enumerate}
\end{lemma}

\begin{proof}
Let $\rho>0$ be small enough with respect to the injectivity radius of $\Omega_{\mathrm{nc}}^{\epsilon}$, for instance $\rho=\alpha\min\{\eta,l_0,\epsilon\}$, for $l_0$ the constant from Lemma \ref{deal_with_elliptic} and $\alpha\ll 1$.

Take a finite partition $\{P_{1},\ldots,P_{m}\}$ of $\Omega_{\mathrm{nc}}^{\epsilon}$
whose elements are contained in the quotients by $\Gamma$ of balls
of radius $\rho$, namely $\Gamma k_{i}B_{\rho}^{G}$
for $k_{i}\in G$. Of course, we may assume that these sets are mutually disjoint. 
Moreover, if $\mu$ is given, we can choose this partition to satisfy $\mu(\partial P_{j})=0$ for all $1\leq j\leq m$, since for all $x\in\Gamma\backslash G$ we have $\mu(\partial B_{r}(x))=0$ for all but countably many $r>0$.  
Assume, for now, that $\mu(\partial\Omega_{\mathrm{c}}^{\epsilon})=0$ as well.

Define $\mathscr{P}=\{\Omega_{\mathrm{c}}^{\epsilon},P_{1},\ldots,P_{k}\}$ and take some $S=\underset{n=-N}{\overset{N}{\bigcap}}T_{a}^{n}Q_{n}\in\mathscr{P}_{N}$. It is clear that for all $x\in S$ and $-N\leq n\leq N$, 
\begin{equation*}
T_{a}^{-n}x\in \Omega_{\mathrm{c}}^{\epsilon}\iff Q_{n}=\Omega_{\mathrm{c}}^{\epsilon}.
\end{equation*}
In particular, the function
\begin{equation*}
f(x)=\frac{1}{2N+1}\underset{n=-N}{\overset{N}{\sum}}1_{\Omega_{\mathrm{c}}^{\epsilon}}(T_{a}^{n}x)
\end{equation*}
is constant on $S$.
Define
\begin{equation*}
X^{\prime}=\{x\in T_{a}^{-N}( \Omega_{\mathrm{nc}}^{\epsilon}):\:f(x)\leq\kappa\}.
\end{equation*}
Given $\mu$,
\begin{align*}
\mu(\Omega_{\mathrm{c}}^{\epsilon})&=\frac{1}{2N+1}\underset{n=-N}{\overset{N}{\sum}}\mu(T_{a}^{n}( \Omega_{\mathrm{c}}^{\epsilon}))=\int fd\mu=\int_{f(x)>\kappa}fd\mu+\int_{f(x)\leq\kappa}fd\mu\\&\geq\kappa\cdot\mu(\{x\in\mathcal{F}X:\:f(x)>\kappa\})+0
\end{align*}
Set $A=T_{a}^{-N}(\Omega_{\mathrm{nc}}^{\epsilon})$ and $B=\{x\in \mathcal{F}X:\:f(x)\leq\kappa\}$
(and so $X^{\prime}=A\cap B)$. Then, as $\mu$ is a $T_a$-invariant probability measure supported
on $\Omega_{\mathcal{F}}$,
\begin{equation*}
\mu(X^{\prime})\geq1-\mu(A^{c})-\mu(B^{c})\geq1-(1+\kappa^{-1})\mu(\Omega_{\mathrm{c}}^{\epsilon})\geq1-2\kappa^{-1}\mu(\Omega_{\mathrm{c}}^{\epsilon}).
\end{equation*}
Clearly $X^{\prime}$ satisfies conditions \ref{lemma_2:condition 1}-\ref{lemma_2:condition 2} of the lemma, as it
is precisely the union of elements $S=\underset{n=-N}{\overset{N}{\bigcap}}T_{a}^{n}Q_{n}$
such that $Q_{-N}\not=\Omega_{\mathrm{c}}^{\epsilon}$, and
$Q_{n}=\Omega_{\mathrm{c}}^{\epsilon}$ for up to $\kappa(2N+1)$
different $-N\leq n\leq N$.

It remains only to cover each such element by $3^{(\mathtt{d}-1)\kappa(2N+1)}$
Bowen $(N,\eta)$-balls. For simplicity, for such $S$ define
\begin{equation*}
S^{\prime}=T_{a}^{N}S=\underset{n=-N}{\overset{N}{\bigcap}}T_{a}^{N+n}Q_{n}=\underset{n=0}{\overset{2N}{\bigcap}}T_{a}^{n}Q_{n-N}
\end{equation*}
and
\begin{equation*}
W=\{0\leq n\leq2N:\:Q_{n-N}=\Omega_{\mathrm{c}}^{\epsilon}\}.
\end{equation*}
We will cover $S^{\prime}$ by $3^{(\mathtt{d}-1)\kappa(2N+1)}$ backward Bowen $(2N,\rho)$-balls. By taking images of these balls by $T_{a}^{-N}$, we will get a cover of $S$ by the same amount of (two-sided) Bowen $(N,\rho)$ balls. As $\rho<\eta$, each such Bowen $(N,\rho)$-ball
is contained in a Bowen $(N,\eta)$-ball, which will give the desired result.\\

We prove the following claim by induction: for all $0\leq n\leq2N$, the
set $S^{\prime}$ can be covered by $3^{(\mathtt{d}-1)|\{0,\ldots,n\}\cap W|}$
many backward Bowen $(n,\rho)$-balls.

First, for $n=0$, as
\begin{equation*}
S^{\prime}\subset Q_{-N}\in\{P_{1},\ldots,P_{k}\},
\end{equation*}
and $P_{j}\subset\Gamma k_{j}B_{\rho}^{G}$, the result is clear
for $3^{0}$ sets, namely $P_{j}$.

Assume we covered $S^{\prime}$ by $3^{(\mathtt{d}-1)|\{0,\ldots,n\}\cap W|}$
many sets for some $n\in[0,2N-1]$. Let $\Gamma kB_{n,\rho}^{-}$
be one of those. If $n+1\not\in W$ then $S^{\prime}\subset T_{a}^{n+1}P_{j}$ for some $j$. Therefore $$S^{\prime}\cap\Gamma kB_{n,\rho}^{-}\subset\Gamma kB_{n,\rho}^{-}\cap\Gamma k_{j}B_{\rho}^{G}a^{n+1}.$$
Since $S^\prime\subset \Omega_{\mathrm{nc}}^{\epsilon}$, and $\rho$ was chosen to be small enough with respect to the injectivity radius, it follows that in fact $$S^\prime\cap\Gamma kB_{n,\rho}^{-}\subset \Gamma k^{\prime}B_{n+1,\rho}^{-}$$ for some $k^\prime\in G$.
Therefore, by replacing $\Gamma kB_{n,\rho}^{-}$ with $\Gamma k^{\prime}B_{n+1,\rho}^{-}$,
for each set $\Gamma kB_{n,\rho}^{-}$ in the given cover, we
obtain a new cover of $S^{\prime}$ by backward Bowen $(n+1,\rho)$-balls,
again of size $3^{(\mathtt{d}-1)|\{0,\ldots,n\}\cap W|}=3^{(\mathtt{d}-1)|\{0,\ldots,n+1\}\cap W|}$.

Otherwise, $n+1\in W$. In this case we simply split each backward
Bowen $(n,\rho)$-ball into $3^{\mathtt{d}-1}$ backward Bowen
$(n+1,\rho)$-balls, each of them smaller in the $N^{-}$ directions
by a factor of $e<3$. We emerge with $$3^{(\mathtt{d}-1)}\cdot3^{(\mathtt{d}-1)|\{0,\ldots,n\}\cap W|}=3^{(\mathtt{d}-1)|\{0,\ldots,n+1\}\cap W|}$$
many sets of the right form.
This proves the claim.

To conclude, in the last step of the induction, as $|W|\leq\kappa(2N+1)$
(since $S\subset X^{\prime}$), we get up to $3^{(\mathtt{d}-1)\kappa(2N+1)}$
sets as desired.\\

In order to finish the proof, one only needs to treat the case where $\mu(\partial\Omega_{\mathrm{c}}^{\epsilon})>0$. In this case, choose $\epsilon^{\prime}>\epsilon$ such that $\mu(\Omega_{\mathrm{c}}^{\epsilon^\prime})<\frac{4}{3}\mu(\Omega_{\mathrm{c}}^{\epsilon})$ and $\mu(\partial\Omega_{\mathrm{c}}^{\epsilon^{\prime}})=0$. This is possible because, as we mentioned, there are only countably many cusp regions whose boundary is of positive measure. Then the partition $X^{\prime}$ that was constructed for $\epsilon^{\prime}$ works for $\epsilon$ as well.
\end{proof}

We now prove item \ref{thm1 item2} of Theorem \ref{theorem1}, which is the main result of the theorem.
\begin{proof}[Proof of Theorem \ref{theorem1}]
Let $\mu$ be a weak-$\star$ limit of a subsequence of $(\mu_{n})_{n\in\mathbb{N}}$. Without loss of generality, by passing to a subsequence, assume $(\mu_n)_{n\in\mathbb{N}}$ converges to $\mu$.
Fix some $0<\epsilon<\epsilon_{\mathtt{d}}$, small enough with respect
to all previous lemmas, and for which $\kappa=\mu(\Omega_{\mathrm{c}}^\epsilon)^{1/2}<\frac{1}{2}$ and $\mu(\partial\Omega_{\mathrm{c}}^{\epsilon})=0$.
Let $\mathscr{P}$ be the partition from Lemma \ref{lemma_2}, all
whose elements have $\mu$-null boundaries.

Let $\epsilon_{0}>0$ be given, small enough so that the assumptions
of Theorem \ref{theorem1} hold.
Let $m\in\mathbb{N}$ be such that $H_{i}=\lambda_{i}^{\epsilon_{0}}<\epsilon$
for all $i\geq m$, and fix some $i\geq m$. Set $N_{i}=\lfloor-\frac{1}{2}\log\lambda_{i}\rfloor$. Define $X_{i}$ to be the set guaranteed in Lemma \ref{lemma_2},
with respect to $N_{i}$ and $\kappa$.
Here we assumed
$\kappa=\mu(\Omega_{\mathrm{c}}^{\epsilon})^{1/2}>0$. Otherwise, this is essentially the convex co-compact case, and we may set instead any $0<\kappa<\frac{1}{2}$ to obtain, by Lemma \ref{lemma_2}, a set $X_i$ of full $\mu_i$-measure, and continue the same way.

For all $P=S_{j}\in\mathscr{P}_{N_{i}}$ in the union $\underset{j=1}{\overset{l}{\bigcup}}S_{j}$
defining $X_{i}$, there is a covering $P\subset\underset{k=1}{\overset{3^{(\mathtt{d}-1)\kappa(2N_{i}+1)}}{\bigcup}}A_{k}$
by Bowen $(N_{i},\eta)$-balls. Define 
\begin{equation*}
\tilde{P}_{n}=(A_{n}\cap P)\smallsetminus\underset{k=1}{\overset{n-1}{\bigcup}}A_{k}
\end{equation*}
for all $n\geq 1$.
Then $\{\tilde{P}_{n}\}_{n=1}^{3^{(\mathtt{d}-1)\kappa(2N_{i}+1)}}$ is
a collection of mutually disjoint sets, each contained in a Bowen $(N_{i},\eta)$-ball,
which their union is $P$. So $\{\tilde{P}_{n}\}_{n=1}^{3^{(\mathtt{d}-1)\kappa(2N_{i}+1)}}$
is a measurable partition of $P=S_{j}$.

Define
\begin{equation*}
\mathcal{Q}_{i}=\underset{P\in\mathscr{P}_{N_{i}}:\:P\cap X_{i}=\emptyset}{\bigcup}\{P\}\cup\underset{P\in\mathscr{P}_{N_{i}}:\:P\cap X_{i}\not=\emptyset}{\bigcup}\{\tilde{P}_{n}\}_{n=1}^{3^{(\mathtt{d}-1)\kappa(2N_{i}+1)}}.
\end{equation*}
Clearly, $Q_{i}$ is a measurable partition of $X$.

Take $A\in\mathcal{Q}_{i},\:B\in\mathscr{P}_{N_{i}}$.
\begin{enumerate}
\item Assume $A=P$ for $P\in\mathscr{P}_{N_{i}}$ such that $P\cap X_{i}=\emptyset$. Then $A,B\in\mathscr{P}_{N_{i}}$ and so $A\cap B=\emptyset$ or $A=B$,
as the elements of $\mathscr{P}_{N_{i}}$ are mutually disjoint.

\item Otherwise, $A=\tilde{P}_{n}$ for $P\in\mathscr{P}_{N_{i}}$ such
that $P\subset X_{i}$ (that is, $P$ is one of the sets defining
$X_{i}$). Then $A\cap B=\emptyset$ unless $B=P$, in which case $A\cap B=A$.

\end{enumerate}
In particular, for a given $B\in\mathscr{P}_{N_{i}}$, we have
\begin{equation*}
\log((\mu_{i})_{B}(A))\cdot(\mu_{i})_{B}(A)\not=0
\end{equation*}
only for at most $3^{(\mathtt{d}-1)\kappa(2N_{i}+1)}$ elements $A\in\mathcal{Q}_{i}$, where $(\mu_i)_B$ is the restriction of $\mu_i$ to $B$, normalized to be a probability measure. 
Therefore the entropy satisfies
\begin{equation*}
H_{(\mu_{i})_{B}}(\mathcal{Q}_{i})\leq\log3^{(\mathtt{d}-1)\kappa(2N_{i}+1)}=\log3^{\mathtt{d}-1}\cdot\kappa(2N_{i}+1).
\end{equation*}

We get that the conditional entropy satisfies
\begin{equation*}
H_{\mu_{i}}(\mathcal{Q}_{i}|\mathscr{P}_{N_{i}})=\underset{B\in\mathscr{P}_{N_{i}}}{\sum}\mu_{i}(B)\cdot H_{(\mu_{i})_{B}}(\mathcal{Q}_{i})\leq\log3^{\mathtt{d}-1}\cdot\kappa(2N_{i}+1).
\end{equation*}
As $\mathcal{Q}_{i}$ is finer than $\mathscr{P}_{N_{i}}$, we get
$\mathcal{Q}_{i}=\mathcal{Q}_{i}\vee\mathscr{P}_{N_{i}}$ and so
\begin{equation*}
H_{\mu_{i}}(\mathcal{Q}_{i})=H_{\mu_{i}}(\mathscr{P}_{N_{i}})+H_{\mu_{i}}(\mathcal{Q}_{i}|\mathscr{P}_{N_{i}})
\end{equation*}
and so
\begin{equation*}
H_{\mu_{i}}(\mathcal{Q}_{i})\leq H_{\mu_{i}}(\mathscr{P}_{N_{i}})+\log3^{\mathtt{d}-1}\cdot\kappa(2N_{i}+1).
\end{equation*}

Therefore, in order to get a bound on $H_{\mu_{i}}(\mathscr{P}_{N_{i}})$
we are interested in bounding $H_{\mu_{i}}(\mathcal{Q}_{i})$. Clearly, 
\begin{equation*}
H_{\mu_{i}}(\mathcal{Q}_{i})\geq H_{\mu_{i}}(\mathcal{Q}_{i}|\{X_{i},X_{i}^{c}\})\geq\mu_{i}(X_{i})\cdot H_{(\mu_{i})_{X_{i}}}(\mathcal{Q}_{i}).
\end{equation*}
Using the fact that entropy is controlled by an $L^{2}$-norm, and
the already mentioned fact that for $A\in\mathcal{Q}_{i}$ either
$A\subset X_{i}$ or $A\cap X_{i}=\emptyset$, we obtain
\begin{equation*}
H_{(\mu_{i})_{X_{i}}}(\mathcal{Q}_{i})\geq-\log\underset{A\in\mathcal{Q}_{i}}{\sum}\big((\mu_{i})_{X_{i}}(A)\big)^{2}=-\log\underset{A\in\mathcal{Q}_{i},\:A\subset X_{i}}{\sum}\frac{\mu_{i}(A)^{2}}{\mu_{i}(X_{i})^{2}}.
\end{equation*}

Let us estimate this sum. Every $A\in\mathcal{Q}_{i}$ with $A\subset X_{i}$
is contained in a Bowen $(N_{i},\eta)$-ball. In addition, $\{A\in\mathcal{Q}_{i}:\:A\subset X_{i}\}$
is a mutually disjoint family and $$X_{i}\subset T_{a}^{-N_i}(\Omega_{\mathrm{nc}}^{\epsilon})\subset T_{a}^{-N_i}(\mathcal{F}\core_{H_{i}}(X)).$$
Therefore, by the assumptions of Theorem \ref{theorem1} (as explained in the
proof of Lemma~\ref{lemma_1}),
\begin{align*}
\underset{A\in\mathcal{Q}_{i},\:A\subset X_{i}}{\sum}\mu_{i}(A)^{2}&=\mu_{i}\times\mu_{i}\Big(\underset{A\in\mathcal{Q}_{i},\:A\subset X_{i}}{\bigcup}T_{a}^{N_i}(A)\times  T_{a}^{N_i}(A)\Big)\\&
\leq E(\epsilon_{0})\lambda_{i}^{\delta-\alpha\epsilon_{0}}\leq E(\epsilon_{0})e^{-2(\delta-\alpha\epsilon_{0})N_{i}}
\end{align*}
where $E(\epsilon_0)$ is the implicit constant in assumption \ref{assumption_2_thm_1} of Theorem \ref{theorem1}. Therefore,
\begin{equation*}
H_{(\mu_{i})_{X_{i}}}(\mathcal{Q}_{i})\geq-\log\frac{E(\epsilon_{0})e^{-2(\delta-\alpha\epsilon_{0})N_{i}}}{\mu_{i}(X_{i})^{2}}=2\log(\mu_{i}(X_{i}))-\log E(\epsilon_{0})+2(\delta-\alpha\epsilon_{0})N_{i}
\end{equation*}
and so,
\begin{equation*}
H_{\mu_{i}}(\mathcal{Q}_{i})\geq2\mu_{i}(X_{i})\log(\mu_{i}(X_{i}))-\mu_{i}(X_{i})\log E(\epsilon_{0})+2\mu_{i}(X_{i})(\delta-\alpha\epsilon_{0})N_{i}.
\end{equation*}
Note that the first two terms are bounded as $i$ approaches $\infty$. Yet, the third term approaches $\infty$ since
$$\underset{i\to\infty}{\liminf}(\mu_i(X_i))\geq\underset{i\to\infty}{\liminf}(1-2\kappa^{-1}\mu_i(\Omega_{c}^{\epsilon}))=1-2\kappa^{-1}\mu(\Omega_{c}^{\epsilon})\geq1-2\kappa>0$$
and of course $N_i\to\infty$. The first equality was due to $\mu$ being a probability measure by Lemma \ref{lemma_1}, and $\Omega_{\mathrm{c}}^{\epsilon}$ being a continuity set.

Therefore, for all large enough $i$,
\begin{align*}
H_{\mu_{i}}(\mathcal{Q}_{i})&\geq2\mu_{i}(X_{i})(\delta-\alpha\epsilon_{0})N_{i}-\alpha\epsilon_{0}\cdot\mu_{i}(X_{i})N_{i}=\mu_{i}(X_{i})(2\delta-3\alpha\epsilon_{0})N_{i}\\&\geq(1-2\kappa^{-1}\mu_{i}(\Omega_{\mathrm{c}}^{\epsilon}))\cdot(2\delta-3\alpha\epsilon_{0})\cdot N_{i}.
\end{align*}
Altogether, we get 
\begin{equation*}
(1-2\kappa^{-1}\mu_{i}(\Omega_{\mathrm{c}}^{\epsilon}))\cdot(2\delta-3\alpha\epsilon_{0})\cdot N_{i}\leq H_{\mu_{i}}(\mathcal{Q}_{i})\leq H_{\mu_{i}}(\mathscr{P}_{N_{i}})+\log3^{\mathtt{d}-1}\cdot\kappa(2N_{i}+1).
\end{equation*}

Fix some $N_{0}\in\mathbb{N}$. By subadditivity of entropy, and $\mu_{i}$'s
$T_{a}$-invariance, we get for all $i>N_{0}$ (set $k_{i}=\lceil\frac{N_{i}}{N_{0}}\rceil$): \begin{align*}
H_{\mu_{i}}(\mathscr{P}_{N_{i}})&\leq H_{\mu_{i}}(\mathscr{P}_{k_{i}N_{0}})=H_{\mu_{i}}(\underset{j=1}{\overset{k_{i}}{\bigvee}}\underset{n=(2(j-1)-k_{i})N_{0}}{\overset{(2j-k_{i})N_{0}}{\bigvee}}T_{a}^{-n}\mathscr{P})\\&
\leq\underset{j=1}{\overset{k_{i}}{\sum}}H_{\mu_{i}}(\underset{n=(2(j-1)-k_{i})N_{0}}{\overset{(2j-k_{i})N_{0}}{\bigvee}}T_{a}^{-n}\mathscr{P})\\&
=\underset{j=1}{\overset{k_{i}}{\sum}}H_{\mu_{i}}(T_{a}^{(2j-1-k_{i})N_{0}}\underset{n=(2(j-1)-k_{i})N_{0}}{\overset{(2j-k_{i})N_{0}}{\bigvee}}T_{a}^{-n}\mathscr{P})\\&
=\underset{j=1}{\overset{k_{i}}{\sum}}H_{\mu_{i}}(\underset{n=-N_{0}}{\overset{N_{0}}{\bigvee}}T_{a}^{-n}\mathscr{P})=k_{i}H_{\mu_{i}}(\mathscr{P}_{N_{0}})
\end{align*}
and so
\begin{equation*}
(1-2\kappa^{-1}\mu_{i}(\Omega_{\mathrm{c}}^{\epsilon}))\cdot(2\delta-3\alpha\epsilon_{0})\cdot N_{i}\leq\lceil\frac{N_{i}}{N_{0}}\rceil H_{\mu_{i}}(\mathscr{P}_{N_{0}})+\log3^{\mathtt{d}-1}\cdot\kappa(2N_{i}+1)
\end{equation*}
i.e.\ 
\begin{equation}\label{entropy_estimate}
H_{\mu_{i}}(\mathscr{P}_{N_{0}})\geq\Big((1-2\kappa^{-1}\mu_{i}(\Omega_{\mathrm{c}}^{\epsilon}))\cdot(2\delta-3\alpha\epsilon_{0})-2\log3^{\mathtt{d}-1}\cdot\kappa\Big)\cdot\frac{N_{i}}{\lceil\frac{N_{i}}{N_{0}}\rceil}-\log3^{\mathtt{d}-1}\cdot\kappa\cdot\frac{1}{\lceil\frac{N_{i}}{N_{0}}\rceil}.
\end{equation}

Note that $\underset{i\to\infty}{\lim}\mu_{i}(\Omega_{\mathrm{c}}^{\epsilon})=\mu(\Omega_{\mathrm{c}}^{\epsilon})$. 
Moreover, $\frac{N_{i}}{\lceil\frac{N_{i}}{N_{0}}\rceil}\to N_{0}$
and $\frac{1}{\lceil\frac{N_{i}}{N_{0}}\rceil}\to0$ as $i\to\infty$. Therefore, the RHS of equation \eqref{entropy_estimate} approaches
\begin{equation*}
\Big((1-2\kappa^{-1}\mu(\Omega_{\mathrm{c}}^{\epsilon}))\cdot(2\delta-3\alpha\epsilon_{0})-2\log3^{\mathtt{d}-1}\cdot\kappa\Big)\cdot N_{0}
\end{equation*}

As for the LHS of Equation \eqref{entropy_estimate}, we claim that
it converges to $H_{\mu}(\mathscr{P}_{N_{0}})$. Recall that per our choice $\mu(\partial A)=0$ for all $A\in\mathscr{P}$.
As $T_{a}$ is a measure-preserving homeomorphism, we get $\mu(\partial(T_{a}^{k}(A)))=0$
for all $A\in\mathscr{P}$ and $|k|\leq N_{0}$.
In general $\partial(\underset{i\in I}{\bigcap}A_{i})\subset\underset{i\in I}{\bigcup}\partial A_{i}$
for any finite set of indices $I$. As every set $A\in\mathscr{P}_{N_{0}}$
is of the form $A=\underset{i=-N_{0}}{\overset{N_{0}}{\bigcap}}T_{a}^{k}(A)$,
we obtain $\mu(\partial A)=0$ for all $A\in\mathscr{P}_{N_{0}}$.
Recall that by definition
\begin{equation*}
H_{\mu_{i}}(\mathscr{P}_{N_{0}})=-\underset{A\in\mathscr{P}_{N_{0}}}{\sum}\mu_{i}(A)\log(\mu_{i}(A)).
\end{equation*}
As it is just a finite sum of elements, each converging to $\mu(A)\log(\mu(A))$
(by the Portmanteau theorem), we obtain
\begin{equation*}
\underset{i\to\infty}{\lim}H_{\mu_{i}}(\mathscr{P}_{N_{0}})=H_{\mu}(\mathscr{P}_{N_{0}})
\end{equation*}
as desired.

Therefore, by taking the limit of equation \eqref{entropy_estimate}
as $i\to\infty$ we get
\begin{equation*}
H_{\mu}(\mathscr{P}_{N_{0}})\geq\Big((1-2\kappa^{-1}\mu(\Omega_{\mathrm{c}}^{\epsilon}))\cdot(2\delta-3\alpha\epsilon_{0})-2\log3^{\mathtt{d}-1}\cdot\kappa\Big)N_{0}.
\end{equation*}
Therefore
\begin{align*}
h_{\mu}(T)&=\underset{\mathscr{Q}}{\sup}\underset{n\to\infty}{\lim}\frac{\underset{i=-n}{\overset{n}{\bigvee}}T^{-n}\mathscr{Q}}{2n+1}\geq\underset{N_{0}\to\infty}{\lim}\frac{H_{\mu}(\mathscr{P}_{N_{0}})}{2N_{0}+1}\\&\geq\underset{N_{0}\to\infty}{\lim}\Big((1-2\mu(\Omega_{\mathrm{c}}^{\epsilon})^{1/2})\cdot(2\delta-3\alpha\epsilon_{0})-2\log3^{\mathtt{d}-1}\cdot\mu(\Omega_{\mathrm{c}}^{\epsilon})^{1/2}\Big)\cdot\frac{N_{0}}{2N_{0}+1}\\&=\Big((1-2\mu(\Omega_{\mathrm{c}}^{\epsilon})^{1/2})\cdot(2\delta-3\alpha\epsilon_{0})-2\log3^{\mathtt{d}-1}\cdot\mu(\Omega_{\mathrm{c}}^{\epsilon})^{1/2}\Big)\cdot\frac{1}{2}
\end{align*}
As $\underset{\epsilon\to0^{+}}{\lim}\mu(\Omega_{\mathrm{c}}^{\epsilon})=0$,
and as the calculation holds for all small enough $\epsilon,\epsilon_{0}>0$ (apart for countably many $\epsilon$),
we can take the limit as $\epsilon,\epsilon_{0}\to0^{+}$ and get
$h_{\mu}(T)\geq\delta$.\end{proof}

To deduce item \ref{thm1 item3} of Theorem \ref{theorem1} and conclude the proof, it is enough to show that if $\Gamma$ is Zariski-dense in $G$ then $m_{\BM}^{\mathcal{F}}$ is the unique measure of entropy $\delta(\Gamma)$, which is the maximal entropy. For the geodesic flow, it follows from Theorem~\ref{maximal_entropy} even without assuming $\Gamma$ is Zariski-dense. In particular, Theorem \ref{theorem1} can be restated in the geodesic flow.
For the frame flow, the characterization of the Bowen-Margulis measure as the unique measure of maximal entropy essentially reduces to showing that $m_{\BM}^{\mathcal{F}}$ is ergodic, a result that was established by Winter in \cite{winter2015mixing}. This reduction is well-known, for completeness we provide a proof below. 

\begin{proposition}
Assume that $\Gamma<G$ is Zariski-dense and geometrically finite. Then $m_{\BM}^{\mathcal{F}}$ is the unique measure of maximal entropy $\delta$. 
\end{proposition}

\begin{proof}
Let $\mu$ be an $A$-invariant probability measure on $\Gamma\backslash G$.
First, note that $h_{\mu}(a_1)=h_{(\pi_M)_{\ast}\mu}(a_1)$ where $(\pi_M)_{\ast}\mu$ is the pushforward of $\mu$ to $\Gamma\backslash G/M$, since compact (isometric) extensions do not increase entropy.
In particular, the maximal entropy of the frame flow is $\delta$, which is realized by $m_{\BM}^{\mathcal{F}}$.

Next, assume that $\mu$ is of maximal entropy. Then its pushforward measure is of entropy $\delta$. Uniqueness of measure of entropy $\delta$ on $\Gamma\backslash G/M$ is known in our setup (Theorem \ref{maximal_entropy}), and so $\mu$ is some lift of $m_{\BM}$.

For all $m\in M$ define a measure $\mu_{m}$ on $\Gamma\backslash G$ by $\mu_{m}(Y)=\mu(Ym)$. These measures are all $A$-invariant probability measures. It is clear that by averaging these measures with respect to the Haar measure $\lambda$ on $M$, we obtain a lift of $m_{\BM}$ which is $M$-invariant from the right. This must be the lift using the Haar measure $\lambda$, by uniqueness of the Haar measure. In other words, this lift is the Bowen-Margulis measure on $\Gamma\backslash G$. So we obtained
\begin{equation}\label{decomposition_of_m_BM}
    m_{\BM}^{\mathcal{F}}=\underset{M}{\int}\mu_{m}d\lambda(m).
\end{equation}

By \cite{winter2015mixing}, $m_{\BM}^{\mathcal{F}}$ is mixing and in particular ergodic. Therefore, it is an extreme point in the space of invariant probability measures, and in particular the decomposition \eqref{decomposition_of_m_BM} is trivial and so $\mu=m_{\BM}^{\mathcal{F}}$.
\end{proof}

\section{Proof of Theorem \ref{theorem2}}
The following lemma is very useful for estimating entropy on hyperbolic
orbifolds. It goes back to Katok, who showed in \cite{katok1980lyapunov} that equality holds if the space is compact.
\begin{lemma}\label{entropy_bound_by_Bowen_balls}Let $\mu$ be an $A$-invariant
probability measure on $\mathcal{F}X=\Gamma\backslash G$. For any given $N\geq1$ and
$\epsilon_{0}>0$, let $\BC_{\rho}(N,\epsilon_{0})$ be the minimal
number of Bowen $(N,\rho)$-balls needed to cover any subset of $\mathcal{F}X$
of measure larger than $1-\epsilon_{0}$.

Then
\begin{equation*}
h_{\mu}(T_a)\leq\underset{\epsilon_{0}\to0^{+}}{\liminf}\underset{N\to\infty}{\liminf}\frac{\log(\BC_{\rho}(N,\epsilon_{0}))}{2N}
\end{equation*}
\end{lemma}

\begin{remark}\label{remark_about_BC}By Remark \ref{bowen_ball_remark}, each Bowen $(N,\rho^{\prime})$-ball can be covered by $$\ll \lceil\frac{\rho^{\prime}}{\rho}\rceil^{\frac{(\mathtt{d}+1)\mathtt{d}}{2}}\ll_{\rho,\rho^{\prime}} 1$$ many Bowen $(N,\rho)$-balls.
Therefore
\begin{equation*}
\underset{N\to\infty}{\liminf}\frac{\log(\BC_{\rho}(N,\epsilon_{0}))}{2N}=\underset{N\to\infty}{\liminf}\frac{\log(\BC_{\rho^{\prime}}(N,\epsilon_{0}))}{2N},
\end{equation*}
and the limit does not depend on $\rho$. 
\end{remark}

\begin{proof}
Let $\upsilon>0$. Let $\mathscr{P}=\{Q,S_{1},\ldots,S_{l}\}$ be a finite measurable partition
satisfying the following conditions:
\begin{enumerate}
\item $h_{\mu}(T_{a},\mathscr{P})>h_{\mu}(T_{a})-\upsilon$
\item $\mu((\partial S_{i})B_{\kappa}^{G})<E\kappa$ for all $1\leq i\leq l$,
for some constant $E>0$ and for all $\kappa>0$ small enough.
\end{enumerate}
The first condition can be met by the definition of the measure-theoretic
entropy. The second condition is possible due to Lebesgue's theorem, as for any $x\in\Gamma\backslash G$ the function
$\phi(r)=\mu(B_{r}(x))$ is monotone and so a.e.\ differentiable.

Let $\mathscr{P}_{N}=\underset{i=-N}{\overset{N}{\bigvee}}T_{a}^{-i}\mathscr{P}$. For $x\in \mathcal{F}X$, let $[x]_{\mathscr{P}_N}$ stand for the unique atom of $\mathscr{P}_N$ containing $x$. Set $\rho_{N}=\rho N^{-2}$. We would like to show that $\mu(E_{N})>1-DN^{-1}$ for a constant
$D>0$ and for all $N\in\mathbb{N}$ large enough, where
\begin{equation*}
E_{N}=\{x\in \mathcal{F}X:\:xB_{N,2\rho_{N}}\subset[x]_{\mathscr{P}_{N}}\}.
\end{equation*}
That is, for most points $x\in \mathcal{F}X$, the atom in $\mathscr{P}_{N}$
containing $x$ contains a Bowen ball about $x$ as well.

Indeed, if $x\in E_{N}^{c}$ then there is $h\in B_{N,2\rho_{N}}$
such that $xh\not\in[x]_{\mathscr{P}_{N}}$. In particular, it follows
that there are $|n|\leq N$ and $P_{1}\not=P_{2}\in\mathscr{P}$ such
that $x\in T_{a}^{n}P_{1}$ and $xh\in T_{a}^{n}P_{2}$. As $h\in B_{N,2\rho_{N}}$,
it can be written by $h=a^{-n}\tilde{h}a^{n}$ for $\tilde{h}\in B_{2\rho_{N}}^{G}$. We obtain that $xa^{-n}$ and $xha^{-n}=xa^{-n}\tilde{h}$ belong
to different sets in $\mathscr{P}$, and so $xa^{-n}$ must be $B_{2\rho_{N}}^{G}$-close
to the boundary of $S_j$ for some $1\leq j\leq l$, that
is $xa^{-n}\in(\partial S_{j})B_{2\rho_{N}}^{G}$. Therefore, we obtained
\begin{equation*}
E_{N}^{c}\subset\underset{n=-N}{\overset{N}{\bigcup}}\underset{j=1}{\overset{l}{\bigcup}}T_{a}^{n}\big((\partial S_j)B_{2\rho_{N}}^{G}\big).
\end{equation*}
Since
\begin{equation*}
\mu\Big(\underset{n=-N}{\overset{n}{\bigcup}}\underset{j=1}{\overset{l}{\bigcup}}T_{a}^{n}((\partial S_j)B_{2\rho_{N}}^{G})\Big)\leq(2N+1)\cdot l\cdot2E\rho_{N}
\end{equation*}
the estimate on $\mu(E_N)$ follows.\\

We note that
\begin{equation*}
\underset{N\to\infty}{\liminf}\frac{\log(\BC_{\rho}(N,\epsilon_{0}))}{2N}=\underset{N\to\infty}{\liminf}\frac{\log(\BC_{\rho_{N}}(N,\epsilon_{0}))}{2N},
\end{equation*}
because as in Remark~\ref{remark_about_BC}
\begin{equation*}
\BC_{\rho}(N,\epsilon_{0})\leq\BC_{\rho_{N}}(N,\epsilon_{0})\ll N^{(\mathtt{d}+1)\mathtt{d}}\BC_{\rho}(N,\epsilon_{0}).
\end{equation*}
Take some
\begin{equation*}
f>\underset{\epsilon_{0}\to0^{+}}{\liminf}\underset{N\to\infty}{\liminf}\frac{\log(\BC_{\rho}(N,\epsilon_{0}))}{2N}.
\end{equation*}
Then there is a sequence $\epsilon_{i}\to0^{+}$ and a sequence $N_{i}\to\infty$
(dependent on $\epsilon_{i}$) such that $f>\frac{\log(\BC_{\rho_{N_{i}}}(N_{i},\epsilon_{i}))}{2N_{i}}$
and $N_{i}>\frac{D}{\epsilon_{i}}$. Then there is a subset $X_{i}\subset \mathcal{F}X$ with $\mu(X_{i})>1-\epsilon_{i}$
which can be covered by $\lfloor e^{2N_{i}f}\rfloor$ Bowen $(N_{i},\rho_{N_{i}})$
balls $$\{y_{j}^{i}B_{N_{i},\rho_{N_{i}}}\}_{1\leq j\leq\lfloor e^{2N_{i}f}\rfloor}.$$
Moreover, $\mu(E_{N_{i}})\geq1-\epsilon_{i}$.

Set $Y_{i}=X_{i}\cap E_{N_{i}}$, which clearly satisfies $\mu(Y_{i})\geq1-2\epsilon_{i}$. Take $x\in Y_{i}$. As $x\in X_{i}$, there is $1\leq j\leq\lfloor e^{2N_{i}f}\rfloor$
such that $x\in y_{j}^{i}B_{N_{i},\rho_{N_{i}}}$ and so
\begin{equation*}
y_{j}^{i}B_{N_{i},\rho_{N_{i}}}\subset xB_{N_{i},2\rho_{N_{i}}}.
\end{equation*}
As $x\in E_{N_{i}}$ we obtain
\begin{equation*}
y_{j}^{i}B_{N_{i},\rho_{N_{i}}}\subset xB_{N_{i},2\rho_{N_{i}}}\subset[x]_{\mathscr{P}_{N_{i}}}=[y_{j}^{i}]_{\mathscr{P}_{N_{i}}}.
\end{equation*}
Therefore, $Y_{i}$ can be covered by $\lfloor e^{2N_{i}f}\rfloor$
atoms of $\mathscr{P}_{N_{i}}$. Let $Z_{i}$ be the union of those
$\lfloor e^{2N_{i}f}\rfloor$ atoms. In particular, $\mu_{Z_{i}}(A)\log\mu_{Z_{i}}(A)\not=0$ only for
$\lfloor e^{2N_{i}f}\rfloor$ sets $A\in\mathscr{P}_{N_{i}}$, and
so $H_{\mu_{Z_{i}}}(\mathscr{P}_{N_{i}})\leq2N_{i}f$. 

Set $\mathscr{Q}=\{Z_{i},Z_{i}^{c}\}$. As $\mathscr{P}_{N_{i}}$ is finer than $\mathscr{Q}$,
\begin{align*}
H_{\mu}(\mathscr{P}_{N_{i}})&=H_{\mu}(\mathscr{Q})+H_{\mu}(\mathscr{P}_{N_{i}}|\mathscr{Q})=H_{\mu}(\mathscr{Q})+\mu(Z_{i})H_{\mu_{Z_{i}}}(\mathscr{P}_{N_{i}})+\mu(Z_{i}^{c})H_{\mu_{Z_{i}^{c}}}(\mathscr{P}_{N_{i}})\\&\leq\log2+1\cdot2N_{i}f+2\epsilon_{i}\cdot\log \big((l+1)^{2N_{i}+1}\big)\\&
=\log2+2N_{i}f+2\epsilon_{i}\log (l+1)\cdot (2N_{i}+1).
\end{align*}

We get
\begin{equation*}
h_{\mu}(T_{a})-\upsilon<h_{\mu}(T_{a},\mathscr{P})=\underset{N\to\infty}{\lim}\frac{H_{\mu}(\mathscr{P}_{N})}{2N+1}=\underset{i\to\infty}{\lim}\frac{H_{\mu}(\mathscr{P}_{N_{i}})}{2N_{i}+1}\leq f.
\end{equation*}

As $\upsilon>0$ and
\begin{equation*}
f>\underset{\epsilon_{0}\to0^{+}}{\liminf}\underset{N\to\infty}{\liminf}\frac{\log(\BC_{\rho}(N,\epsilon_{0}))}{2N}
\end{equation*}
are arbitrary, we obtain
\begin{equation*}
h_{\mu}(T_{a})\leq\underset{\epsilon_{0}\to0^{+}}{\liminf}\underset{N\to\infty}{\liminf}\frac{\log(\BC_{\rho}(N,\epsilon_{0}))}{2N}
\end{equation*}
as desired. \end{proof}

\begin{proof}[Proof of Theorem \ref{theorem2}]
First of all, let us assume $\mu$ is ergodic.

Take $0<\epsilon^{\prime}<\epsilon_{\mathtt{d}}$ small enough so
that every bi-infinite $T_{a}$-orbit of any point in $\Omega_{\mathcal{F}}$ meets
$\Omega_{\mathrm{nc}}^{\epsilon^{\prime}}$, that is $\Omega_{\mathcal{F}}\subset\underset{k=-\infty}{\overset{\infty}{\bigcup}}T_{a}^{-k}(\Omega_{\mathrm{nc}}^{\epsilon^{\prime}})$. Take $0<\epsilon<\epsilon^{\prime}$ small enough so that $T_{a}^{-1}(\Omega_{\mathrm{nc}}^{\epsilon^{\prime}})\subset\Omega_{\mathrm{nc}}^{\epsilon}$. Moreover, assume $\epsilon$ is small enough for Lemma \ref{main_lemma} to hold.

Note that $\mu(\Omega_{\mathrm{nc}}^{\epsilon^{\prime}})>0$, because
$$1=\mu(\Omega_{\mathcal{F}})\leq\mu(\underset{k=-\infty}{\overset{\infty}{\bigcup}}T_{a}^{-k}(\Omega_{\mathrm{nc}}^{\epsilon^{\prime}}))\leq\underset{k=-\infty}{\overset{\infty}{\sum}}\mu(T_{a}^{-k}(\Omega_{\mathrm{nc}}^{\epsilon^{\prime}}))=\underset{k=-\infty}{\overset{\infty}{\sum}}\mu(\Omega_{\mathrm{nc}}^{\epsilon^{\prime}}).$$

Note that from Birkhoff's Pointwise Ergodic Theorem, almost every $x\in \mathcal{F}X$
satisfies
\begin{equation*}
\underset{N\to\infty}{\lim}\frac{1}{N}\overset{N-1}{\underset{k=0}{\sum}}1_{\Omega_{\mathrm{nc}}^{\epsilon^{\prime}}}(T_{a}^{k}x)=\mu(\Omega_{\mathrm{nc}}^{\epsilon^{\prime}})>0
\end{equation*}
and in particular $x\in\underset{k=0}{\overset{\infty}{\bigcup}}T_{a}^{-k}(\Omega_{\mathrm{nc}}^{\epsilon^{\prime}})
$, as otherwise this limit would be zero. So $\mu(\underset{k=0}{\overset{\infty}{\bigcup}}T_{a}^{-k}(\Omega_{\mathrm{nc}}^{\epsilon^{\prime}}))=1$.
Therefore, for any $\epsilon_0>0$ we can fix $K\geq0$ such that the set
\begin{equation*}
Y=\underset{k=0}{\overset{K}{\bigcup}}T_{a}^{-k}(\Omega_{\mathrm{nc}}^{\epsilon^{\prime}})
\end{equation*}
satisfies $\mu(Y)>1-\frac{\epsilon_{0}}{3}$.

Moreover, again by Birkhoff's Pointwise Ergodic Theorem, for all $1\leq i\leq\mathtt{d}-1$
\begin{equation*}
\underset{N\to\infty}{\lim}\frac{1}{2N+1}\overset{N}{\underset{n=-N}{\sum}}1_{\Omega_{\mathrm{c},i}^{\epsilon}}(T_{a}^{n}x)=\mu(\Omega_{\mathrm{c},i}^{\epsilon})
\end{equation*}
for almost every $x\in X$, i.e\ the sequence of functions
\begin{equation*}
\frac{1}{2N+1}\overset{N}{\underset{n=-N}{\sum}}1_{T_{a}^{-n}(\Omega_{\mathrm{c},i}^{\epsilon})}
\end{equation*}
converges to $\mu(\Omega_{\mathrm{c},i}^{\epsilon})$ almost-everywhere.
By Egorov's theorem, for all $1\leq i\leq\mathtt{d}-1$ there is a
subset $X_{i}$ with
\begin{equation*}
\mu(X_{i})>1-\frac{\epsilon_{0}}{3(\mathtt{d}-1)}
\end{equation*}
on which
\begin{equation*}
\frac{1}{2N+1}\overset{N}{\underset{n=-N}{\sum}}1_{T_{a}^{-n}(\Omega_{\mathrm{c},i}^{\epsilon})}
\end{equation*}
converges to $\mu(\Omega_{\mathrm{c},i}^{\epsilon})$ uniformly, and
so there is $N_{i}\in\mathbb{N}$ such that for all $N>N_{i}$ and $x\in X_{i}$ 
\begin{equation*}
\frac{1}{2N+1}\overset{N}{\underset{n=-N}{\sum}}1_{\Omega_{\mathrm{c},i}^{\epsilon}}(T_{a}^{n}x)>\kappa_{i}\coloneqq\mu(\Omega_{\mathrm{c},i}^{\epsilon})-\epsilon_{0}
\end{equation*}
holds. 

Set $N_{0}=\underset{1\leq i\leq\mathtt{d}-1}{\max}N_{i}$ and fix
$N>N_{0}$. Set $E=\underset{i=1}{\overset{\mathtt{d}-1}{\bigcap}}X_{i}$, and
\begin{equation*}
F=E\cap T_{a}^{N+K}Y\cap T_{a}^{-(N+K)}Y.
\end{equation*}
Note that
\begin{equation*}
F=\underset{(k_{1},k_{2})\in\{0,\ldots,K\}^{2}}{\bigcup}E\cap T_{a}^{N+K-k_{1}}(\Omega_{\mathrm{nc}}^{\epsilon^{\prime}})\cap T_{a}^{-(N+K)-k_{2}}(\Omega_{\mathrm{nc}}^{\epsilon^{\prime}}).
\end{equation*}
As in the proof of Lemma \ref{lemma_1}, we can re-shift each of the sets
\begin{equation*}
F_{k_{1},k_{2}}=E\cap T_{a}^{N+K-k_{1}}(\Omega_{\mathrm{nc}}^{\epsilon^{\prime}})\cap T_{a}^{-(N+K)-k_{2}}(\Omega_{\mathrm{nc}}^{\epsilon^{\prime}})
\end{equation*}
and use Lemma \ref{number of Z} and Lemma \ref{main_lemma} to cover it by Bowen balls, as follows. For any $(k_{1},k_{2})$, set $c=\lfloor\frac{k_{1}+k_{2}}{2}\rfloor,\:d=\lfloor\frac{k_{2}-k_{1}}{2}\rfloor$ and $N^{\prime}=N+K+d$. Then
\begin{equation*}
T_{a}^{c}F_{k_{1},k_{2}}\subset T_{a}^{N^{\prime}}(\Omega_{\mathrm{nc}}^{\epsilon^{\prime}})\cap T_{a}^{-N^{\prime}}(\Omega_{\mathrm{nc}}^{\epsilon^{\prime}})
\end{equation*}
or
\begin{equation*}
T_{a}^{c}F_{k_{1},k_{2}}\subset T_{a}^{N^{\prime}}(\Omega_{\mathrm{nc}}^{\epsilon^{\prime}})\cap T_{a}^{-(N^{\prime}+1)}(\Omega_{\mathrm{nc}}^{\epsilon^{\prime}}),
\end{equation*}
depending on the parity of $k_{1},k_{2}$, but in any case
\begin{equation*}
T_{a}^{c}F_{k_{1},k_{2}}\subset T_{a}^{N^{\prime}}(\Omega_{\mathrm{nc}}^{\epsilon})\cap T_{a}^{-N^{\prime}}(\Omega_{\mathrm{nc}}^{\epsilon}).
\end{equation*}

Moreover, for all $x=T_{a}^{c}y\in T_{a}^{c}F_{k_{1},k_{2}}$,
\begin{align*}
\underset{n=-N^{\prime}}{\overset{N^{\prime}}{\sum}}1_{\Omega_{\mathrm{c},i}^{\epsilon}}(T_{a}^{n}x)&=\underset{n=-N^{\prime}}{\overset{N^{\prime}}{\sum}}1_{\Omega_{\mathrm{c},i}^{\epsilon}}(T_{a}^{n+c}y)=\underset{n=-(N+K)+(c-d)}{\overset{(N+K)+(c+d)}{\sum}}1_{\Omega_{\mathrm{c},i}^{\epsilon}}(T_{a}^{n}y)\\&\geq(2N+1)\frac{1}{2N+1}(\underset{n=-N}{\overset{N}{\sum}}1_{\Omega_{\mathrm{c},i}^{\epsilon}}(T_{a}^{n}y)-|c-d|-|c+d|)\\&\geq(2N+1)(\kappa_{i}-\frac{2K}{2N+1}).
\end{align*}

We have obtained
\begin{align*}
T_{a}^{c}F_{k_{1},k_{2}}\subset\Big\{&x\in T_{a}^{N^{\prime}}(\Omega_{\mathrm{nc}}^{\epsilon})\cap T_{a}^{-N^{\prime}}(\Omega_{\mathrm{nc}}^{\epsilon}):\\&\:\forall1\leq i\leq\mathtt{d}-1,\:\left|\{n\in[-N^{\prime},N^{\prime}]\cap\mathbb{Z}:\:T_{a}^{n}x\in \Omega_{\mathrm{c},i}^{\epsilon}\}\right|\geq(2N+1)\kappa_{i}-2K\Big\}.
\end{align*}

By Lemma \ref{number of Z}, $T_{a}^{c}F_{k_{1},k_{2}}$ is the union of at most $\ll_{\epsilon}e^{\frac{3\log(|\log\epsilon|)}{|\log\epsilon|}N^{\prime}}$
sets of the form $T_{a}^{c}F_{k_{1},k_{2}}\cap Z(V,\epsilon)$. For $T_{a}^{c}F_{k_{1},k_{2}}\cap Z(V,\epsilon)$ to be non-empty, necessarily
\begin{equation*}
|V^{-1}(i)|\geq(2N+1)\kappa_{i}-2K
\end{equation*}
for all $1\leq i\leq\mathtt{d}-1$. Therefore, by Lemma \ref{main_lemma}, $T_{a}^{c}F_{k_{1},k_{2}}\cap Z(V,\epsilon)$
can be covered by
\begin{equation*}
\ll_{\epsilon} C^{\frac{4N^{\prime}}{|\log\epsilon|}}\cdot e^{(2N^{\prime}+1)\delta-\underset{i=1}{\overset{\mathtt{d}-1}{\sum}}\frac{2\delta-i}{2}\cdot((2N+1)\kappa_{i}-2K)}
\end{equation*}
Bowen $(N^{\prime},\eta)$-balls.
Therefore, $T_{a}^{c}F_{k_1,k_2}$ can be covered by
\begin{align*}
&\ll_{\epsilon}e^{\frac{3\log(|\log\epsilon|)}{|\log\epsilon|}N^{\prime}}\cdot C^{\frac{4N^{\prime}}{|\log\epsilon|}}\cdot e^{(2N^{\prime}+1)\delta-\underset{i=1}{\overset{\mathtt{d}-1}{\sum}}\frac{2\delta-i}{2}\cdot((2N+1)\kappa_{i}-2K)}\\
&\quad\leq e^{\frac{4\log(|\log\epsilon|)}{|\log\epsilon|}N^{\prime}}\cdot e^{(2N^{\prime}+1)\delta-\underset{i=1}{\overset{\mathtt{d}-1}{\sum}}\frac{2\delta-i}{2}\cdot((2N+1)\kappa_{i}-2K)}
\end{align*}
Bowen $(N^{\prime},\eta)$-balls.
Therefore $F_{k_{1},k_{2}}$ can be covered by that many $T_{a}^{-c}$
images of Bowen $(N^{\prime},\eta)$-balls. Each such image is contained
 in a Bowen $(N^{\prime}-c,\eta)$ ball.
As
\begin{equation*}
N^{\prime}-c=N+K+d-c\geq N,
\end{equation*}
each such ball is in fact contained in a Bowen $(N,\eta)$-ball.

Note that $F$ is the union of $(K+1)^{2}$ such sets $F_{k_1,k_2}$ and so it can be covered by the same number of balls, up to multiplicative constant.
Since
\begin{equation*}
\mu(F)\geq1-(\mathtt{d}-1)\frac{\epsilon_{0}}{3(\mathtt{d}-1)}-\frac{\epsilon_{0}}{3}-\frac{\epsilon_{0}}{3}=1-\epsilon_{0},
\end{equation*}
we get an upper bound on $\BC_{\eta}(N,\epsilon_{0})$.
Keeping in mind that $\underset{N\to\infty}{\lim}\frac{N^{\prime}}{N}=1$, we get
\begin{align*}
\underset{N\to\infty}{\liminf}\frac{\log\BC_{\eta}(N,\epsilon_{0})}{2N+1}
&\leq\underset{N\to\infty}{\lim}\frac{1}{2N+1}\Big(\frac{4\log(|\log\epsilon|)}{|\log\epsilon|}N^{\prime}+(2N^{\prime}+1)\delta\\
&\hspace{2.5cm}-\underset{i=1}{\overset{\mathtt{d}-1}{\sum}}\frac{2\delta-i}{2}\cdot((2N+1)\kappa_{i}-2K)\Big)\\
&=\frac{2\log(|\log\epsilon|)}{|\log\epsilon|}+\delta-\underset{i=1}{\overset{\mathtt{d}-1}{\sum}}\frac{2\delta-i}{2}\kappa_{i}\\
&=\frac{2\log(|\log\epsilon|)}{|\log\epsilon|}+\delta-\underset{i=1}{\overset{\mathtt{d}-1}{\sum}}\frac{2\delta-i}{2}\mu(\mathcal{F}\cusp_{\epsilon}^{i}(X))+O(\epsilon_{0}).
\end{align*}
Therefore, by Lemma \ref{entropy_bound_by_Bowen_balls},
\begin{equation*}
h_{\mu}(T_{a})\leq\underset{\epsilon_{0}\to0^{+}}{\liminf}\underset{N\to\infty}{\liminf}\frac{\log\BC_{\eta}(N,\epsilon_{0})}{2N+1}\leq\delta-\underset{i=1}{\overset{\mathtt{d}-1}{\sum}}\frac{2\delta-i}{2}\mu(\mathcal{F}\cusp_{\epsilon}^{i}(X))+\frac{2\log(|\log\epsilon|)}{|\log\epsilon|},
\end{equation*}
as desired.\\

In order to deduce the theorem for non-ergodic $\mu$, we can decompose $\mu$ to its ergodic
components $\mu=\int\mu_{t}d\tau(t)$ over some measure space. As
\begin{equation*}
\mu(\mathcal{F}\cusp_{\epsilon}^{i}(X))=\int\mu_{t}(\mathcal{F}\cusp_{\epsilon}^{i}(X))d\tau(t)
\end{equation*}
and
\begin{equation*}
h_{\mu}(T_{a})=\int h_{\mu_{t}}(T_{a})d\tau(t),
\end{equation*}
the result follows immediately from the ergodic case.\end{proof}

\section{Proofs of Theorems \ref{criterion for closed geodesics}-\ref{theorem3}}\label{proof of criterion Section}
In this section we consider periodic orbits of the geodesic flow on the unit tangent bundle of hyperbolic orbifolds, i.e.\ periodic $a_\bullet$-orbits on $\Gamma\backslash G/M$. However, the natural quotient to consider from a homogeneous dynamics point of view is $\Gamma\backslash G$, so for the proof of Theorem \ref{criterion for closed geodesics} we will instead work with periodic $Ma_\bullet$-orbits on $\Gamma\backslash G$. For $x\in\Gamma\backslash G$, we say that $xMA$ is a periodic $Ma_\bullet$-orbit of length $t>0$ if $\Gamma g ma_t=\Gamma g$ for some $m\in M$. With this definition, we can identify between periodic $a_\bullet$-orbits on $\Gamma\backslash G/M$ and periodic $Ma_\bullet$-orbits on $\Gamma\backslash G$.

The following proposition shows that nearby periodic orbits of similar length and orientation, are identical.
\begin{proposition}\label{distance_between_geodesics}
Let $\tau>0$ be small enough. Assume $x_{1},x_{2}\in\Omega_{\mathrm{nc}}^{\tau}$
belong to periodic $Ma_\bullet$-orbits of lengths $t_{1},t_{2}$ respectively, i.e.\ $x_{i}m_ia_{t_{i}}=x_{i}$ for $i=1,2$. 
Moreover, assume that $|t_{1}-t_{2}|<\frac{1}{6}\tau$, $d(m_1,m_2)<\frac{1}{6}\tau$ and $x_{2}=x_{1}u$ for $u\in B_{N,\frac{1}{6}\tau}^{+}$ or for $u\in B_{N,\frac{1}{6}\tau}^{-}$, where $N\geq \max\{\lceil t_1\rceil,\lceil t_2\rceil\}$.
Then $x_{1}$ and $x_{2}$ are part of the same periodic $Ma_\bullet$-orbit, i.e.\ $x_{1}MA=x_{2}MA$.
\end{proposition}

\begin{proof}
Let $\tau\leq l_0$, for $l_0$ the constant from Lemma \ref{deal_with_elliptic}. Moreover, assume $\tau$ is small enough with respect to the injectivity radius of $\exp:\mathfrak{so}(1,\mathtt{d})\to G$ near $0\in\mathfrak{so}(1,\mathtt{d})$, as will be described momentarily.

Let $x_{1}=\Gamma g$ for $g\in G$, and $x_{2}=x_{1}u$ for $u\in B_{N,\frac{1}{6}\tau}^{+}$. 
Since $\Gamma ga_{t_{1}}m_1=\Gamma g$, there is $\gamma_{1}\in\Gamma$
such that $\gamma_{1}g=ga_{t_{1}}m_1$. Likewise, there is $\gamma_{2}\in\Gamma$
such that $\gamma_{2}gu=gua_{t_{2}}m_2$. 

Let $u=n_{+}man_{-}$ be a decomposition of $u$ in $N^+MAN^-$. Then
\begin{equation}\label{decomposition_of_normalization}
    m_1^{-1}a_{-t_{1}}ua_{t_{2}}m_2=
    \big(m_1^{-1}(a_{-t_{1}}n_{+}a_{t_{1}})m_1\big)\big(m_1^{-1}(maa_{t_{2}-t_{1}})m_2\big)\big(m_2^{-1}(a_{-t_{2}}n_{-}a_{t_{2}})m_2\big).
\end{equation}
The conjugation by $a_{t_i}$ shrinks the $N^{-}$ part by $e^{t_i}$ and
enlarges the $N^{+}$ part by the
same factor. Conjugation of $N^+$ and $N^{-}$ by $M$ does not change the size. 
Since $u\in B_{\frac{1}{6}\tau e^{-N}}^{N^{+}}B_{\frac{1}{6}\tau}^{MA}B_{\frac{1}{6}\tau}^{N^{-}}$, it then follows that
\begin{align*}
    d(\gamma_{1}^{-1}\gamma_{2}gu,gu)&=d(m_1^{-1}a_{-t_{1}}ua_{t_{2}}m_2,u)\leq d(m_1^{-1}a_{-t_{1}}ua_{t_{2}}m_2,e)+d(e,u)\\
    &\leq(3\cdot\frac{1}{6}\tau+|t_1-t_2|+d(m_1,m_2))+\frac{1}{6}\tau<\tau.
\end{align*}

Since $\tau\leq l_0$, it follows from Lemma \ref{deal_with_elliptic} that if $\gamma_1^{-1}\gamma_2\not=e$ then it is not elliptic. By Remark~\ref{loxodromic_remark}, $\gamma_1^{-1}\gamma_2$ is neither loxodromic. Moreover, it cannot be parabolic, since it moves $gu$ by less than $\tau$ yet $x_2=\Gamma gu\in \Omega_{\mathrm{nc}}^{\tau}$. Therefore, $\gamma_{1}^{-1}\gamma_{2}$ must be the identity.

In particular, we get that $u=m_1a_{t_{1}}ua_{-t_{2}}m_2^{-1}$. 
Then both Equation \eqref{decomposition_of_normalization} and $u=n_+amn_-$ decompose $u$ as $u=\exp(X_1)=\exp(X_2)$ for $X_1,X_2\in\mathfrak{so}(1,\mathtt{d})$ close to $0$.
Assuming that $\tau$ was chosen to be small enough (in a way which does not depend on $u$) so that the exponential map is injective in a neighborhood of $0\in\mathfrak{so}(1,\mathtt{d})$ which contains both $X_1$ and $X_2$, it follows that $X_1=X_2$ and so the two decompositions are identical. In particular, $n_+=n_-=e$.
Therefore, $u\in MA$ and so $x_{1}$ and $x_{2}$ are in the same $Ma_\bullet$-orbit, as desired.\\

The same proof works for the case $u\in B_{N,\frac{1}{6}\tau}^{-}$, by replacing $t_i$ with $-t_i$ and $m_i$ with $m_i^{-1}$.
\end{proof}

The following is a trivial corollary of Proposition \ref{distance_between_geodesics}.
\begin{corollary}\label{geodesics_in_bowen} Let $0<\tau<\tilde{\epsilon}_{\mathtt{d}}$, $\rho>0$ and $N\in\mathbb{N}$. Then a forward or backward Bowen $(N,\rho)$-ball
contained in $\Omega_{\mathrm{nc}}^{\tau}$
intersects at most $\ll_{\rho} \tau^{-(\mathtt{d}^2-\mathtt{d}+2)}$ many periodic $Ma_{\bullet}$-orbits
of lengths in $(N-1,N]$.
\end{corollary}
\begin{proof}
Clearly, it suffices to prove the corollary for $\tau$ small enough with respect to Proposition \ref{distance_between_geodesics}. The proof is identical for forward and backward balls.

First, note that as $M\cong\SO(\mathtt{d}-1)$ is a $\frac{(\mathtt{d}-1)(\mathtt{d}-2)}{2}$-dimensional compact space, it has a $\frac{\tau}{12}$-dense subset of size $f\ll\tau^{-\frac{(\mathtt{d}-1)(\mathtt{d}-2)}{2}}$. 

We split the forward Bowen $(N,\rho)$-ball into
$$\ll\lceil\frac{\rho}{\frac{1}{12}\tau}\rceil^{\frac{(\mathtt{d}+1)\mathtt{d}}{2}}\ll_{\rho} {\tau}^{-\frac{(\mathtt{d}+1)\mathtt{d}}{2}}$$ 
many forward Bowen $(N,\frac{1}{12}\tau)$-balls.
Note that each such ball is contained in $\Omega_{\mathrm{nc}}^{\tau}$. 
By Proposition
\ref{distance_between_geodesics}, for any sub-interval $I\subset(N-1,N]$ of length $|I|<\frac{1}{6}\tau$, each such forward Bowen ball intersects at most $f$ many periodic $Ma_\bullet$-orbits whose lengths are in $I$. Therefore
it intersects at most $f\lceil\frac{1}{\frac{1}{6}\tau}\rceil\ll f{\tau}^{-1}$
many periodic $Ma_\bullet$-orbits of lengths in $(N-1,N]$.

Altogether, we obtain at most $$\ll_{\rho} \tau^{-\frac{(\mathtt{d}+1)\mathtt{d}}{2}}\tau^{-\frac{(\mathtt{d}-1)(\mathtt{d}-2)}{2}}\tau^{-1}=\tau^{-(\mathtt{d}^2-\mathtt{d}+2)}$$ periodic $Ma_\bullet$-orbits of lengths in $(N-1,N]$.\end{proof}

The following proposition is a major step of the proof of Theorem \ref{criterion for closed geodesics}.
\begin{proposition}\label{counting closed geodesics}
Let $\Gamma<G$ be a non-elementary geometrically finite subgroup. Let $\epsilon>0$ be small enough. Then for any $\beta\in[0,1]$, the number of periodic $Ma_\bullet$-orbits of lengths at most $T$, which spend in $\Omega_{\mathrm{c}}^{\epsilon}$ at least a fraction $\beta$ of their time, is bounded from above by 
\begin{equation*}
    \ll\lceil T\rceil|\log\epsilon|^{3}e^{(\delta-\frac{2\delta-r_{\max}}{2}\beta+\frac{2\log C}{|\log\epsilon|}+\frac{3\log|\log\epsilon|}{|\log\epsilon|})\lceil T\rceil}
\end{equation*}
\end{proposition}
\begin{proof}
Let $N\leq \lceil T\rceil$ be a natural number. Let $\epsilon_{c}$ be small enough with respect to Lemma \ref{main_lemma}.

By Lemma \ref{size_of_cusps}, every closed geodesic must intersect the compact part of the orbifold. It follows that for each periodic $Ma_\bullet$-orbit of length in $(N-1,N]$ which spends in $\Omega_{\mathrm{c}}^{\epsilon}$ at least a fraction $\beta$ of its time, there are a function $$V:[0,N]\to\{0,\ldots,\mathtt{d}-1\}$$ 
which satisfies
$$|V^{-1}(\{1,\ldots,\mathtt{d}-1\})|\geq\beta (N-1)$$ 
and a point $x$ of the orbit, such that $x\in Z_{+}(V,\epsilon,\epsilon_{c})$. 

By Lemma \ref{main_lemma},
$Z_+(V,\epsilon,\epsilon_{c})$ can be covered by
\begin{equation*}
\ll C^{\frac{2N}{|\log\epsilon|}}\cdot e^{\delta N-\frac{2\delta-r_{\max}}{2}|V^{-1}(\{1,\ldots,\mathtt{d}-1\})|}\ll e^{(\delta-\frac{2\delta-r_{\max}}{2}\beta+\frac{2\log C}{|\log\epsilon|})N}
\end{equation*}
forward Bowen $(N,\eta)$-balls for $\eta\ll 1$. Note that each of these balls is contained in $\Omega_{\mathrm{nc}}^{\epsilon_{c}}$, and so by Corollary \ref{geodesics_in_bowen} each ball intersects at most $\ll_{\eta} 1$ many periodic $Ma_\bullet$-orbits of lengths in $(N-1,N]$.

Moreover, due to Lemma \ref{number of Z}, there are only
\begin{equation*}
    \leq |\log\epsilon|^{3}e^{\frac{3\log|\log\epsilon|}{|\log\epsilon|}N}
\end{equation*}
many functions $V$ for which $Z_{+}(V,\epsilon,\epsilon_{c})\not=\emptyset$.

Altogether, the number of periodic $Ma_\bullet$-orbits of lengths in $(N-1,N]$ which spend in $\Omega_{\mathrm{c}}^{\epsilon}$ at least a fraction $\beta$ of their time is at most
\begin{equation*}
    \ll_{\eta}f_{N}\coloneqq  |\log\epsilon|^{3}e^{\frac{3\log|\log\epsilon|}{|\log\epsilon|}N}e^{(\delta-\frac{2\delta-r_{\max}}{2}\beta+\frac{2\log C}{|\log\epsilon|})N}.
\end{equation*}
Therefore, the number of periodic $Ma_\bullet$-orbits of lengths at most $T$ which spend in $\Omega_{\mathrm{c}}^{\epsilon}$ at least a fraction $\beta$ of their time
is bounded from above by
\begin{equation*}
    \ll_{\eta} \underset{N=1}{\overset{\lceil T \rceil}{\sum}}f_{N}\leq \lceil T \rceil f_{\lceil T \rceil}=\lceil T\rceil|\log\epsilon|^{3}e^{(\delta-\frac{2\delta-r_{\max}}{2}\beta+\frac{2\log C}{|\log\epsilon|}+\frac{3\log|\log\epsilon|}{|\log\epsilon|})\lceil T\rceil}
\end{equation*}\end{proof}
For a periodic $Ma_\bullet$-orbit $l$ on $\Gamma\backslash G$, let $\tilde{\mu}_l$ be the natural $MA$-invariant probability measure on $l$, i.e.\ the product measure of the Haar measure on the $M$-part and the normalized Lebesgue measure on $[0,T)$ on the $A$-part, where $T$ is the length of the periodic orbit. The pushforward measure $(\pi_{M})_{\ast}\tilde{\mu}_l=\mu_{l}$ is the natural $A$-invariant probability measure on the corresponding periodic $a_\bullet$-orbit on $\Gamma\backslash G/M$, i.e.\ the measure which is supported on the orbit and distributes the mass uniformly on it.
For a finite set $\varphi$ of periodic $Ma_\bullet$-orbits, the natural $MA$-invariant probability measure averaging on $\varphi$ is defined by $\tilde{\mu}_{\varphi}=\frac{1}{|\varphi|}\sum_{l\in\varphi}\tilde{\mu}_l$.

\begin{proof}[Proof of Theorem \ref{criterion for closed geodesics}]
Let $\tilde{\mu}_i$ be the natural $MA$-invariant probability measure on $\psi(T_i)$. We will show that $(\tilde{\mu}_{i})_{i\in\mathbb{N}}$ satisfies the conditions of Theorem \ref{theorem1}.

Set $N_{i}=\lceil T_{i}\rceil$ and $\lambda_{i}=e^{-N_{i}}$.
Let $\epsilon_{0}>0$ be arbitrary and set $H_{i}=\lambda_{i}^{\epsilon_{0}}$.

First, we show that $\tilde{\mu}_{i}(\Omega_{\mathrm{c}}^{H_{i}})\to0$.
Since all closed geodesics intersect the compact part of the orbifold, it follows from Lemma \ref{size_of_cusps} and Remark \ref{log_distance_remark} that periodic $Ma_\bullet$-orbits which intersect $\Omega_{\mathrm{c}}^{H_i}$ spend at least $\frac{\epsilon_0 N_i}{3}$ time in $\Omega_{\mathrm{c}}^{\sqrt{H_i}}$, assuming $i$ is large enough. 
In particular, if their lengths are at most $T_i$, they spend in $\Omega_{\mathrm{c}}^{\sqrt{H_i}}$ at least a fraction $\frac{\epsilon_0}{3}$ of their time. Therefore, by Proposition~\ref{counting closed geodesics}, the number $f_i$ of periodic $Ma_\bullet$-orbits of lengths at most $T_i$ which intersect $\Omega_{\mathrm{c}}^{H_i}$ satisfies
\begin{align*}
    f_i&\ll\lceil T_i\rceil|\log\sqrt{H_i}|^{3}e^{(\delta-\frac{2\delta-r_{\max}}{6}\epsilon_0+\frac{2\log C}{|\log\sqrt{H_i}|}+\frac{3\log|\log\sqrt{H_i}|}{|\log\sqrt{H_i}|})\lceil T_i\rceil}
    \\&
    \ll_{\epsilon_0} N_i^{4+\frac{6}{\epsilon_0}}e^{(\delta-\frac{2\delta-r_{\max}}{6}\epsilon_0)N_i}.
\end{align*}
Therefore, the measure of the cusp can be bounded by
$$\tilde{\mu}_{i}(\Omega_{\mathrm{c}}^{H_{i}})=\frac{1}{|\psi(T_i)|}\underset{l\in \psi(T_i)}{\sum}\tilde{\mu}_{l}(\Omega_{\mathrm{c}}^{H_{i}})\leq\frac{1}{|\psi(T_i)|}\cdot f_{i}\cdot1\leq e^{-(\delta-\alpha_{i})T_{i}}f_{i},$$
where the RHS clearly converges to $0$ as $i\to\infty$. Together with the fact that $\tilde{\mu}_i$ is $A$-invariant and so supported on $\Omega_{\mathcal{F}}$, it follows that assumption \ref{assumption_1_thm_1} of Theorem \ref{theorem1} holds.

Next, we show that assumption \ref{assumption_2_thm_1} of Theorem \ref{theorem1} is satisfied. Let
$$E=\{(x,y)\in \mathcal{F}\core_{H_{i}}(X)\times \mathcal{F}\core_{H_{i}}(X):\:x\in yB_{1}^{N^{+}}B_{1}^{MA}B_{\lambda_{i}}^{N^{-}}\}.$$
We want to bound $\tilde{\mu}_{i}\times\tilde{\mu}_{i}(E)$. Let $(x,y)\in E$. Note that $$E_{y}=\{x_0\in \Omega_{\mathcal{F}}:\:(x_0,y)\in E\}$$ is a
subset of the backward Bowen $(N_{i},\rho)$-Ball $\Gamma yB_{N_{i},\rho}^{-}$ for $\rho\ll 1$,
which is contained in $\overline{\Omega_{\mathrm{nc}}^{H_{i}}}$. By Corollary \ref{geodesics_in_bowen}, we get that $E_y$ intersects 
$$\ll N_iH_i^{-(\mathtt{d}^2-\mathtt{d}+2)}\ll e^{2\mathtt{d}^{2}\epsilon_{0}N_{i}}$$
many periodic $Ma_\bullet$-orbits of lengths at most $N_{i}$, once $i$ is large enough. 
In particular,
$$\tilde{\mu}_{i}(E_{y})\ll\frac{1}{|\psi(T_i)|}e^{2\mathtt{d}^2\epsilon_{0}N_{i}}\leq e^{-(\delta-\alpha_{i})T_{i}}e^{2\mathtt{d}^2\epsilon_{0}N_{i}}\ll\lambda_{i}^{\delta-3\mathtt{d}^2\epsilon_{0}}$$
once $i$ is large enough.
Therefore, by Fubini's Theorem together with the fact that $\tilde{\mu}_i$ is supported on $\Omega_{\mathcal{F}}$, $$\tilde{\mu}_{i}\times\tilde{\mu}_{i}(E)\ll\lambda_{i}^{\delta-3\mathtt{d}^2\epsilon_{0}}$$
as well, as required. 

This shows, by Theorem \ref{theorem1}, that every weak-$\star$ limit of a subsequence of $(\tilde{\mu}_i)_{i\in\mathbb{N}}$ has entropy $\delta$. Due to uniqueness of measure of maximal entropy for the geodesic flow (Theorem \ref{maximal_entropy}), we obtain that $m_{\BM}$ is the unique limit of any subsequence of $(\mu_i)_{i\in\mathbb{N}}$, where $\mu_i=(\pi_M)_{\ast}\tilde{\mu}_i$, and so $(\mu_{i})_{i\in\mathbb{N}}$ converges to $m_{\BM}$. \end{proof}
From this point on, we will discuss only periodic $a_\bullet$-orbits on $\Gamma\backslash G/M$ (rather than periodic $Ma_\bullet$-orbits on $\Gamma\backslash G$).

\begin{proof}[Proof of Theorem \ref{theorem3}]
Assume by contradiction that such $h$ does not exist. Then there is a positive sequence $h_{i}\to0$
and an increasing sequence $T_{i}\to\infty$ such that the sets 
$$A_{i}^{+}=\left\{l\in\Per_{\Gamma}(T_i):\:\int_{l}fd\mu_{l}-\int_{\Gamma\backslash G/M}fdm_{\BM}>\epsilon\right\}$$
are of magnitude $|A_{i}^{+}|>\frac{1}{2}e^{(\delta-h_{i})T_{i}}$.
If this is not the case, then we would simply consider the sets
$$A_{i}^{-}=\left\{l\in\Per_{\Gamma}(T_i):\:\int_{\Gamma\backslash G/M}fdm_{\BM}-\int_{l}fd\mu_{l}>\epsilon\right\}$$
instead of $A_{i}^{+}$, and continue the same way.

Consider the measures $\mu_{i}=\frac{1}{|A_{i}^{+}|}\underset{l\in A_{i}^{+}}{\sum}\mu_{l}$ averaging on $A_{i}^{+}$. By Theorem~\ref{criterion for closed geodesics}, the sequence $(\mu_{i})_{i\in\mathbb{N}}$ converges to $m_{\BM}$.
However,
\begin{align*}
\Big|\int fd\mu_{i}-\int fdm_{\BM}\Big|&=\Big|\frac{1}{|A_{i}^{+}|}\underset{l\in A_{i}^{+}}{\sum}\int fd\mu_{l}-\int fdm_{\BM}\Big|\\&
=\frac{1}{|A_{i}^{+}|}\Big|\underset{l\in A_{i}^{+}}{\sum}[\int fd\mu_{l}-\int fdm_{\BM}]\Big|>\frac{1}{|A_{i}^{+}|}|\underset{l\in A_{i}^{+}}{\sum}\epsilon|=\epsilon
\end{align*}
for all $i$, which is a contradiction.
\end{proof}

\section{Proofs of Theorem \ref{amenable covers} and Corollary \ref{equidistribution in the covering}}\label{Amenable covers Section}
A key tool for the proof is the following theorem of \cite{paulin2015equilibrium}, which holds for the general settings of a complete connected Riemannian manifold with pinched negative curvature and a reversible potential $F$. Here we restrict ourselves to the hyperbolic space and take the potential $F\equiv0$. In our settings the theorem is due to Roblin \cite{roblin2005un} (cf.~also related work by Sharp \cite{sharp2007critical}), and goes back to Brooks \cite{brooks1985the} in less general settings.
\begin{theorem}[{\cite[Theorem 2.2.2]{roblin2005un}},{\cite[Theorem 11.14]{paulin2015equilibrium}}]\label{PPS2}
Let $\Gamma_0<G$ be a discrete subgroup. Let $\Gamma\triangleleft\Gamma_0$ be a normal non-elementary subgroup, such that $\Gamma\backslash\Gamma_0$ is amenable. 
Then $\delta(\Gamma)=\delta(\Gamma_0)$.
\end{theorem}

We will also use the following theorem. In order to obtain this form of the statement, the fact that the critical exponent of a non-elementary subgroup is strictly positive was used.
\begin{theorem}[{\cite[Theorem 4.7]{paulin2015equilibrium}}]\label{PPS1}
Let $\Gamma<G$ be a discrete non-elementary subgroup. Let $W$ be any relatively compact open subset of $T^{1}(\Gamma\backslash\mathbb{H}^{\mathtt{d}})$ meeting the non-wandering set of the geodesic flow $\Omega_{\Gamma}$.
Then 
\begin{equation*}\lim_{T\to\infty}\frac{1}{T}\log(\left|\{p\in \Per_{\Gamma}(T):\:p\cap W\not=\emptyset\}\right|)=\delta(\Gamma)
\end{equation*}
\end{theorem}

In order to prove Theorem $\ref{amenable covers}$ we need to pass the question from counting periodic $a_\bullet$-orbits in $\Gamma\backslash G/M$ to counting them in $\Gamma_0\backslash G/M$. In doing so, for each $p\in\Per_{\Gamma}$ we consider its projection $\pi_{\Gamma_0}(p)$ to $\Gamma_0\backslash G/M$, which is a periodic $a_\bullet$-orbit in $\Gamma_0\backslash G/M$. We need to show that each element of $\Per_{\Gamma_0}$ is not counted too many times.
\begin{lemma}\label{multiplicity in the cover}
Let $\Gamma_0<G$ be discrete, and let $\Gamma\triangleleft\Gamma_0$ be a non-elementary normal subgroup.
Then there exists some relatively compact open set $W\subset T^{1}(\Gamma\backslash \mathbb{H}^{\mathtt{d}})$ intersecting the non-wandering set $\Omega_\Gamma$, with the following property. For all $p_0\in\Per_{\Gamma_0}$, the number of $p\in\Per_{\Gamma}(T)$ intersecting $W$, such that $\pi_{\Gamma_0}(p)=p_0$, is $\ll T$.
\end{lemma}

\begin{proof}
Let $v\in\pi_{\Gamma}^{-1}(\Omega_\Gamma)\subset T^{1}(\mathbb{H}^{\mathtt{d}})$ and $\delta>0$ be such that the projection $\pi_{\Gamma_0}$ is injective on $U_{2\delta}$, where $U_{r}=B_{r}^{T^{1}(\mathbb{H}^{\mathtt{d}})}(v)$ is the ball of radius $r$ around $v$. Let $W=\pi_{\Gamma}(U_{\delta})$.

Let $\Gamma v_1,\Gamma v_2 \in \Gamma \backslash G /M$ be points of some periodic $a_\bullet$-orbits $p_1,p_2\in\Per_{\Gamma}(T)$ respectively. Assume both orbits intersect $W$, and so without loss of generality we may assume $v_1,v_2\in U_{\delta}$. Moreover, assume that $\Gamma_0 v_1$ and $\Gamma_0 v_2$ belong to the same periodic $a_\bullet$-orbit, i.e.\ $\pi_{\Gamma_0}(p_1)=\pi_{\Gamma_0}(p_2)$, and so there is some $t\in\mathbb{R}$ such that $\Gamma_{0}v_1=\Gamma_{0}v_2 a_t$.

We assume that $|t|<\delta$ and aim to prove that $p_1=p_2$.
Indeed, note that $d(v_2 a_t,v)<2\delta$ from the triangle inequality. Since $v_1,v_2 a_t\in U_{2\delta}$ and $\Gamma_0 v_1 =\Gamma_0 v_2 a_t$, it follows from injectivity of the projection $\pi_{\Gamma_0}$ on this set, that $v_1=v_2 a_t$. In particular $p_1=p_2$.

It follows that for a given $p_0\in\Per_{\Gamma_0}$ there can be at most $\ll \frac{T}{\delta}$ distinct elements $p\in\Per_{\Gamma}(T)$ intersecting $W$ and satisfying $\pi_{\Gamma_0}(p)=p_0$. \end{proof}

\begin{proof}[Proof of Theorem \ref{amenable covers}]
Let $\phi(T)$ be as in the statement of Theorem \ref{amenable covers}, i.e.\ the set of periodic $a_\bullet$-orbits in $\Gamma_0\backslash G/M$ of length at most $T$, which remain periodic and of the same length in $\Gamma\backslash G/M$. Let $\phi^{\prime}(T)$ be the set of periodic $a_\bullet$-orbits in $\Gamma_0\backslash G/M$, which in $\Gamma\backslash G/M$ are periodic and of length at most $T$ (not necessarily the same as their length in $\Gamma_0\backslash G/M$). Let $W$ be as in Lemma \ref{multiplicity in the cover}.

Clearly, $$\{\pi_{\Gamma_0}(p):\: p\in \Per_{\Gamma}(T),\:p\cap W\not=\emptyset\}\subset\phi^{\prime}(T).$$
Using Lemma \ref{multiplicity in the cover} it follows that
$$|\phi^{\prime}(T)|\gg\frac{1}{T}|\{p\in \Per_{\Gamma}(T):\:p\cap W\not=\emptyset\}|,$$
and so, by Theorem \ref{PPS1},
\begin{equation}\label{liminf_of_phi_prime}
\liminf_{T\to\infty}\frac{1}{T}\log|\phi^{\prime}(T)|\geq\delta(\Gamma).
\end{equation}

Note that $\phi^{\prime}(T)\smallsetminus\phi(T)$ consists of elements $p_0\in\Per_{\Gamma_0}$ that remain periodic in $\Gamma\backslash G/M$, but of different length than their original length. This situation could happen only if $p_0$ is being unfolded in $\Gamma$ an integer amount of times, and in particular its length in $\Gamma_0\backslash G/M$ must be at most $\frac{T}{2}$.

Since $\phi(T)\subset\phi^{\prime}(T)$, we get 
$$|\phi(T)|\geq |\phi^{\prime}(T)|-|\Per_{\Gamma_0}(\frac{T}{2})|.$$
Combining Theorem \ref{number_of_periodic_orbits}, Theorem \ref{PPS2}, and Equation \eqref{liminf_of_phi_prime}, we obtain that
$$\liminf_{T\to\infty}\frac{1}{T}\log|\phi(T)|\geq\delta(\Gamma)=\delta(\Gamma_0).$$
The result now follows directly from Theorem \ref{criterion for closed geodesics}.\end{proof}
\begin{remark}\label{remark on covers}
It is clear from the proof of Theorem \ref{amenable covers} that the natural probability measures $\nu_{T}^{\prime}$ averaging on $\phi^{\prime}(T)$, rather than on $\phi(T)$, equidistribute as well, since $\phi(T)\subset\phi^{\prime}(T)$ and our argument only required the set on which we average to be large enough.
\end{remark}

\begin{remark}\label{remark_about_liminf_phi_prime}
In the proof of Theorem \ref{amenable covers}, namely in Equation \eqref{liminf_of_phi_prime}, we proved that
$$\liminf_{T\to\infty}\frac{1}{T}\log|\phi^{\prime}(T)|\geq\delta(\Gamma).$$ The proof of this fact did not use the assumption that $\Gamma\backslash \Gamma_0$ is amenable.
Moreover, it can be shown that in any regular cover, amenable or not, this $\liminf$ is in fact a $\lim$ and its value is precisely $\delta(\Gamma)$. We will not require this statement in this paper.
\end{remark}

\begin{proof}[Proof of Corollary \ref{equidistribution in the covering}]
Define $\nu^{\prime}_T$ to be the natural invariant probability measure averaging on $\phi^{\prime}(T)$, and $\mu_{T}$ to be the one averaging on $\Per_{\Gamma}(T)$. In the latter case, it is an average over an infinite set, so we choose to normalize the infinite sum by $N_{T}<\infty$, i.e.\ $\mu_{T}=\frac{1}{N_T}\sum_{l\in\Per_{\Gamma}(T)}\mu_{l}$.
    
Let $f$ be as in the statement of Corollary \ref{equidistribution in the covering}. Define $\tilde{f}:\Gamma_0\backslash G/M\to\mathbb{R}$ by $$\tilde{f}(\Gamma_0 v)=\sum_{\tau\in\Gamma\backslash\Gamma_0}f(\tau \Gamma v).$$ Clearly, the sum absolutely converges due to the assumption on $f$, and depends only on $\Gamma_0 v$ (not on $v$ itself). Then $\tilde{f}$ is a continuous function, and by assumption it is bounded as well.

Note that $\phi^{\prime}(T)$ is the projection from $\Gamma\backslash G/M$ to $\Gamma_0\backslash G/M$ of any set of representatives for the $\Gamma\backslash\Gamma_0$-equivalence classes of $\Per_{\Gamma}(T)$.
Therefore, by construction, $$\int fd\mu_{T}=\int\tilde{f}d\nu^{\prime}_{T}.$$
    
By Theorem \ref{amenable covers} (or to be precise, by Remark~\ref{remark on covers}), $\nu^{\prime}_T$ converges to $m_{\BM}$, so $$\lim_{T\to\infty}\int\tilde{f}d\nu^{\prime}_{T}=\int\tilde{f}dm_{\BM},$$
which gives the desired equality.\end{proof}
 
\section{Proof of Theorem \ref{not all of the mass escapes in covers}}
\begin{proof}[Proof of Theorem \ref{not all of the mass escapes in covers}]
The key idea is to show that most closed geodesics spend at least some bounded (from below) fraction of their time in the compact part of the orbifold.

Let $\beta,\epsilon>0$ be small enough so that
\begin{equation}\label{choice_of_beta}
g(\beta,\epsilon)\coloneqq\delta(\Gamma_0)-\frac{2\delta(\Gamma_0)-r_{\max}(\Gamma_0)}{2}(1-\beta)+\frac{2\log C}{|\log\epsilon|}+\frac{3\log|\log\epsilon|}{|\log\epsilon|}<\delta(\Gamma).
\end{equation}
This is possible since
$$\lim_{(\beta,\epsilon)\to(0^+,0^+)} g(\beta,\epsilon)=\frac{r_{\max}(\Gamma_0)}{2}<\delta(\Gamma).$$
We may choose $\epsilon>0$ to be small enough with respect to Proposition \ref{counting closed geodesics} as well.

Let $\pi_{M}$ stand for the projection from $\Gamma_0\backslash G$ to $\Gamma_0\backslash G/M$.
By Proposition \ref{counting closed geodesics}, the number of periodic $a_\bullet$-orbits on $\Gamma_0\backslash G/M$ with lengths at most $T$, which spend at most a fraction $\beta$ of their time in $\pi_{M}(\Omega_{\mathrm{nc}}^{\epsilon})$, and so spend at least a fraction $(1-\beta)$ of their time in $\pi_{M}(\Omega_{\mathrm{c}}^{\epsilon})$, is at most
$$h_T\ll_{\eta}\lceil T\rceil|\log\epsilon|^{3}e^{(\delta(\Gamma_0)-\frac{2\delta(\Gamma_0)-r_{\max}(\Gamma_0)}{2}(1-\beta)+\frac{2\log C}{|\log\epsilon|}+\frac{3\log|\log\epsilon|}{|\log\epsilon|})\lceil T\rceil}$$
which is a relatively small amount.

Let $f$ be as in the statement of Theorem \ref{not all of the mass escapes in covers}.
Let $\tilde f$, $\mu_T$ and $\nu^{\prime}_T$ be as in the proof of Corollary \ref{equidistribution in the covering}. We want to prove that $$\liminf_{T\to\infty}\int f d\mu_T>0,$$ i.e.\ $$\liminf_{T\to\infty}\int \tilde{f} d\nu^{\prime}_T>0.$$
Recall that $\nu^{\prime}_T$ averages on the set $\phi^{\prime}(T)$ which is of size $N_T$, which by Remark~\ref{remark_about_liminf_phi_prime} satisfies $$\lim_{T\to\infty}\frac{1}{T}\log N_T\geq\delta(\Gamma).$$
Then
$$\nu^{\prime}_{T}(\pi_{M}(\Omega_{\mathrm{nc}}^{\epsilon}))\geq \frac{N_T-h_T}{N_T}\beta$$
and so, by the choice of $\beta$ and $\epsilon$ in Equation \eqref{choice_of_beta},
$$\liminf_{T\to\infty}\nu^{\prime}_{T}(\pi_{M}(\Omega_{\mathrm{nc}}^\epsilon))\geq \beta.$$
Let $m=\min_{x\in \overline{\pi_{M}(\Omega_{\mathrm{nc}}^{\epsilon}})}\tilde{f}(x)>0$. We obtain $$\liminf_{T\to\infty}\int \tilde{f} d\nu^{\prime}_T\geq \liminf_{T\to\infty}(m\cdot\nu^{\prime}_T(\pi_{M}(\Omega_{\mathrm{nc}}^{\epsilon})))\geq m\beta>0$$
as desired.\end{proof}

\bibliographystyle{plain}
\bibliography{Excursions_v2}

\end{document}